\documentclass[12pt, a4paper]{article}
\usepackage{amsmath}
\usepackage{amsmath,amssymb,amsthm,amsxtra}
\usepackage{geometry}
\usepackage{relsize}
\usepackage{amsmath}
\usepackage{setspace}
\usepackage{amssymb}
\usepackage{enumerate}
\usepackage{float}
\usepackage{cite}
\usepackage{hyperref}
\usepackage{graphicx}
\usepackage{epstopdf}
\usepackage{tikz}
\usetikzlibrary{matrix}

\newtheorem{thm}{Theorem}[section]
\newtheorem{cor}{Corollary}[section]
\newtheorem{lemma}[thm]{Lemma}
\newtheorem{proposition}[thm]{Proposition}
\newtheorem{defn}[thm]{Definition}
\newtheorem{definition}[thm]{Definition}
\newtheorem{deflem}[thm]{Definition/Lemma}

\newcommand{\RR}{\mathbb{R}}
\newcommand{\BB}{\mathbb{B}}
\newcommand{\CC}{\mathbb{C}}
\newcommand{\GG}{\mathbb{G}}
\newcommand{\ZZ}{\mathbb{Z}}
\newcommand{\QQ}{\mathbb{Q}}
\newcommand{\NN}{\mathbb{N}}
\newcommand{\WW}{\mathbb{W}}

\newcommand{\PP}{\mathbb{P}}
\newcommand{\FF}{\mathbb{F}}
\newcommand{\HH}{\mathbb{H}}

\newcommand{\la}{\langle}
\newcommand{\ra}{\rangle}
\newcommand{\bsm}{\left ( \begin{smallmatrix} }
\newcommand{\esm}{\end{smallmatrix} \right )}
\newcommand{\bma}{\begin{pmatrix}}
\newcommand{\ema}{\end{pmatrix}}

\newcommand{\dc}{\mathcal{D}}
\newcommand{\ec}{\mathcal{E}}
\newcommand{\fc}{\mathcal{F}}

\newcommand{\hc}{\mathcal{H}}
\newcommand{\ic}{\mathcal{I}}

\newcommand{\lc}{\mathcal{L}}

\newcommand{\mc}{\mathcal{M}}

\newcommand{\rc}{\mathcal{R}}

\newcommand{\vc}{\mathcal{V}}

\newcommand{\wc}{\mathcal{W}}
\newcommand{\opn}{\operatorname}
\newcommand{\op}{\oplus}

\newcommand{\Hf}{\mathfrak{H}}

\newcommand{\Gcd}{\opn{gcd}}
\newcommand{\Lcm}{\opn{lcm}}
\newcommand{\beq}{\begin{equation*}}
\newcommand{\eeq}{\end{equation*}}
\newcommand{\bpm}{\begin{pmatrix}}
\newcommand{\epm}{\end{pmatrix}}
\newcommand{\bi}{\begin{itemize}}
\newcommand{\ei}{\end{itemize}}
\title{On the Kodaira dimension of the moduli of deformation generalised Kummer varieties}
\author{Matthew Dawes}
\date{}
\begin{document}
\maketitle
\begin{abstract}
  We prove general type results for orthogonal modular varieties associated with moduli spaces of compact hyperk\"ahler manifolds of deformation generalised Kummer type (\emph{`deformation generalised Kummer varieties'}).
  More precisely, we consider moduli spaces of deformation generalised Kummer fourfolds with split-polarisation of degree $2d$.
  Our main result is that when $d$ is prime or $2d$ is square-free then the associated modular varieties are of general type when $d$ exceeds certain bounds.
  As a corollary, we conclude that the corresponding moduli spaces are also of general type.
\end{abstract}
\section{Introduction}\label{introsec}
\subsection{Overview}
Let $\opn{O}^+(L)$ be the kernel of the real spinor norm on $\opn{O}(L)$ where $L$ is an even integral lattice of signature $(2,n)$.
If $\dc_L$ is the component of the Hermitian symmetric space 
\begin{equation*}\Omega_L = \{ [x] \in \PP(L \otimes \CC) \mid (x,x)=0, (x, \overline{x})>0 \}
\end{equation*}
preserved by $\opn{O}^+(L)$ and $\Gamma \subset \opn{O}^+(L)$ is a subgroup of finite index, then the quotient 
$\fc(\Gamma) = \dc_L/\Gamma$ is known as an \emph{orthogonal modular variety} (or an \emph{orthogonal modular $n$-fold}).  
By the results of Baily and Borel, orthogonal modular varieties are irreducible normal quasi-projective varieties \cite{bailyborel}.
A classical problem in algebraic geometry is to determine their Kodaira dimension.

One expects most orthogonal modular varieties to be of general type.
This is typically proved by constructing pluricanonical forms on a smooth projective model of $\fc(\Gamma)$ using modular (i.e. automorphic) forms for $\Gamma$ and a technique known as the \emph{`low-weight cusp form trick'}.
The exact details of this approach depend on $n$.
A general method is known for $n \geq 9$ which was initially developed for $n=19$ \cite{Kondo, GHSK3} in connection with the moduli of K3 surfaces, finally culminating in the results of Ma \cite{Ma} for $n=17$ and $n>20$.
However, a naive application of these methods fails in small dimension due to problems related to singularities and the existence of special modular forms. 
The purpose of this paper is to introduce techniques for dealing with these problems when $n=4$, in the case of orthogonal modular varieties associated with moduli spaces of compact hyperk\"ahler manifolds of deformation generalised Kummer type (\emph{`deformation generalised Kummer varieties'}) with split polarisation.
Our main result (Theorem \ref{gtthm}) states that when the degree of polarisation exceeds a certain bound, then the corresponding modular varieties are of general type, subject to arithmetic constraints on the degree.
As a corollary, this implies the associated moduli spaces are also of general type.
To the best of our knowledge, these are the first general type results for orthogonal modular fourfolds and moduli spaces of deformation generalised Kummer varieties.
\subsection{Structure of the paper}
In \S\ref{modulisec}, we introduce moduli spaces of deformation generalised Kummer varieties and the modular group $\Gamma$.
In \S\ref{toroidalsec}, we introduce notation and definitions related to toroidal compactifications, modular forms and the low-weight cusp form trick.
In \S\ref{SingsSec}, we study singularities in orthogonal modular fourfolds and classify torsion in $\Gamma$ by considering $p$-elementary lattices.
In \S\ref{BranchSec}, we study the branch divisor of $\fc(\Gamma)$.
In \S\ref{ExtensionSec}, we obtain sufficient conditions for modular forms to define pluricanonical forms on $\fc(\Gamma)$.
In \S\ref{lwcfsec}, we prove the existence of low-weight cusp forms for $\Gamma$ by studying a dimension formula for vector-valued modular forms and applying Borcherds-Gritsenko lifting.
In \S\ref{HMSec} and \S\ref{hmbounds}, we calculate and bound Hirzebruch-Mumford volumes, which underpin the obstruction calculations of \S\ref{ObsSec}.
We conclude with general type results in \S\ref{GTsec}.
 \section{Moduli of Deformation Generalised Kummer varieties}\label{modulisec}
\subsection{Elementary results on lattices and orthogonal groups}
We start by introducing elementary results on lattices and orthogonal groups.
Standard references for lattices are \cite{SPLAG} and \cite{Nikulin}.
Unless otherwise stated, all lattices are assumed to be even and integral.

Let $L$ be a lattice.
The \emph{signature} of $L$ is defined as the pair $(m_+, m_-)$ where $m_+$ and $m_-$ are the number of positive and negative terms in a diagonalisation of the Gram matrix of a basis of $L$, respectively.
We will use $L_1 \op L_2$ to denote an orthogonal direct sum of lattices or finite quadratic forms $L_1$ and $L_2$;
$mL$ to denote the orthogonal direct sum of $m$ copies of $L$;
and $L(n)$ to denote the lattice obtained by multiplying the quadratic form of $L$ by $n$.
We will use $U$ to denote the hyperbolic plane and define a \emph{standard basis} of $U$ as a basis for $U$ with Gram matrix 
\[
\bpm
0 & 1 \\
1 & 0
\epm.
\]
We will use $\la x \ra$ to denote the rank 1 lattice generated by a vector $v$ of length $v^2:=(v,v)=x$ and we assume that all root lattices are negative definite.

Let $L^{\vee}:=\opn{Hom}(L, \ZZ)$ denote the \emph{dual lattice} of $L$.
For $u \in L$, the \emph{divisor} $\opn{div}(u)$ of $u$ in $L$ is defined as the positive generator of the ideal $(u, L)$ and we let  $u^*:=u/\opn{div}(u)$.
The \emph{discriminant group} $D(L)$ of $L$, defined by $D(L) := L^{\vee} / L$, inherits a  $\QQ/2 \ZZ$-valued finite quadratic form $q_L$ from $L$ (the \emph{discriminant form} of $L$) and the tuple $\opn{gen}(L) := (m_+, m_-, q_L)$  encodes the genus of $L$, which we also denote by $\opn{gen}(L)$ \cite{Nikulin}.

If $\opn{O}(L)$ is the orthogonal group of $L$, we use  $\widetilde{\opn{O}}(L)$ to denote the \emph{stable orthogonal group}
\[
  \widetilde{\opn{O}}(L) := \{ g \in \opn{O}(L) \mid \text{$gv = v + L$ $\forall v \in L^{\vee}$} \} 
  \]
  and, for $G \subset \opn{O}(L)$, we let $\widetilde{G} := G \cap \widetilde{\opn{O}}(L)$.
  We will often use the property that if $S \subset L$ is a sublattice of $L$ then $\widetilde{\opn{O}}(S) \subset \widetilde{\opn{O}}(L)$ \cite[Lemma 7.1]{Handbook}.  
For $g \in \opn{O}(L \otimes \RR)$ factored as a product of reflections $g = \sigma_{v_1} \ldots \sigma_{v_m}$ \cite{cassels}, where
\begin{equation}\label{refdef}
  \sigma_v : x \mapsto x - \frac{2(x,v)v}{(v,v)} \in \opn{O}(L \otimes \RR)
\end{equation}
for $v \in L \otimes \RR$,
the \emph{real spinor norm} \cite{kneser} 
\[
\opn{sn}_{\RR}: \opn{O}(L \otimes \RR) \rightarrow \RR^*/(\RR^*)^2
\]
is defined by 
\[
  \opn{sn}_{\RR}(g) =
  \left ( \frac{-(v_1, v_1)}{2} \right )
  \ldots
  \left ( \frac{-(v_m, v_m)}{2} \right )
  \in \RR / (\RR^*)^2.
\]
We will use a superscript $+$ to denote the intersection of a subgroup of $\opn{O}(L \otimes \RR)$ with the kernel of $\opn{sn}_{\RR}$ (e.g. $\opn{O}^+(L)$, $\widetilde{\opn{O}}^+(L)$).

A sublattice $S \subset L$ is said to be \emph{primitive} if $L/S$ is torsion-free.
Two primitive lattice embeddings  $S_1 \hookrightarrow L_1$ and $S_2 \hookrightarrow L_2$ are said to be isomorphic if there exists $g \in \opn{Iso}(L_1, L_2)$ restricting to an isomorphism between $S_1$ and $S_2$ \cite{Nikulin}; and if $L_1 = L_2$ and $\Gamma \subset \opn{O}(L_1)$, we say that $S_1$ and $S_2$ are \emph{equivalent under $\Gamma$} if $g$  can be taken in $\Gamma$.
We will often determine the $\widetilde{\opn{SO}}^+(L)$ orbits of primitive vectors in $L$ using the \emph{Eichler criterion} \cite[\S10]{Eichler} \cite[Proposition 3.3]{abelianisation}, which states that if $u, v \in L$ are primitive and $L \cong 2U \op L_0$ for a lattice $L_0$, then $u$ and $v$ are equivalent under $\widetilde{\opn{SO}}^+(L)$ if and only if $u^2=v^2$ and $u^* \equiv v^* \bmod{L}$. 
\subsection{Deformation generalised Kummer varieties}
Our overview follows \cite{Handbook}. 
If $A$ is an abelian surface and $A^{[n+1]}$ is the Hilbert scheme parametrising $(n+1)$-points on $A$, there is a natural morphism
\[
  p:A^{[n+1]} \rightarrow A
\]
defined by addition on $A$.
Deformations of the fibre $X:=p^{-1}(0)$ are known as \emph{deformation generalised Kummer varieties} and are examples of compact hyperk\"ahler (irreducible holomorphic symplectic) manifolds \cite{beauvillesexamples,huybrechtsbook} (i.e. $X$ is a simply connected compact K\"ahler manifold and $H^0(X, \Omega_X^2)$ is generated by an everywhere non-degenerate holomorphic 2-form).
The second integral cohomology of a compact hyperk\"ahler manifold can be endowed with the \emph{Beauville-Bogomolov lattice} \cite{beauvillesexamples, fujiki}.
For deformation generalised Kummer varieties, the Beauville-Bogomolov lattice is given by \cite{rapagnetta2}
\[
H^2(X, \ZZ) = 3 U \op \la -2(n+1) \ra
\]
and will be denoted by $M$ when viewed as an abstract lattice.

To define moduli, we first fix some discrete invariants (further details can be found in \cite{Handbook}).
Let $X$ be a deformation generalised Kummer variety of dimension $2n$.
A choice of ample line bundle $\lc \in \opn{Pic}(X)$ defines a \emph{polarisation} for $X$ and, by taking the first Chern class, we obtain a vector $h \in M$ where $h:=c_1( \lc) \in H^2(X, \ZZ) \cong M$, and a lattice $L := h^{\perp} \subset M$ of signature $(2,4)$.
We assume all polarisations are primitive (i.e. $h$ is primitive in $M$).
The \emph{degree} $2d$ of $\lc$ is defined as the length $2d:=h^2$,
the \emph{polarisation type} $[h]$ of $\lc$ is defined as the $\opn{O}(M)$-orbit of $h$ and 
 the \emph{numerical type} of $(X, \lc)$ is given by the tuple $N=(2n, L, [h])$.
We will only consider polarisations of \emph{split type}, which are those satisfying $\opn{div}(h)=1$.
In such a case,  the lattice $L$ is given by \cite[Corollary 4.0.3]{DawesThesis} 
\begin{equation}\label{ldef}
  L = 2U \op \la - 2d \ra \op \la 2(n+1) \ra,
\end{equation}
which perhaps suggests the etymology of the term.
By the Eichler criterion, the numerical type of a split polarisation is uniquely determined by the dimension $2n$ and the degree $2d$.

By the results of Viehweg \cite{Viehweg}, Matsusaka's big theorem \cite{Matsusaka} and a theorem of Koll\'ar and Matsusaka \cite{KollarMatsusaka}, there exists a GIT quotient $\mc_N$ parametrising compact hyperk\"ahler manifolds of fixed numerical type $N$.
By the Hodge-theoretic Torelli theorem \cite{HuybrechtsBourbaki, MarkmanTorelli, VerbitskyGlobal}, there is a holomorphic period map 
\[
\varphi: \mc_N \rightarrow \dc_L / \opn{O}^+(L, h)
\]
where 
\[
\opn{O}^+(L, h) := \{ g \in \opn{O}^+(L) \mid gh=h \}.
\]
As $\mc_N$ and $\dc_L / \opn{O}^+(M,h)$ are quasi-projective, the map $\varphi$ is a morphism of quasi-projective varieties \cite{borelmetric}.
By replacing $\opn{O}^+(M,h)$ with 
\[
\opn{Mon}^2(M, h):= \opn{Mon}^2(M) \cap \opn{O}^+(M, h),
\]
where $\opn{Mon}^2(M)$ is the group of Markman's monodromy operators  \cite{monok3ii, monok3, intk3}, we obtain the stronger result that $\varphi$ lifts to an open immersion $\widetilde{\varphi}: \mc_N \rightarrow \dc_L/\opn{Mon}^2(M, h)$  \cite[Theorem 3.10]{Handbook}.
\subsection{The modular group $\Gamma$}
From now on, as $\opn{O}^+(M, h)$ fixes $h \in M$, we regard $\opn{Mon}^2(M, h)$ as a subgroup of $\opn{O}(L)$ and use $\Gamma$ to denote the image of $\opn{Mon}^2(M, h)$ in $\opn{O}^+(L)$.  
\begin{proposition}\label{modgp}
  Suppose $h \in M$ is defined by a polarisation of split type and let $2d:=h^2>1$. 
Then,
  \begin{equation}\label{GammaDef}
    \Gamma = \{g \in \opn{O}^+(L) \mid \text{$gv^* = v^* + L$, $gw^* =\opn{det}(g) w^* + L$} \}
  \end{equation}
  where, if $v$ and $w$ are the respective generators of the $\la-2d \ra$ and $\la -2(n+1) \ra$ factors of $L$, then $v^*, w^* \in L^{\vee}$ are given by $v^* = \frac{-1}{2d} v$ and $w^* = \frac{-1}{2(n+1)} w$.
Furthermore, if $g \in \Gamma$ then $\chi_g(1) \equiv 0 \bmod{ 2d}$ and $\chi_g(\opn{det} g) \equiv 0 \bmod{ 2(n+1)}$, where $\chi_g$ is the characteristic polynomial of $g$.
\end{proposition}
\begin{proof}
Let $\{ e_1, f_1, e_2, f_2, e_3, f_3, w \}$ be a basis for $M$, where $\{e_i, f_i\}$ is a standard basis for the $i$-th copy of $U$ in $M$ and $w$ generates the $\la -2(n+1) \ra$ factor of $M$.
As $\opn{div}(h)=1$ then, by the Eichler criterion, we will assume that $h$ is given by $e_3 + d f_3 \in M$ and take $\{ e_1, f_1, e_2, f_2, v, w \}$
as a basis for the sublattice $L \subset M$, where $v=e_3 - df_3$.
If $\wc$ is the group
\[
\wc = \{g \in \opn{O}^+(M) \mid g_{\vert D(M)} = \pm \opn{id} \}
\]
and $\chi : \wc \rightarrow \{ \pm 1 \}$ is the character defined by the action of $\wc$ on $D(M)$ then, by \cite{Mongardi},
\begin{align}\label{mongp}
  \opn{Mon}^2(M)
  & =\opn{Ker}(\opn{det} \circ \chi) \\
  & = \{ g \in \opn{O}^+(M) \mid g(w^*) \equiv \opn{det}(g) w^* \bmod{ M} \}. \nonumber
\end{align}

We now characterise the inclusion $\opn{O}^+(L) \subset \opn{O}^+(M, h)$.
Following \cite[\S5]{Nikulin}, there is a series of finite abelian groups
\begin{equation}\label{finabseries}
  M/(L \op \la h \ra) \subset (L^{\vee}/ L) \op (\la h \ra^{\vee}/\la h \ra) = D(L) \op D(\la h \ra)
\end{equation}
corresponding to the series of lattice inclusions
\[
  L \op \la h \ra \subset M \subset M^{\vee} \subset L^{\vee} \op \la h \ra^{\vee}.
\]
If $H=M / (L \op \la h \ra )$, we let $P_L$ and $P_h$ denote the projections $P_L:H \rightarrow D(L)$ and $P_h: H \rightarrow D(\la h \ra)$ in \eqref{finabseries};
we define $H_h:=P_h(H)$, $H_{L}:=P_L(H)$ and the map $\gamma^M_{h,L}$ by  
\[
\gamma^M_{L, h}:=P_h \circ P_L^{-1}: H_{L} \rightarrow H_h.
\]
By \cite[Corollary 1.5.2]{Nikulin}, $g \in \opn{O}^+(L)$ extends to an element of $\opn{O}^+(M, h)$ if and only if $\gamma^M_{L, h} = \gamma^M_{L, h} \circ \overline{g}$ where $\overline{g}$ denotes the image of $g$ under the natural homomorphism $\opn{O}(L) \rightarrow \opn{O}(D(L))$.
As $e_3 = \frac{1}{2}(h+v)$ and $f_3 = \frac{1}{2d}(h-v)$ then $H$ is generated by the classes of $e_3$ and $f_3$.
We calculate  
$P_L(e_3)$ $\equiv \frac{1}{2} v$ $\equiv -d v^* \bmod{ L}$ and
$P_L(f_3)$ $\equiv \frac{-1}{2d}v$ $\equiv v^* \bmod{ L}$, implying $H_{M,L}=\la v^* \ra$.
Therefore, 
\begin{align*}
\gamma^M_{L, h}(v^*) = P_h(P_L^{-1}(v^*)) = P_h(f_3) & = P_h(\frac{1}{2d}(h-v)) \\
& \equiv \frac{h}{2d}   \equiv  h^* \bmod { \la h \ra}, 
\end{align*}
implying $g \in \opn{O}^+(L)$ extends to $\opn{O}^+(M, h)$ if and only if
\begin{equation}\label{gext}
h^* = \gamma^M_{L,h} (\overline{g}(v^*)).
\end{equation}
By \eqref{mongp}, if $g \in \opn{O}^+(L)$ extends to $\opn{Mon}^2(M, h)$ then $g(w^*) = \opn{det}(g) w^* \bmod{ M}$ and as $(v^*, w^*) = 0$, we have 
\begin{equation*}
  (v^*, w^*) = (gv^*, gw^*) = (gv^*, \opn{det}(g)w^* + M ) = (gv^*, w^*) = 0 \bmod{ \ZZ}.
\end{equation*}
Therefore, $gv^* = \alpha v^*$ for some $\alpha \in \ZZ$ taken modulo $2d$, and so $g \in \opn{O}^+(L)$ extends to $\opn{Mon}^2(M, h)$ if and only if $h^* \equiv \alpha h^* \bmod{ \la h \ra}$.
Therefore, by \eqref{gext}, $\alpha \equiv 1 \bmod{ 2d}$ and  the first part of the claim follows.

For the second part of the claim, as in \cite{Brasch}, we note that the matrix of $g$ is of the form 
\[
\begin{pmatrix}
  * & * & * & * & 2d* & 2(n+1)* \\
  * & * & * & * & 2d*  & 2(n+1)* \\
  * & * & * & * & 2d* & 2(n+1)* \\
  * & * & * & * & 2d* & 2(n+1)* \\
  * & * & * & * & 1 + 2d* & 2(n+1)* \\
  * & * & * & * & 2d* & \opn{det}g + 2(n+1)*\\
\end{pmatrix}
\] 
where $*$ denotes an element of $\ZZ$.
\end{proof}
 \section{Toroidal compactifications}\label{toroidalsec}
In order to prove Theorem \ref{lwcft} we need an explicit description of a toroidal compactification $\overline{\fc}(\Gamma)$ of $\fc(\Gamma)$.
We let $L$ denote a lattice of signature $(2,n)$ where $n \geq 2$ and we suppose $\Gamma \subset \opn{O}^+(L)$ is a subgroup of finite index.
\subsection{Construction of toroidal compactifications}
Toroidal compactifications are defined fully in the monograph \cite{AMRT}.
Our overview follows \cite{Kondo}, in which more details can be found.
Let $\overline{\dc}_L$ denote the closure of $\dc_L$ in the \emph{compact dual} $\check{\dc}_L := \{[w] \in \PP(L \otimes \CC) \mid (w,w)=0 \}$ of $\dc_L$.
A \emph{boundary component} (or a \emph{cusp}) $F$ of $\dc_L$ is defined as a maximal connected complex analytic subset of $\overline{\dc}_L \backslash \dc_L$.
A boundary component $F$ is said to be \emph{rational} if the stabiliser $N(F):=\{g \in G_{\RR} \mid g(F) \subset F \}$ is defined over $\QQ$, where $G_{\RR}$ is a connected component of $\opn{O}(L \otimes \RR)$.
We let $W(F)$ denote the unipotent radical of $N(F)$ and let $U(F)$ denote the centre of $W(F)$.
We will use a subscript $\CC$ to denote the complexification of a subgroup of $G_{\RR}$ and a subscript $\ZZ$ to denote the intersection of a group with $\Gamma$ (e.g. $N(F)_{\CC}$ and $U(F)_{\ZZ}$).
The stabiliser $N(F)$ is a maximal parabolic subgroup of $G_{\RR}$ and, conversely, every maximal parabolic subgroup of $G_{\RR}$ is the stabiliser of a boundary component, subject to the condition that $G_{\RR}$ is simple (which is true in our setting c.f. \cite[Proposition 3.9 \S III.3]{AMRT}).
As $\Gamma \subset \opn{O}(2, n)$, the maximal parabolic subgroups of $\Gamma$ are precisely the stabilisers of totally isotropic primitive sublattices $E \subset L$, with $\opn{rank}E = 1$ and $\opn{rank}E=2$ corresponding to boundary points and  boundary curves, respectively. (For more details see \cite{AMRT, borelji, Piatetski-Shapiro, Satake1, Handbook}.)
For each rational boundary component $F$, the group $U(F)$ contains a self-dual homogeneous cone $C(F)$ \cite[\S III.4]{AMRT}.
To construct a toroidal compactification $\overline{\fc}(\Gamma)$ of $\fc(\Gamma)$, we start by taking an admissible polyhedral decomposition $\{\sigma_i\}$ of $C(F)$ for each rational boundary component $F$ \cite[\S II.4]{AMRT}.
The fan $\{ \sigma_i \}$ defines a toric variety $T(F)_{\{\sigma_i \}} \supset T(F)$ where $T(F)$ is the torus $T(F):=U(F)_{\CC} / U(F)_{\ZZ}$.
If $\dc_L(F):=U(F)_{\CC}. \dc_L (\subset \check{\dc}_L)$, there exists a holomorphic isomorphism $\dc_L(F)  \cong U(F)_{\CC} \times \CC^m \times F$ where $m=0$ if $\opn{dim} F = 0$ and $m=n-2$ if $\opn{dim} F = 1$. 
We define $(\dc_L / U(F)_{\ZZ})_{\{\sigma_i \}}$ as the interior of the closure of $\dc_L/U(F)_{\ZZ}$ in $( \dc(F)/U(F)_{\ZZ}) \times_{T(F)} T(F)_{\{\sigma_i \}}$. By \cite[Theorem 5.2 \S III]{AMRT}, there exists a compact analytic space $\overline{\fc}(\Gamma)$ and a morphism
\[
  \pi_F: (\dc_L / U(F)_{\ZZ})_{\{\sigma_i \}}
  \rightarrow
  \overline{\fc}(\Gamma).
  \]
  The space $\overline{\fc}(\Gamma)$ has (at most) finite quotient singularities and the modular variety $\fc(\Gamma)$ is Zariski open in $\overline{\fc}(\Gamma)$.
  Each point of $\overline{\fc}(\Gamma)$ is contained in the image of $\pi_F$ or the projection $\pi: \dc_L \rightarrow \fc(\Gamma)$.
\subsection{Coordinates for $\overline{\fc}(\Gamma)$}
We introduce coordinates for $\overline{\fc}(\Gamma)$ in the neighbourhood of a boundary curve, following \cite{Kondo, GHSK3, scattone}.
Let $E \subset L$ be a primitive totally isotropic sublattice  corresponding to the boundary curve $F$.
As in \cite[Lemma 2.23]{GHSK3}, there exists a $\ZZ$-basis $\{e_i'\}_{i=1}^{n+2}$ for $L$ such that 
$\{e_1'$, $e_2'\}$ and 
$\{e_i'\}_{i=1}^n$
define $\ZZ$-bases for $E$ and $E^{\perp}$, respectively.
The Gram matrix of $\{e_i'\}_{i=1}^{n+2}$ is given by 
\begin{equation}\label{Q'def}
  Q'
  = (e_i', e_j')
  = \begin{pmatrix}
    0 & 0 & A \\
    0 & B & C \\
    {}^{\tau}A & {}^{\tau}C & D
  \end{pmatrix},
\end{equation}
where 
\begin{equation}\label{Adeltae}
  A =
  \begin{pmatrix}
    \delta & 0 \\
    0 & \delta e
  \end{pmatrix}
\end{equation}
and $\delta, e \in \ZZ$. 
As in \cite[Lemma 2.24]{GHSK3}, by applying the change of basis 
\begin{equation}\label{basechangeN}
  N =
  \begin{pmatrix}
    I & 0 & R' \\
    0 & I & R \\
    0 & 0 & I
  \end{pmatrix}
\end{equation}
to $\{e_i'\}$, where $R = -B^{-1} C$ and $R'$ satisfies
 $ D  - {}^{\tau} CB^{-1}C + {}^{\tau}R'A + {}^{\tau}AR' = 0$, we obtain a $\QQ$-basis $\{e_i\}_{i=1}^{n+2}$ for $L$ with Gram matrix
\begin{equation}\label{Qdef}
  Q :=
  (e_i, e_j)
  =
  \begin{pmatrix}
    0 & 0 & A \\
    0 & B & 0 \\
    A & 0 & 0
  \end{pmatrix}
\end{equation}
and $A, B$ as before. 
With respect to $\{e_i\}$, the groups $N(F)$, $W(F)$, $U(F)$ are given by \cite[Lemma 2.25]{GHSK3}  
\begin{align}                                     
  N(F) & =
  \left
  \{
  \begin{pmatrix}
    U & V & W \\
    0 & X & Y \\
    0 & 0 & Z
  \end{pmatrix}
  \mbox{\larger[30] $\mid$ }                                                                
  \begin{matrix}
    {}^t UAZ=A, {}^tXBX=B, {}^tXBY+{}^tVAZ=0,\\
    {}^t YBY + {}^t ZAW + {}^t WAZ=0,\
    \opn{det} U>0\end{matrix}\right\}, \label{nf} \\
  W(F) & =
  \left \{
  \begin{pmatrix}
    I & V & W \\
    0 & I & Y \\
    0 & 0 & I
  \end{pmatrix}
  \mbox{\larger[30] $\mid$ }
  BY+{}^tVA=0,\ {}^tYBY+AW+{}^tWA=0
  \right \}, \nonumber \\ U(F) & =
  \left\{
  \begin{pmatrix}
    I & 0 & \begin{pmatrix} 0 & ex \\ -x & 0\end{pmatrix} \\
      0 & I & 0\\
      0 & 0 & I
  \end{pmatrix}
  \mbox{\larger[30] $\mid$ }
  x \in \RR
  \right \}. \label{uf}
\end{align}
As in \cite{Kondo, GHSK3}, we define coordinates for $\dc_L(F)$.  
If $[t_1: \ldots : t_{n+2}]$ are homogeneous coordinates for $\PP(L \otimes \CC)$ on the basis $\{e_i\}$ then, 
\begin{align}\label{DLFcoord}
  \dc_L(F)
  & \cong U(F)_{\CC} \times \CC^m \times F \nonumber \\
  & = \{(z, \underline{w}, \tau) \in \CC \times \CC^m \times \HH^+ \} \\
  & = \{[z:t_2:\underline{w}:\tau:1] \in \PP(L \otimes \CC) \mid   2 \delta et_2
  =
  -2 \delta z \tau - \underline{w} B ^{\tau}\underline{w} \nonumber
\},
\end{align}
with $e$ and $\delta$ as in \eqref{Adeltae}.
By setting $u:=\opn{exp}_e(z):=e^{2 \pi i z / e}$ one obtains a coordinate for the torus $T(F) \cong \CC^*$ and the partial compactification at $F$ is obtained by allowing $u=0$ \cite[p.539]{GHSK3}.
Furthermore, for $g\in N(F)$ as in \eqref{nf} with
\begin{equation*}
  Z=
  \begin{pmatrix}
    a & b\\
    c & d
  \end{pmatrix},
\end{equation*}
the action of $g$ on $\dc_L(F)$ is given by  
\begin{equation}\label{gaction}
  \begin{cases}
  z  \mapsto 
  \frac{z}{\det Z} + (c \tau + d)^{-1} \left( \frac{c}{2 \delta \det Z}{}^t \underline{w} B \underline{w}
  +\underline{V}_1\underline{w} + W_{11} \tau + W_{12} \right) \\
  \underline{w}  \mapsto  (c\tau+d)^{-1}
  ( X \underline{w} + Y 
  \bsm
    \tau \\ 1
    \esm
    )\\
    \tau  \mapsto  (a \tau + b)/(c \tau + d),
  \end{cases}
\end{equation}
where ${\underline V}_i$ is the $i$-th row of the matrix $V$ in \eqref{nf} \cite[Proposition 2.26]{GHSK3}\cite[p.259]{Kondo}.
\subsection{The low-weight cusp form trick}
We construct pluricanonical forms $\Omega(f)$ on a smooth projective model $\widetilde{\fc}(\Gamma)$ for $\fc(\Gamma)$ from modular forms $f$.
The modular forms $f$ must lie outside obstruction spaces $\opn{EllObs}$ and  $\opn{RefObs}$, which we define below.
\begin{defn}[{\hspace{1sp}\cite[p.495]{Handbook}}]
  Let $\Gamma \subset \opn{O}^+(L)$ be a subgroup of finite index with character $\chi:\Gamma \rightarrow \CC^*$ and let  $\dc_L^{\bullet}$ be the affine cone of $\dc_L$.
  A \textbf{modular form $f$ of weight $k$ and character $\chi$} for $\Gamma$ is a holomorphic function $f:\dc_L^{\bullet} \rightarrow \CC$ such that, for all $Z \in \dc_L^{\bullet}$, both
  \begin{enumerate}
  \item $f(gZ) = \chi(g) f(Z)$ for all $g \in \Gamma$ 
  \item $f(tZ) = t^{-k} f(Z)$ for all $t \in \CC^*$
  \end{enumerate}
  are satisfied. 
  We denote the space of all such modular forms by $M_k(\Gamma, \chi)$ and the subspace of cusp forms by $S_k(\Gamma, \chi)$.
\end{defn}
For $v \in L^{\vee}$ with $v^2<0$, the \emph{rational quadratic divisor} $\dc_L(v) \subset \dc_L$ is defined by
\[
  \dc_L(v) = \{[x] \in \dc_L \mid (x, v)=0 \}.
  \]
We define the space of \emph{reflective obstructions}
\begin{equation}\label{RefObsdef}
\opn{RefObs}_{(4-a)k}(\Gamma, \chi_2) \subset M_{(4-a)k}(\Gamma, \chi_2)
\end{equation}
as the space of weight $(4-a)k$ modular forms vanishing to order $<k$ on $\dc_L(v) \subset \dc_L$ for all $\pm \sigma_v \in \Gamma$.
If $\ec$ is a set of primitive embeddings of signature $(2,3)$ lattices in $L$, the space of \emph{elliptic obstructions} 
\begin{equation}\label{EllObsdef}
\opn{EllObs}_{(4-a)k}(\Gamma, \ec, \chi_2) \subset M_{(4-a)k}(\Gamma, \chi_2)
\end{equation}
is defined as the space of weight $(4-a)k$ modular forms vanishing to order $<2k$ on all rational quadratic divisors $\dc_L(K^{\perp})$ for $(K \hookrightarrow L) \in \ec$.
In later sections, we will mostly be concerned with the growth of  $\opn{EllObs}_{(4-a)k}(\Gamma, \ec, \chi_2)$, $\opn{RefObs}_{(4-a)k}(\Gamma)$ and $M_k(\Gamma, \chi)$ as $k \rightarrow \infty$: as this is independent of the character $\chi_2$, we will typically ignore it, writing $\opn{EllObs}_{(4-a)k}(\Gamma, \ec)$, $\opn{RefObs}_{(4-a)k}(\Gamma)$ and $M_k(\Gamma)$ instead.

\sloppy
In Theorem \ref{lwcft}, we will assume that $\ec$ satisfies the conditions of Proposition \ref{intsingextprop} and \ref{bcextprop}.
We will construct suitable $\ec$ in Lemma \ref{eclem}.
\begin{thm}\label{lwcft}
  Assume $\ec$ satisfies the conditions of Proposition \ref{intsingextprop} and \ref{bcextprop}. 
  Suppose that $0 \neq f_a \in S_a(\Gamma, \chi_1)$ for $0<a<4$ and $g \in M_{(4-a)k}(\Gamma, (\opn{det}^k)\chi_1^{-k} )$
  lies outside the obstruction spaces
  $\opn{EllObs}_{(4-a)k}(\Gamma, \ec, (\opn{det}^k)\chi_1^{-k} )$
  and
  $\opn{RefObs}_{(4-a)k}(\Gamma, (\opn{det}^k)\chi_1^{-k})$.
  Let 
  \[
    \begin{array}{ccc}
      f:=f_a^k g \in M_{4k}(\Gamma, \opn{det}^k) 
      & \text{and} &
      \Omega(f):=f \omega^{\otimes k}
    \end{array}
  \] 
  where $\omega$ is a holomorphic volume form for $\dc_L$.
  Then there exists a toroidal compactification $\overline{\fc}(\Gamma)$ of $\fc(\Gamma)$ and a desingularisation $\widetilde{\fc}(\Gamma)$ of $\overline{\fc}(\Gamma)$ such that for all $f$ as above, the $\Gamma$-invariant differential form $\Omega(f)$ defines a pluricanonical form on $\widetilde{\fc}(\Gamma)$.
\end{thm}
\begin{proof}
  The argument follows \cite[Theorem 1.1]{GHSK3}. 
Take a toroidal compactification $\overline{\fc}(\Gamma)$ of $\fc(\Gamma)$ and a resolution of singularities $\widetilde{\fc}(\Gamma) \rightarrow \overline{\fc}(\Gamma)$.
  By Corollary 2.21/Corollary 2.22 of \cite{GHSK3} (and the correction of \cite{Ma}), we will assume all singularities of $\overline{\fc}(\Gamma)$ are canonical above a 0-dimensional boundary component.
  If $\omega$ is a holomorphic volume form on $\dc_L$ then, as $f \in M_{4k}(\Gamma, \opn{det}^k)$, the differential form $\Omega(f):=f \omega^{\otimes k}$ is $\Gamma$-invariant on $\dc_L$ and defines a section of the pluricanonical bundle $kK_{\fc(\Gamma)}$ away from the cusps and branch locus. 
  By \cite[Corollary 2.13]{GHSK3}, the smooth part of the branch locus of $\pi: \dc_L \rightarrow \fc(\Gamma)$ coincides with the fixed locus of elements $\pm \sigma_v \in \Gamma$ and, by the results of \cite{MaIrregularCusps}, there are no fixed divisors in the boundary of $\overline{\fc}(\Gamma)$.
  Therefore, as $f$ vanishes to order $k$ along the fixed loci of quasi-reflections, $\Omega(f)$ defines a pluricanonical form away from the singularities and cusps of $\fc(\Gamma)$.
  By \cite[Theorem 1.1 p.192]{AMRT}, as $f_a$ is a cusp form, then $f$ has zeros of order $k$ along the boundary divisor of $\overline{\fc}(\Gamma)$.
Therefore  $\Omega(f)$ defines a pluricanonical form on the smooth locus of $\overline{\fc}(\Gamma)$.
  By Proposition \ref{intsingextprop} and \ref{bcextprop}, $\Omega(f)$ defines a pluricanonical form over all non-canonical singularities of $\overline{\fc}(\Gamma)$, from which the result follows.
\end{proof}
 \section{Non-canonical singularities in $\overline{\fc}(\Gamma)$}\label{SingsSec}
In this section, 
we classify non-canonical singularities in arbitrary orthogonal modular fourfolds; 
study elliptic elements of $\Gamma$ in terms of their invariant lattices;
calculate elliptic elements in $\Gamma$ defining non-canonical singularities in $\fc(\Gamma)$;
and study non-canonical singularities in the boundary of $\overline{\fc}(\Gamma)$.
\subsection{Non-canonical singularities in orthogonal modular fourfolds}
We classify non-canonical singularities in arbitrary orthogonal modular fourfolds $\fc(\Delta)$ where $\Delta \subset \opn{O}^+(N)$ and $N$ is a lattice of signature $(2,4)$.
For $A, B \in M_n(\CC)$ we use $A \sim B$ to denote $A={}^gB=gBg^{-1}$ for some $g \in \opn{GL}(n, \CC)$.
For $m \in \NN$, we use $\frac{1}{m}(a_1, \ldots, a_n)$ to denote the diagonal matrix $\opn{diag}(\xi_m^{a_1}, \ldots, \xi_m^{a_n})$  where
$a_i \in \ZZ$  are taken modulo $m$ and
$\xi_m:=e^{2 \pi i / m}$.
Where no confusion is likely to arise, we will also use $\frac{1}{m}(a_1, \ldots, a_n)$ to denote the quotient $\CC^n / \la g \ra$.

As the singularities of $\overline{\fc}(\Gamma)$ are finite quotient singularities, they can be studied using the Reid-Tai criterion \cite{YPG, Tai}, as in \cite{Kondo, GHSK3}.
For $g \sim \frac{1}{m}(a_1, \ldots, a_n)$, the \emph{Reid-Tai sum} $\Sigma(g)$ is defined as
\[
\Sigma(g) =
\begin{cases}
  \sum \left \{ \frac{a_i}{m} \right \} & \text{if $g \neq I$} \\
  1 & \text{if $g=I$}
\end{cases}
\]
where $\{ x \} := x - [x]$ denotes the fractional part of $x \in \QQ$.
The \emph{Reid-Tai criterion} states that if $G \subset \opn{GL}(n, \CC)$ is a finite group containing no quasi-reflections (i.e. no elements with precisely $n-1$ eigenvalues equal to 1) then $\CC^n / G$ has canonical singularities \cite{YPG} if and only if $\Sigma(g) \geq 1$ for all $g \in G$ \cite{ReidCanonical3-folds, YPG,  Tai}. 
If $G$ contains quasi-reflections, there is a modified version of the Reid-Tai criterion due to Katharina Ludwig \cite[Proposition 5.24]{Handbook}.
If $g \sim \frac{1}{m}(a_1, \ldots, a_n)$ and $m=st$ where $t$ is the smallest positive integer such that $g^t$ is the identity or a quasi-reflection and $\xi_m^{k a_i}=1$ for all $1 \leq i \leq n-1$ where $\xi_m^{k a_n}$ is a primitive $s$-th root of unity, then the \emph{modified Reid-Tai sum} $\Sigma'(g)$ is defined by
\[
\Sigma'(g) =
\begin{cases}
  \left \{ \frac{s a_n}{n} \right \} + \sum_{i=1}^{n-1} \left \{ \frac{a_i}{m} \right \} & \text{if $g \neq I$} \\
  1 & \text{if $g=I$.}
\end{cases}
\]
The \emph{modified Reid-Tai criterion} states that $\CC^n / G$ has canonical singularities if $\Sigma'(g) \geq 1$ for all $g \in G$.

Around the image of $[w] \in \dc_N$, the modular variety $\fc(\Delta)$ is locally isomorphic to $T_{[w]} \dc_N / \opn{Iso}_{\Delta}([w])$ where $T_{[w]} \dc_N$ is the tangent space of $\dc_N$ at $[w]$ and $\opn{Iso}_{\Delta}([w])$ is the isotropy subgroup
\[ 
\opn{Iso}_{\Delta}([w]) := \{ g \in \Delta \mid g.[w] = [w] \}.
\]
As in \cite{GHSK3}, if $\WW:= \CC.w \subset N \otimes \CC$, there is an isomorphism
\begin{equation*}T_{[w]} \dc_N \cong \opn{Hom}(\WW, \WW^{\perp}/\WW)
\end{equation*}
and a character $\alpha:\opn{Iso}_{\Delta}([w]) \rightarrow \CC^*$ defined by the action of $\opn{Iso}_{\Delta}([w])$ on $\WW$.
We refer to $[w] \in \dc_N$ as a \emph{canonical singularity of $\fc(\Delta)$} if $T_{[w]} \dc_N/ \opn{Iso}_{\Delta}([w])$ has only canonical singularities; otherwise, we refer to $[w]$ as a  \emph{non-canonical singularity of $\fc(\Delta)$}.
\begin{lemma}\label{NCOM4}
  Suppose $g \in \Gamma$ fixes $[w] \in \dc_N$.
  Then
  \begin{equation}\label{TwSing}
    T_{[w]} \dc_N / \la g \ra
  \end{equation}
  is a non-canonical singularity if and only if the isomorphism class of \eqref{TwSing} is given in Table \ref{SingTable}. 
  The corresponding values of $\alpha(g)$ and the characteristic polynomial $\chi_g(x)$ are also given in Table \ref{SingTable}, where $\phi_r$ denotes the $r$-th cyclotomic polynomial.
\end{lemma}
\begin{proof}
  We study rational representations of the cyclic group using the Reid-Tai criterion, as in \cite{Kondo, GHSK3}.
Suppose $g \in \opn{Iso}_{\Delta}([w])$ is of order $m$,
  $g^k$ is of prime power order $p^r$
  and let  
  \[
  N \otimes \QQ \cong \bigoplus_i \vc_{p^{r_i}},
  \]
  be the decomposition of $N \otimes \QQ$ as a $g^k$-module, where $\vc_m$ is the unique faithful $\QQ$-irreducible representation of the cyclic group $C_m$ \cite{SerreLin} and $r_1 \leq \ldots \leq r_n=r$.
  As $\opn{deg} \vc_{p^{r_i}} = \phi(p^{r_i}) = p^{r_i}-1 \leq \opn{dim} N_{\QQ} = 6$ then $p^r \in \{1,2,3,4,5,7,8,9\}$ and $m \leq 5.7.8.9 = 2520$.
  (In fact, $m \in  \{1,$ $2,$ $3,$ $4,$ $5,$ $6,$ $7,$ $8,$ $9,$ $10,$ $12,$ $14,$ $15,$ $18,$ $20,$ $24,$ $30\}$.)
  We then proceed to perform an exhaustive computer search using the modified Reid-Tai criterion over all faithful 6-dimensional $\QQ$-representations of $C_m$ for $m \leq 2520$.
  (Our search also returned the false-positives
  $h= \frac{1}{6}(0,0,2,3)$
  and
  $\frac{1}{6}(0,0,3,4)$:
  if such $h$ were to exist then $h^2$ and $h^3$ both define quasi-reflections  of order $\neq 2$, contradicting Corollary 2.13 of \cite{GHSK3}.) 
  The result follows.
\end{proof}
\begin{table}\centering
  \begin{tabular}{|c|c|c||c|c|c|}
    \hline 
    \textbf{Singularity} & $\chi_g$ & $\alpha(g)$ &  \textbf{Singularity} & $\chi_g$ & $\alpha(g)$  \\ 
    \hline
    & & & & & \\[-1em]
    $\frac{1}{3}(0,0,1,1)$ & $\phi_{3}^{3} $& $\xi_{3}^{\pm 1}$ &

    $\frac{1}{3}(0,0,1,1)$ & $\phi_{6}^{3} $& $\xi_{6}^{\pm 1}$ \\
    & & & & & \\[-1em]

    $\frac{1}{6}(0,0,2,2)$ & $\phi_{3} \phi_{6}^{2} $& $\xi_{6}^{\pm 2}$ &  
    
    $\frac{1}{6}(0,0,2,2)$ & $\phi_{3} \phi_{6}^{2} $& $\xi_6^{\pm 1}$ \\
    & & & & & \\[-1em]
      
    $\frac{1}{6}(0,1,1,2)$ & $\phi_{2}^{2} \phi_{3}^{2} $& $\xi_{6}^{\pm 2}$ &  
    
    $\frac{1}{6}(1,1,1,1)$ & $\phi_{2}^{4} \phi_{3} $& $\xi_{6}^{\pm 2}$ \\
      & & & & & \\[-1em]

    $\frac{1}{6}(0,1,1,2)$ & $\phi_{1}^{2} \phi_{6}^{2} $& $\xi_{6}^{\pm 1}$ &

    $\frac{1}{12}(0,2,2,4)$ & $\phi_{4} \phi_{6}^{2} $& $\xi_{12}^{\pm 2}$ \\
      & & & & & \\[-1em]

    $\frac{1}{12}(2,2,2,2)$ & $\phi_{4}^{2} \phi_{6} $& $\xi_{12}^{\pm 2}$ &  
    
    $\frac{1}{12}(0,2,2,4)$ & $\phi_{3} \phi_{4} \phi_{6} $& $\xi_{12}^{\pm 4}$ \\
    & & & & & \\[-1em]

    $\frac{1}{12}(0,2,2,4)$ & $\phi_{3} \phi_{4} \phi_{6} $& $\xi_{12}^{\pm 2}$ &  
    
    $\frac{1}{12}(2,2,2,2)$ & $\phi_{3} \phi_{4}^{2} $& $\xi_{12}^{\pm 4}$ \\
      & & & & & \\[-1em]

    $\frac{1}{12}(0,2,2,4)$ & $\phi_{3}^{2} \phi_{4} $& $\xi_{12}^{\pm 4}$ &  
    
    $\frac{1}{24}(4,4,4,4)$ & $\phi_{6} \phi_{8} $& $\xi_{24}^{\pm 4}$ \\ 
    & & & & & \\[-1em]
    
    $\frac{1}{24}(4,4,4,4)$ & $\phi_{3} \phi_{8} $& $\xi_{24}^{\pm 8}$ & 
    
    $\frac{1}{30}(5,5,5,5)$ & $\phi_{5} \phi_{6} $& $\xi_{30}^{\pm 5}$ \\
      & & & & & \\[-1em]  
    
    $\frac{1}{30}(5,5,5,5)$ & $\phi_{3} \phi_{10} $& $\xi_{30}^{\pm 10}$       & & & \\  
    \hline
  \end{tabular}
  \caption{Non-canonical cyclic quotient singularities in orthogonal modular 4-folds.} \label{SingTable}
\end{table}
\subsection{Elliptic elements of $\Gamma$ and $p$-elementary lattices}
Suppose $L$ and $\Gamma$ are as in \eqref{ldef} and \eqref{GammaDef} and let $M$ denote the Beauville lattice.
We now classify elliptic elements of $\Gamma$ in terms of their invariant and perp-invariant lattices.
For $g \in \opn{O}(N)$ where $N$ is an arbitrary lattice, we define the \emph{invariant} and \emph{perp-invariant} lattices of $g$ by
\begin{center}
  \begin{tabular}{c c c}
    $T := \{ x \in N \mid gx=x \}$ & and &   $S:=T^{\perp} \subset N$,
  \end{tabular}
  \end{center}
respectively.
We recall that a lattice $N$ is said to be \emph{$p$-elementary} if $D(N) \cong C_p^{\op a}$ where $C_p$ is the cyclic group of prime order $p$.
A partial classification of the genera of $p$-elementary lattices is given in \cite{RudakovShafarevich1, RudakovShafarevich2} (see also \cite{SPLAG}).
${}$
\\
\\
\noindent \textbf{Assumption:} from now on, we will assume (unless otherwise stated) that the parameters $n$ and $d$ defining the lattice $L$ satisfy $n=2$ and $d>48$.
${}$
\\
\begin{proposition}\label{pelemProp}
  Suppose $g \in \opn{O}(M)$ is of prime order $p$.  
  \begin{enumerate}
  \item If $p=2$ then $D(S)$ or $D(T)$ is $2$-elementary.
  \item If $3 \leq p \leq 19$ then $D(S)$ is $p$-elementary.
  \end{enumerate}
\end{proposition}
\begin{proof}
  We follow \cite[Lemma 5.5]{SmithTheory}.
  (We also proved a version of this result in the unpublished  thesis  \cite{DawesThesis}.)
If $N$ is the $g$-module
\[
N = \frac{M}{S \op T}
\]
then, by \cite[Lemma 4.3/Remark 4.4]{SmithTheory}, there exists $a \in \ZZ$ such that $N \cong C_p^{\op a}$.
By standard results on overlattices (e.g. \cite[Lemma 2.1, Chapter I.1]{CCS}),
\[
\vert M : S \op T \vert^2 = \opn{det}(S) \opn{det}(T) \opn{det}(M)^{-1}
\]
and so
\begin{equation}\label{detTdetS}
\opn{det}(S) \opn{det}(T) = 6p^{2a},
\end{equation}
implying $\opn{det}(S) = 2^{\delta} 3^{\epsilon} p^{\alpha}$ and $\opn{det}(T) = 2^{1-\delta}3^{1-\epsilon}p^{\beta}$ where $\alpha+\beta = 2a$ and $\epsilon, \delta \in \{0, 1 \}$.

We now show that all elementary factors of $D(S)$ and $D(T)$ are isomorphic to $C_2$, $C_3$ or $C_p$. 
As in \S5 of \cite{Nikulin}, the overlattices
\[
S \op T \subset M \subset M^{\vee} \subset S^{\vee} \op T^{\vee}
\]
define an inclusion of finite abelian groups
\[
N \subset D(S) \op D(T)
\]
and natural projections $p_S:N \rightarrow D(S)$ and $p_T:N \rightarrow D(T)$.
As both $S,T \subset L$ are primitive then both $p_S$ and $p_T$ are injective \cite[p.111]{Nikulin} and so $C_p^{\op a} \subset D(S)$ and $D(T)$.
Therefore, if $p=2$ then $a \leq \alpha + \delta$ and $a \leq \beta + (1-\delta)$. 
As $\alpha + \beta = 2a$ then $a - \delta \leq \alpha \leq a + (1-\delta)$ and so $\alpha + \delta = a$ or $a+1$.
Then, as $N \subset D(S)$ and by using the classification of finite abelian groups, $D(S) = C_2^{\op a} \op C_3^{\epsilon}$ or $D(S) = C_4 \op C_2^{\op (a-1) } \op C_3^{\epsilon}$ if $\delta = 1$.
The latter case cannot occur: as $D(S) \subset D(T) \op D(M) / G$ where $G \subset D(T) \op D(M)$ then, if $\delta=1$ we have $D(T) \cong C_2^{\op a} \op C_3^{\op (1-\epsilon)}$, implying $D(T)$ and so $D(S)$ contain no 4-torsion.
Similarly, if $p=3$ then $D(S) = C_2^{\delta} \op C_3^{\op(a+\epsilon)}$.
If $p>3$ then, as $a \leq \alpha$, $a \leq \beta$ and $2a= \alpha + \beta$ we have $\alpha=\beta=a$ and so $D(S) = C_2^{\op \delta} \op C_3^{\epsilon} \op C_p^{\op a}$.

We now prove the main result.
Let $x_2, x_3 \in S^{\vee}$ represent generators for $C_2^{\op \delta}, C_3^{\op (a + \epsilon)} \subset D(S)$, respectively (if they exist) and define $\sigma:=1 + g + \ldots + g^{p-1}$.  
If $p=2$ then, by \eqref{detTdetS}, $\opn{det}(T)$ or $\opn{det}(S)$ is coprime to 3 and so $S$ or $T$ is $2$-elementary.
If $p=3$ and $\delta=0$ then, as $D(S)$ contains a single element of order 2, $g \overline{x_2} = \overline{x_2}$ and so $\sigma(\overline{x_2}) = \overline{x_2}$.
On the other hand, from the $\QQ$-representation theory of the cyclic group, $\sigma$ is identically zero on $S$, implying the contradiction $\overline{x_2}=0$.
Similarly, if $p>3$ then $g \overline{x_2} = \overline{x_2}$ and $g \overline{x_3} = \pm \overline{x_3}$.
Therefore, $\sigma(\overline{x_2})=p \overline{x_2}$ and $\sigma(\overline{x_3}) = \overline{x_3}$ or $p\overline{x_3}$.
However, as $\sigma$ vanishes on $S$ and $(2,p)=(3,p)=1$ then both $\overline{x_2}$ and $\overline{x_3}$ must be zero.
The result follows.
\end{proof}
\begin{lemma}\label{perpinvrest}
  Let $g \in \opn{O}(M, h)$ where $h^2>0$ and let $g'$ be the restriction of $g$ to the sublattice $L=h^{\perp} \subset M$.
  If $S \subset M$ is the perp-invariant lattice of $g$ and $S' \subset L$ is the perp-invariant lattice of $g'$ then $S=S'$.
\end{lemma}
\begin{proof}
  If $T'$ is the invariant lattice of $g'$ then
\[
  S'
     = \{ x \in L \mid x \perp T' \} 
     = \{x \in L \mid x \perp T \} 
     = \{ x \in M \mid \text{$x \perp h$, $x \perp T$} \}
\]
and as $h \in T$ then 
$
S '
= \{ x \in M \mid x \perp T \} 
= S.
$
\end{proof}
\begin{lemma}\label{noncanonS}
  If $[w] \in \dc_L$ is a non-canonical singularity of $\fc(\Gamma)$ then $[w]$ is fixed by an element $g \in \Gamma$ of order 3 with characteristic polynomial $\chi_g = \phi_1^2 \phi_3^2$
  and perp-invariant lattice $2U$ or $U \op A_2(-1)$.   
\end{lemma}
\begin{proof}
  By Lemma \ref{NCOM4}, $[w]$ is fixed by an element $h \in \Gamma$ such that $h^k=:g$ is of order 3 for some $k$.
  By eliminating cases using Proposition \ref{modgp},  we can assume that the characteristic polynomial $\chi_h = \phi_1^2 \phi_6^2$.
  For $l \in \ZZ$, let $T_l, S_l \subset L$ denote the invariant and perp-invariant lattices of $h^l$, respectively.
  By Lemmas \ref{pelemProp} and \ref{perpinvrest}, $D(S_2) \cong C_3^a$ for some $a \in \NN$.
  The genera of 3-elementary lattices are classified in \cite[Theorem 13, p.386]{SPLAG}.
  When $a \leq 4$, each genus of 3-elementary lattices contains at most one class, which follows from \cite[Corollary 22, p.395]{SPLAG} in the indefinite case and from \cite[Table 15.1, p.360]{SPLAG} and \cite{Nipp} in the definite case.
  Therefore, as $L$ is not 3-elementary, then any 3-elementary lattice admitting a primitive embedding in $L$ must be one of 
  $2U$,
  $U \op U(3)$,
  $2U(3)$,
  $U \op A_2(-1)$,
  $U(3) \op A_2(-1)$,
  $U$,
  $U(3)$,
  $A_2 (\pm 1)$  or
  $2A_2(-1)$.
  By considering $\chi_h$ and the action of $h$ on $L \otimes \QQ$, we have $S_2 = S_3$ and $\opn{rank} S_2 = 4$.
  By Lemma \ref{pelemProp}, $S_2$ is 3-elementary and either $S_3$ or $T_3$ is 2-elementary.
  If $S_3$ is 2-elementary, then $S_3=2U$; 
  if $S_3$ is not 2-elementary then $S_3$ is 3-elementary and so, in the notation of Proposition \ref{pelemProp}, $\delta=0$, $a=0$ and $\epsilon \in \{0, 1 \}$, implying $S_3 = 2U$ or $U \op A_3(\pm 1)$. 
\end{proof}
\subsection{Elliptic elements of $\Gamma$} 
We now classify elliptic elements of $\Gamma$ defining non-canonical singularities in $\fc(\Gamma)$.
\begin{lemma}\label{UA2_3tors}
  The group $\opn{O}(1,3)$ contains no element of order 3 with characteristic polynomial $\phi_3^2$.
\end{lemma}
\begin{proof}
  We note that if $h \in \opn{O}(1,3)$ is of order 3 then $h \in \opn{SO}^+(1,3)$. 
  Following \cite[p.55]{CarterSegalMacdonald}, we construct a double cover homomorphism $\opn{SL}(2, \CC) \rightarrow \opn{SO}^+(1,3)$ by identifying $\RR^4$ with the $2 \times 2$ Hermitian matrices $\Hf$ via 
  \[
  \RR^4 \ni (t,x,y,z) \mapsto
  \bpm
  t + x & y - iz \\
  y + iz & t - x
  \epm \in \Hf,  
  \]
  allowing $\opn{det}$ to define a signature $(1,3)$ quadratic form on $\Hf^4$ and mapping each $g \in \opn{SL}(2, \CC)$ to $T_g \in \opn{SO}^+(1,3)$, where $T_g: A \mapsto g A g^{-1}$ for $A \in \Hf$.

  If $g \in \opn{SL}(2, \CC)$ is such that $T_g$ is of order 3,
  then $g$ is of finite order with conjugate eigenvalues $\alpha$ and $\overline{\alpha}$.
  Without loss of generality, we can assume $\alpha = e^{2 \pi i /3}$ and $g = \opn{diag}(\alpha, \overline{\alpha})$, implying 
  \[
  T_g =
  \bpm
  1 & 0 & 0 & 0 \\
  0 & 1 & 0 & 0 \\
  0 & 0 & -\frac{1}{2} & -\frac{\sqrt{3}}{2} \\
  0 & 0 & \frac{\sqrt{3}}{2} & -\frac{1}{2}
  \epm
  \]
  and $\chi_{T_g} = \phi_3 \phi_1^2$.
\end{proof}
\begin{deflem}\label{SL2emb}
  There is an isomorphism 
  \[
    \Delta:  (\opn{SL}(2, \ZZ) \times \opn{SL}(2, \ZZ))/\{\pm (I, I) \} \rightarrow \opn{SO}^+(2U)
  \]
  defined on a standard basis of $2U$ by
  \begin{center}
    \begin{tabular}{c c c}
      $\Delta(A,I) \mapsto
      \begin{pmatrix}
        a & 0 & 0 & b \\
        0 & d & -c & 0 \\
        0 & -b & a & 0 \\
        c & 0 & 0 & d
      \end{pmatrix}$
      & 
      and
      &
      $  \Delta(I, A) \mapsto
      \begin{pmatrix}
        d & 0 & c & 0 \\
        0 & a & 0 & -b \\
        b & 0 & a & 0 \\
        0 & -c & 0 & d
      \end{pmatrix}$
    \end{tabular}
  \end{center}
  where 
  \begin{equation*}
    A =
    \begin{pmatrix}
      a & b \\
      c & d \\
    \end{pmatrix} \in \opn{SL}(2, \ZZ).
  \end{equation*}
\end{deflem}
\begin{proof}
  The isomorphism is well known (e.g. \cite{sterk}): we adopt the notation of \cite{Brasch} and construct the homomorphism following the approach of \cite{abelianisation}.
Identify $2U$ with the matrices $M_2(\ZZ)$ via the map
\[
  U \op U
    \ni (w,x,y,z)
    \mapsto 
    \begin{pmatrix}
      w & -y \\
      z & x \\
    \end{pmatrix}
    \in M_2(\ZZ),
\]
and let $-\opn{det}$ define a quadratic form on $M_2(\ZZ)$. 
Then, if $(A, B) \in \opn{SL}(2, \ZZ) \times \opn{SL}(2, \ZZ)$, the element $\Delta(A, B) \in \opn{O}^+(2U)$ is defined by 
\[
\Delta(A, B):
    \begin{pmatrix}
      w & -y \\
      z & x \\
    \end{pmatrix}
    \mapsto
    A
    \begin{pmatrix}
      w & -y \\
      z & x \\
    \end{pmatrix}
    B^{-1}.
    \]
\end{proof}
\begin{lemma}\label{2Utorsion}
  Let $g \in \opn{O}^+(2U)$ be of order 3 with characteristic polynomial $\chi_g=\phi_3^2$.
  Then, up to $\opn{O}^+(2U)$-conjugacy, $g$ is represented by  $\Delta(I, g_3)^{\pm 1}$ or $\Delta(g_3, I)^{\pm 1}$, where
\[
  g_3 =
    \begin{pmatrix}
      0 & 1 \\
      -1 & -1 \\
    \end{pmatrix}.
 \]
\end{lemma}
\begin{proof}
  We apply Lemma \ref{SL2emb}.
  If $g \in \opn{O}^+(2U)$ is of order 3 then $g \in \opn{SO}^+(2U)$.
  The conjugacy classes of 3-torsion in $\opn{SL}(2, \ZZ)$ are represented by $g_3^{\pm 1}$ \cite{Serre} where
\[
  g_3 =
    \begin{pmatrix}
      0 & 1 \\
      -1 & -1 \\
    \end{pmatrix}.
    \]
    The result follows by calculating the characteristic polynomials of
    $\Delta(g_3, I)^{\pm 1}$, 
    $\Delta(I, g_3)^{\pm 1}$, 
    $\Delta(g_3, g_3)^{\pm 1}$ and
    $\Delta(g_3^{-1}, g_3)^{\pm 1}$.
\end{proof}
\begin{defn}
  Every embedding
  $\iota: 2U \hookrightarrow L$ with $B := (2U)^{\perp} \subset L$
  defines an embedding
  $\Delta_{\iota}: \opn{SL}(2, \ZZ) \times \opn{SL}(2, \ZZ)/ \pm (I, I) \hookrightarrow \Gamma$
  such that
  $\Delta_{\iota}(g,h) \vert_{2U}:=\Delta(g,h)$ and
  $\Delta_{\iota}(g,h) \vert_B := I$
  for all $g,h \in \opn{SL}(2, \ZZ)$.
  The extension of $\Delta(g,h)$ to $\Delta_{\iota}$ is well defined as $\widetilde{\opn{SO}}^+(2U) = \opn{SO}^+(2U)$. 
\end{defn}
\begin{lemma}\label{DeltaLemma}
  Suppose $[w] \in \dc_L$ is fixed by $g \in \Gamma$.
  If $T_{[w]} \dc_L / \la g \ra$ is a non-canonical singularity then $g = \Delta_{\iota}(g_6, I)^{\pm 1}$ or $\Delta_{\iota}(I, g_6)^{\pm 1}$ for a primitive embedding $\iota: 2U \hookrightarrow L$, where
  \begin{equation*}
    g_6=
    \begin{pmatrix}
      0 & -1 \\
      1 & 1
    \end{pmatrix} \in \opn{SL}(2, \ZZ).
  \end{equation*}
\end{lemma}
\begin{proof}
  If $T_{[w]} \dc_L / \la g \ra$ is a non-canonical singularity then, by using Lemma \ref{NCOM4} and eliminating cases using Proposition \ref{modgp}, $\chi_g=\phi_1^2 \phi_6^2$.
  By Lemma \ref{noncanonS} and \ref{UA2_3tors}, the perp-invariant lattice of $g$ is given by $2U$.
  Therefore, by Lemma \ref{2Utorsion}, $g^2 = \Delta_{\iota}(g_3, I)^{\pm 1}$ or $\Delta_{\iota}(I, g_3)^{\pm 1}$ for a primitive embedding $\iota: 2U \hookrightarrow L$ (where primitivity follows by definition of $S$).
  Hence, by Definition/Lemma \ref{SL2emb},
  $g = \Delta_{\iota}(g_6, I)^{\pm 1}$ or $\Delta_{\iota}(I, g_6)^{\pm 1}$. 
\end{proof}
\subsection{Non-canonical singularities in the boundary of $\overline{\fc}(\Gamma)$}
We now consider the non-canonical singularities of $\overline{\fc}(\Gamma)$ above a boundary curve $F$.
\begin{lemma}\label{bcsingclass}
  Suppose $P \in \overline{\fc}(\Gamma)$ lies above a boundary curve $F$.
  Let $G \subset G(F)$ be the isotropy subgroup of $P$ and let $T_P$ be the tangent space of $(\dc_L/U(F)_{\ZZ})_{\{\sigma_i \}}$ at $P$.
  If $g \in G$ is such that $T_P  / \la g \ra$ is a non-canonical singularity, then $g$ acts as $\frac{1}{6}(0,1,1,2)^{\pm 1}$ or $\frac{1}{6}(1,1,1,2)^{\pm 1}$ and is represented by
  \begin{equation}\label{gNFdef}
    g =
    \begin{pmatrix}
      U & V & W \\
      0 & I & Y \\
      0 & 0 & Z
    \end{pmatrix}
    \in N(F)_{\ZZ}, 
  \end{equation}
  where $U$ and $Z$ are of order 6.    
\end{lemma}
\begin{proof}
  By the modified Reid-Tai criterion, the isotropy subgroup $G$ contains an element $g$ such that $\Sigma'(g) <1$.
  If $g$ is represented by 
  \[
  g  =
  \begin{pmatrix}
    U & V & W \\
    0 & X & Y \\
    0 & 0 & Z
  \end{pmatrix} \in N(F)
  \]
  on the basis \eqref{Qdef} then, in the notation of \eqref{gaction} (following \cite[(8.2), p.292]{Kondo}), $g$ acts on $T_P$ by
  \begin{equation}\label{boundaryaction}
  \begin{pmatrix}
    \opn{exp}_e(t) & 0 & 0 \\
    * & \xi^{-1} X & 0 \\
    * & * & \xi^{-2} 
  \end{pmatrix},
  \end{equation}
  where $\xi = (c \tau + d)$ is a root of unity \cite[Proposition 2.28]{GHSK3}. 
As one can verify by direct calculation, if $g \in G(F)$ is of finite order then $g^m=I$ for $m>0$ if and only if $U^m=X^m=Z^m=I$.
  Therefore, by \eqref{boundaryaction}, $\opn{exp}_e(t)$ is a $12$-th root of unity.
By using this bound to perform an exhaustive computer search, we find that if $\Sigma'(g)<1$ then $g^{\pm 1}$ acts as one of
  $\frac{1}{3}(0,0,1,1)$,
  $\frac{1}{6}(0,0,2,3)$,
  $\frac{1}{6}(0,1,1,2)$,
  $\frac{1}{6}(0,2,2,3)$ or
  $\frac{1}{6}(1,1,1,2)$.
  The case $\frac{1}{6}(0,0,2,3)$ is a false positive as the corresponding group is generated by quasi-reflections and the quotient is therefore smooth \cite{Chevalley}.
  For each of the other cases, the characteristic polynomial $\chi_g$ is given by
  $\phi_3 \phi_r^2$,
  $\phi_6 \phi_r^2$,
  $\phi_2^2 \phi_r^2$ or
  $\phi_1^2 \phi_r^2$,
  where $r = 3$ or $6$.
  By further eliminating cases using Proposition \ref{modgp}, $g^{\pm 1}$ must act as either 
  $\frac{1}{6}(0,1,1,2)$
  or
  $\frac{1}{6}(1,1,1,2)$
  where, in both cases, $\xi=e^{5 \pi i/3}$ and $X=I$. 
  Then, by \eqref{nf}, $U \sim Z$ and so both must be of order 6.
\end{proof}
Our later calculations are simplified if we work with a basis for $L \otimes \QQ$ adapted to $g \in G(F)$.
\begin{lemma}\label{bcsplitlem}
  Assuming the conditions of Lemma \ref{bcsingclass}, let $g \in N(F)_{\ZZ}$ be such that $T_P / \la g \ra$  is a non-canonical singularity. 
  Then, without loss of generality, the basis \eqref{Qdef} can be chosen so that 
  \[
  g =
  \begin{pmatrix}
    U & 0 & W' \\
    0 & I & 0 \\
    0 & 0 & Z
  \end{pmatrix}
  \]
  with $U$ and $Z$ as in \eqref{gNFdef}.
\end{lemma}
\begin{proof}
  We begin by proving some facts about $U$ and $V$ using the basis \eqref{Qdef}, on which $g$ is given by \eqref{gNFdef}.
  As $g$ is of finite order in $G(F)_{\ZZ}$ then, from the description of $U(F)_{\ZZ}$ in \eqref{uf}, the submatrix
  \[
  \begin{pmatrix}
    U & V \\
    0 & I
  \end{pmatrix}
  \]
  is of finite order.
By \eqref{basechangeN} and the definition of $N(F)$ in \eqref{nf}, $V \in M_2(\ZZ)$ 
and as $U$ is of order 6 then
  \[
  U \sim
  \begin{pmatrix}
    0 & -1 \\
    1 & 1
  \end{pmatrix}^{\pm 1},
  \]
  (following from the classification of torsion in $\opn{SL}(2, \ZZ)$ e.g. \cite{Diamond, Serre}) implying $I - U \in \opn{SL}(2, \ZZ)$.

  We now take the $\ZZ$-basis \eqref{Q'def} for $L$, on which $g$ is given by
  \[
    g =
    \begin{pmatrix}
      U & V & W' \\
      0 & I & Y' \\
      0 & 0 & Z'
    \end{pmatrix},
  \]
  with $U$ and $V$ as before (following from \eqref{basechangeN}).
  By applying the change of basis  
  \[
    N_1:=
    \begin{pmatrix}
      I & (I-U)^{-1}V & 0 \\
      0 & I & 0 \\
      0 & 0 & I
    \end{pmatrix}, 
  \]
  we obtain a $\ZZ$-basis $\{f_i \}$ for $L$ such that
  \begin{align*}
    Q:= ((f_i,f_j)) =
    \begin{pmatrix}
      0 & 0 & A \\
      0 & B & C' \\
      A & {}^{\tau}C' & D
    \end{pmatrix} 
    &&
    \text{ and }  
    &&
    g =
    \begin{pmatrix}
      U & 0 & W' \\
      0 & I & Y' \\
      0 & 0 & Z'
    \end{pmatrix}.
    \end{align*}
    As ${}^{\tau}g Q g = Q$ then $BY' = C'(I - Z')$ and as  $Z' \in \opn{SL}(2, \ZZ)$ is of order 6 then $I - Z' \in \opn{SL}(2, \ZZ)$.
  Therefore, as $\opn{det} B \neq 0$ then $B^{-1} C' = Y'(I-Z')^{-1} \in M_2(\ZZ)$ and we can apply the change of basis
  \begin{equation*}
    N_2 : =
    \begin{pmatrix}
      I & 0 & 0 \\
      0 & I & -B^{-1} C' \\
      0 & 0 & I
    \end{pmatrix}
  \end{equation*}
  to obtain a $\ZZ$-basis $\{f_i'\}$ for $L$ such that
  \begin{align}\label{bcsplitbasis}
    ((f_i',f_j')) =
    \begin{pmatrix}
      0 & 0 & A \\
      0 & B & 0 \\
      A & 0 & D'
    \end{pmatrix}
    &&
    \text{ and }
    && g=
    \begin{pmatrix}
      U & 0 & W' \\
      0 & I & Y'' \\
      0 & 0 & Z'
    \end{pmatrix}.
  \end{align}
  By \eqref{Qdef}, $B$ is non-degenerate and as $g$ fixes $\la f_3', f_4' \ra$ then $g$ acts on
  $\la f_3', f_4' \ra^{\perp} = \la f_1',f_2',f_5',f_6' \ra$,
  implying $Y''=0$. 
  The result follows.
\end{proof}
 \section{The Branch Divisor of $\fc(\Gamma)$}\label{BranchSec}
As explained in \S\ref{toroidalsec}, the branch divisor of
$\dc_L \rightarrow \fc(\Gamma)$ is precisely the set of rational quadratic divisors
\[
\{ \dc_L(z) \mid \pm \sigma_z \in \Gamma \}.
\]
In this section, we classify vectors $z \in L$ defining $\pm \sigma_z \in \Gamma$ and the associated lattices $z^{\perp} \subset L$.
We will use these results in \S\ref{HMSec} to calculate Hirzebruch-Mumford volumes, which we use to bound the obstruction space $\opn{RefObs}_{(4-a)k}(\Gamma)$ in \S\ref{ObsSec}.
We will not necessarily assume that $n=2$ and we will adopt the convention that $\opn{gcd}(x, 0)=1$ for $x \neq 0$. 
\begin{lemma}\label{conglemma}
  Let $r, s, k, l \in \NN$ where $k$ and $l$ are taken modulo $2r$ and modulo $2s$, respectively. 
If $u=m$ or $2m$ where $m := \opn{lcm}(2r/(2r,k),2s/(2s,l))$,  then the pair $(k,l)$ satisfies conditions \eqref{cong1} - \eqref{cong3}, given by
  \begin{align}
    ur & \mid mk \label{cong1} \\
    2rsu & \mid m^2kl \label{cong2} \\
    2s & \mid m^2 l^2 u^{-1} s^{-1} - 2, \label{cong3} 
  \end{align}
  if and only if one of the following conditions is satisfied.
  \begin{enumerate}
  \item  $k =0$ and either
    \begin{enumerate}
    \item $u=m$ and $l^2 \equiv 1 \bmod{ s}$ if $l$ is odd;
    \item $u=m$ and $l^2 \equiv 2 \bmod{ 2s}$ if $l$ is even;
\item $u=2m$ and $l=2 l_1$ satisfies $2l_1^2 \equiv 2 \bmod{ 2s}$.
\end{enumerate}
  \item $k = r$ and either
    \begin{enumerate}
    \item $u=m=2s$ and $2l^2 \equiv 2 \bmod{ 2s}$ where $m_l$ is even and $l$ is odd; 
    \item $u=m$ and $l^2 \equiv 1 \bmod{ s}$ if $m_l$ is odd and $l$ is even.
\end{enumerate}
  \end{enumerate}
  Additionally, in the above cases,
  \begin{enumerate}
  \item if $k = 0$ and 
    \begin{enumerate}
    \item  $u=m$ then $m=2s$ when $l$ is odd;
    \item  $u=m$ then $m=s$ when $l$ is even;
    \item  $u=2m$ then $m=s$.
    \end{enumerate}
  \item If $k = r$ and
    \begin{enumerate}
    \item  $u=m$ and $l$ odd then $m=2s$;
    \item  $u=m$ and $l$ even then $m=2s$.
    \end{enumerate}
   \end{enumerate}
\end{lemma}
\begin{proof}
  Let 
  $m_k := 2r \opn{gcd}(2r, k)^{-1}$ and
  $m_l := 2s \opn{gcd}(2s,l)^{-1}$.
  
  \noindent$\textbf{u=m.}$ 
  If $u=m$, then \eqref{cong1} - \eqref{cong3} is equivalent to 
  \begin{align}
    r & \mid k \label{cong4} \\
    2rs & \mid m k l \label{cong5} \\
    2s & \mid m l^2 s^{-1} - 2, \label{cong6}
  \end{align}
  and, by \eqref{cong4}, $k=0$ or $r$.
  \begin{enumerate}
  \item If $k = 0$, then \eqref{cong5} is automatically satisfied.
    As $m_k=1$ then $m=\Lcm(m_k, m_l)=m_l$ and $m l = m_l l =  2s l \Gcd(2s, l)^{-1}$.
As $m=m_l$ then \eqref{cong6} is equivalent to $s^{-1} m l^2 \equiv 2l^2 \Gcd(2s, l)^{-1} \equiv 2 \bmod{ 2s}$ and so $l(l \Gcd(2s, l)^{-1}) \equiv 1 \bmod{ s}$.
    Therefore, $\Gcd(s, l) = 1$ and \eqref{cong6} is equivalent to $l^2 \equiv 1 \bmod{ s}$ if $l$ is odd and to $l^2 \equiv 2 \bmod{ 2s}$ if $l$ is even.
  \item If $k= r$ then \eqref{cong5} is equivalent to $2s \mid ml$.
    If $m_l$ is odd then $m=\Lcm(2, m_l)=2 m_l$ and \eqref{cong5} is equivalent to $s \mid m_l l$. 
    If $m_l$ is even then $m=m_l$ and \eqref{cong5} is equivalent to $2s \mid m_l l$. 
    As $m_l l = \Lcm(l, 2s)$ then \eqref{cong5} is satisfied in both cases.
    If $m_l$ is odd then \eqref{cong6} is equivalent to $s^{-1}ml^2 = 2m_l l^2 s^{-1} = 4l^2 \Gcd(2s, l)^{-1} \equiv 2 \bmod{ 2s}$.
    Therefore, $2l^2 \Gcd(2s, l)^{-1} \equiv 1 \bmod{ s}$ and so $\Gcd(s, l)=1$.
    If $m_l$ is even then \eqref{cong6} is equivalent to $m_l l^2 s^{-1} = 2 l^2 \Gcd(2s, l)^{-1} \equiv 2 \bmod{ 2s}$.
    Therefore, $l^2 \Gcd(2s, l)^{-1} \equiv 1 \bmod{ s}$ and $\Gcd(s, l)=1$.
    If $m_l$ is odd then, as $l$ cannot also be odd, \eqref{cong6} is equivalent to $l^2 \equiv 1 \bmod{ s}$. 
    When $m_l$ is even, \eqref{cong6} is given by $l^2 \equiv 1 \bmod { s}$ if $l$ is odd.
    The case of  even $l$ cannot occur: 
    if it did, then $s$ must be even and $l^2 \equiv 2 \bmod{ 2s}$, implying the contradiction $0 \equiv 2 \bmod{ 4}$.
  \end{enumerate}
  In both cases, the statement about $m$ is immediate. \\   
  \noindent$\textbf{u=2m.}$
  If $u=2m$ then \eqref{cong1} - \eqref{cong3} is equivalent to
  \begin{align}
    2r  & \mid k \label{cong7} \\
    4rs & \mid m k l \label{cong8} \\
    2s  & \mid (2s)^{-1}ml^2 - 2. \label{cong9}
  \end{align}
  By \eqref{cong7}, $k \equiv 0 \bmod{ 2r}$ and therefore $m= m_l$. 
  Condition \eqref{cong8} is automatically satisfied.
  Condition \eqref{cong9} is equivalent to $(2s)^{-1}ml^2 = l^2 \Gcd(2s, l)^{-1} \equiv 2 \bmod{ 2s}$.
  When $l=2 l_1$ then $\Gcd(2s, l)=2$ and \eqref{cong9} is equivalent to $2l_1^2 \equiv 2 \bmod{ 2s}$.
  When $l$ is odd \eqref{cong9} admits no solution.
  The statement about $m$ is immediate.   
 \end{proof}
\begin{lemma}\label{qrinvs}
  For primitive $z \in L$, let $e:=\opn{div}(z)$ and define $x$ and $y$ by $z^* = xv^* + yw^* + L$ where $x$ is taken modulo $2r$ and $y$ is taken modulo $2s$.
  \begin{enumerate}
  \item  If $\sigma_z \in \opn{O}^+(L)$ then $\sigma_z \in \Gamma$ if and only if $e,x,y,z^2$ satisfy the conditions of Lemma \ref{conglemma} with $r=d$, $s=n+1$, $u=z^2$, $m=e$ and  $z^2=-e$ or $-2e$.
  \item If $-\sigma_z \in \opn{O}^+(L)$ then $-\sigma_z \in \Gamma$ if and only if $e,x,y,z^2$ satisfy the conditions of Lemma \ref{conglemma} with $r=n+1$, $s=d$, $u=z^2$, $m=e$ and  $z^2 = -e$ or $-2e$.
  \end{enumerate}
\end{lemma}
\begin{proof}
By \eqref{refdef}, $\pm \sigma_z \in \opn{O}^+(L)$ if and only if $z^2<0$ and $z^2 \mid 2e$.
  If $z = e(xv^* + yw^* + l)$ for $l \in L$ then 
\begin{align*}
    \sigma_z(v^*)  & = v^* - \frac{2(v^*, exv^* + eyw^* + el)}{z^2}z  \\
    \sigma_z(w^*)  & = w^* - \frac{2(w^*, exv^* + eyw^* + el)}{z^2}z. 
  \end{align*}
  As $z^2 \mid 2e$ then $z^2 \mid 2e(v^*, l)$, implying 
\begin{align}
    \sigma_z(v^*) & = v^* + (exd^{-1}z^{-2} - 2e(v^*, l)z^{-2})z + L \label{szv1} \\ 
    & = v^* + exd^{-1}z^{-2}(exv^* + eyw^* + el) + L \label{szv2} \\
    & = (1+ e^2x^2d^{-1}z^{-2})v^* + e^2xyd^{-1}z^{-2} w^* + L. \label{szv3} 
  \end{align}
  Equation \eqref{szv2} follows from \eqref{szv1} by noting that $z^2 \mid 2e(v^*, l)$.
  Equation \eqref{szv3} follows from \eqref{szv2} as $e^2x d^{-1}z^{-2}= (ex 2^{-1}d^{-1})(2ez^{-2})$, $2d \mid ex$ and $z^2 \mid 2e$.
  Similarly, $\sigma_z(w^*)$ is given by
\begin{align}
\sigma_z(w^*) & = w^* + (ey(n+1)^{-1}z^{-2} - 2e(w^*, l)z^{-2})z + L \label{szw1}\\
              & = w^* + (ey(n+1)^{-1}z^{-2})(exv^* + eyw^* + el) + L \label{szw2}\\
              & = e^2xy(n+1)^{-1}z^{-2}v^* + (1 + e^2y^2(n+1)^{-1}z^{-2})w^* + L, \label{szw3}
\end{align}
with \eqref{szw2} following from \eqref{szw1} as $z^2 \mid 2e(w^*, l)$ and \eqref{szw3} following from \eqref{szw2} as $(n+1)z^2 \mid e^2y$.

By Proposition \ref{modgp}, $\pm \sigma_z \in \Gamma$ if and only if $\pm \sigma_z(v^*) = v^* + L$ and $\pm \sigma_z(w^*) = -w^* + L$.
Therefore, $\sigma_z \in \Gamma$ if and only if
\begin{align}\label{sigmacong}
&
\begin{cases}
z^2 d \mid e x \\
2d(n+1)z^2 \mid e^2 x y \\
2(n+1) \mid 2 + e^2 y^2 z^{-2}(n+1)^{-1} 
\end{cases}
\end{align}
(following from \eqref{szv2}, \eqref{szw3} and \eqref{szw3}, respectively).
Similarly, $-\sigma_z \in \Gamma$ if and only if
\begin{align}\label{-sigmacong}
&
\begin{cases}
z^2 (n+1) \mid e y \\
2(n+1) d z^2 \mid e^2 x y\\
2d \mid 2 + e^2 x^2 d^{-1} z^{-2}.
\end{cases}
\end{align}
Condition \eqref{sigmacong} is precisely \eqref{cong1} - \eqref{cong3} with $r=d$, $s=n+1$, $u=-z^2$, $m=e$, $x=k$ and $y=l$ and condition \eqref{-sigmacong} is precisely \eqref{cong1} - \eqref{cong3} with $r=n+1$, $s=d$, $u=-z^2$, $m=e$, $x=l$ and $y=k$.
The result follows.
\end{proof}
\begin{proposition}\label{KgenusProp}
  Let $L=2U \op \la -2r \ra \op \la -2s \ra$ where both $r, s > 1$ and let $v$ and $w$ generate the $\la -2r \ra$ and $\la -2s \ra$ factors of $L$, respectively,
  Suppose $z \in L$ is primitive with divisor $m:=\opn{div}(z)$ and length $u:=-z^2>0$ and assume $\pm \sigma_z \in \Gamma$.
  Let  $z^* := z/m   = k v^*+ l w^* + L \in D(L)$ where $k$ and $l$ are taken modulo $2r$ and $2s$, respectively;
  let $K^{u, m}_{r,s}(k,l) := z^{\perp} \subset L$;
  and let $\opn{gen}(K^{u,m}_{r,s}(k,l)):=(2,3; G, -\delta)$ where $G := D(K^{u,m}_{r,s}(k,l))$ and $\delta$ is a finite quadratic form on $G$.
  Then,
  \begin{enumerate}
  \item if $u=m=2s$, $k=0$ and $l$ is odd then $G \cong C_{2r}$ and 
    \begin{center}
      \begin{tabular}{c c c}
        $\delta(a) = \left ( \frac{a^2}{2r} \right ) \bmod{ 2 \ZZ}$
        & for &
        $(a) \in G$; 
      \end{tabular}
    \end{center}
  \item if $u=2m=2s$, $k=0$ and $l$ is even then $G \cong C_2 \op C_2 \op C_{2r}$ and $\delta$ is given by
    \begin{center}
      \begin{tabular}{ c c c}
        $\delta(a,b,c) =
        \frac{(1-l_1^2)a^2}{2s}
        +\frac{(1-l_1^2)ab}{2s}
        +\frac{sb^2}{2}
        + \frac{c^2}{2r}
        \bmod{ 2\ZZ}$
        & for &
        $(a,b,c) \in G$,
      \end{tabular}
    \end{center}
    where $2l_1:=l$;
  \item if $u=m=2s$, $k=r$, $l$ odd and $l^2  \equiv rs \bmod{ 2s}$ then $G \cong C_{2r}$ and 
    \begin{center}
      \begin{tabular}{c c c}
        $\delta(a) = \frac{(1 - sr)a^2}{2r} \bmod{ 2\ZZ}$
        & for &
        $ (a) \in G$;
      \end{tabular}
    \end{center}
  \item if $u=m=2s$, $k=r$, $l$ even and $r$ odd then $G \cong C_2 \op C_r$ and 
    \begin{center}
      \begin{tabular}{c c c}
        $\delta(a,b) =
        \frac{ (-l^2 + 1)a^2}{2s}
        +
        \frac{ l(-l^2 + 1)ab}{2s}
        +
        \frac{(-sr + 1)b^2}{2r}
        \bmod{ 2\ZZ}$
        & for &
        $(a,b) \in G$.
      \end{tabular}
    \end{center}
  \end{enumerate}
  Furthermore, no other cases can occur.
\end{proposition}
\begin{proof}
  We calculate $q_{z^{\perp}}$, following the approach of \cite[Proposition 1.15.1]{Nikulin}.
  As explained in \S4-5 of \cite{Nikulin}, if $M \subset P$ is a primitive embedding of lattices and $N:=M^{\perp} \subset P$ then the series of overlattices
  \begin{equation*}
  M \op N \subset P \subset P^{\vee} \subset M^{\vee} \op N^{\vee} 
  \end{equation*}
  defines a totally isotropic subgroup
  \begin{equation}\label{overiso}
    H_P:=\frac{P}{M \op N} \subset D(M) \op D(N)
  \end{equation}
  and natural projections $p_M:H_P \rightarrow D(M)$ and $p_N:H_P \rightarrow D(N)$.
  The overlattice $P$ is uniquely determined by the subgroup $H_P$ and the embeddings $M \subset P$ and $N \subset P$ are primitive if and only if $P_M$ and $P_N$ are injective \cite[p.111]{Nikulin}.

  Let $L \subset W$ be a primitive embedding where $W$ is unimodular (suitable $W$ always exists by Corollary 1.12.3 of \cite{Nikulin}) and define $K:=\la z \ra^{\perp} \subset L$ and $T := L^{\perp} \subset W$.
  By \eqref{overiso}, if $w_L \in L^{\vee}$ represents an element of $p_L(H_W)$ then there exists $w \in W$ such that $w = w_L + w_T$ and $w_T \in T^{\vee}$ is unique modulo $T$.
  This correspondence defines an isomorphism
  $\gamma_{TL}:q_L \rightarrow -q_T$
  defined on representatives in $L^{\vee}$ and $T^{\vee}$ by
  $\gamma_{TL}:w_L \mapsto w_T$ \cite[\S5]{Nikulin}.
  (In the terminology of \cite{SPLAG}, the map $\gamma_{TL}$ defines \emph{`glue vectors'}.)
  
  By the above discussion, and by considering the overlattice
  \[
  K \op \la z \ra \op T \subset W,
  \]
  if $V := K^{\perp} \subset W$ then $V$ is an overlattice of $\la z \ra \op T$ and $\gamma_{VK}$ defines an isomorphism $q_K \cong -q_V$, which we use to calculate $q_K$.
As in \eqref{overiso}, $V$ is defined by the totally isotropic subgroup
  \[
  H_V \subset  D(\la z \ra) \op D(T)
  \]
  and we let $P_z$ and $P_T$ denote the associated projections.
    As $T$ is primitive in $W$ then $T$ is primitive in $V$;
    as $\la z \ra$ is primitive in $L$ and $L$ is primitive in $W$ then $\la z \ra$ is primitive in $W$ and so $\la z \ra$ is primitive in $V$,
    implying $P_z$ and $P_T$ are injective.
Let $\alpha$ denote $z/u \in D(\la z \ra)$ and let $z^*=z/m=(k,l) \in D(L) \cong C_{2r} \op C_{2s}$.
    The group $P_z(H_V)$ is cyclic and we take a generator $\xi \alpha$ where $\xi \in \ZZ$.
    By injectivity of $P_T$, $\xi \alpha + v_T = v \in V$ where $v_T \in T^{\vee}$ is unique modulo $T$.
    As $q_L \cong -q_T$ and $\xi \alpha \in L^{\vee}$ (noting that $(l, \xi \alpha + v_T) = (l, \xi \alpha) \in \ZZ$ for all $l \in L^{\vee}$), then
    \begin{equation}\label{xialpha}
      \xi \alpha = \xi m u^{-1} z^* \equiv \xi m u^{-1} (k,l) \in D(L).
    \end{equation}
    As $p_z(H_v)$ is cyclic and $p_z(H_V) \cong H_V$ then
    $H_V = \la (\xi \alpha ,\xi m u^{-1} (k,l) ) \ra \subset D(\la z \ra) \op D(T)$.
  By Proposition 1.4.1 and Corollary 1.6.2 of \cite{Nikulin},
  \begin{equation*}
    -q_K = q_V = (q_{\la z \ra} \op q_T \vert H_V^{\perp})/H_V,
  \end{equation*}
  which can be calculated once a suitable $\xi$ is determined: we claim one can take $\xi=u/m$.
  
  We assume $\xi \neq 0$; if $z^* \equiv 0 \bmod{ L }$, this is immediate from \eqref{xialpha};
  if $z^* \not \equiv 0 \bmod{ L }$ and $\xi = 0$ then $H_V = 0$, implying $\la z \ra \op T$ has no overlattice in $K^{\perp}$ (i.e. $D(\la z \ra) \op D(T)$ has no non-trivial isotropic subgroup by \cite[\S4]{Nikulin}), which is a contradiction as $z^* \not \equiv 0 \bmod{L}$.
  Each $\xi$ defines an element $\xi u^{-1}z \in L^{\vee}$ (i.e. $\xi^{-1}u \vert m$).
  Conversely, for all $0 \neq \xi \in \NN$ such that $\xi'u^{-1} z \in L^{\vee}$ (by definition of $\gamma_{TL}$) there exists $v_T \in T^{\vee}$ such that $v = \xi' \alpha + v_T \in W$ and $v \in K^{\perp}$.
  If $\xi' \in \NN$ satisfies $\xi' \vert \xi$,  $\xi' \neq \xi$ and $\xi'^{-1}u \vert m$ then $\xi'$ defines an overlattice $V \subset V' \subset K^{\perp}$ distinct from $V$, implying $V \neq K^{\perp}$.
  Therefore, we can take $\xi = u/m$.
  
  For the rest of the proof, we identify $D(\la z \ra)$ and $D(T)$ with $\ZZ/u \ZZ$ and $\ZZ/(2r) \ZZ \op \ZZ / (2s) \ZZ$,  respectively; we use $x = (x_1, x_2, x_3)$ to denote an element of $D(\la z \ra) \op D(T)$; and we define the map $\gamma:H_V \rightarrow D(T)$ by $\xi \alpha \mapsto \xi m u^{-1} (k,l)$ where $H_V$ is identified with $\la \xi \alpha \ra \subset D(\la z \ra)$ for suitable $\xi \in \ZZ$.
We now consider cases for $u$, $m$, $k$, $l$ using Lemma \ref{qrinvs}.\\

\noindent $\textbf{u=m=2s, k=0, l \text{odd}.}$ 
The discriminant form on $D(\la z \ra) \op D(T)$ is given by
\begin{equation}\label{u=m=2s,k=0,form}
  \frac{-a^2}{2s} +  \frac{b^2}{2r} + \frac{c^2}{2s} \bmod{ 2\ZZ}
\end{equation}
for $(a,b,c) \in D(\la z \ra) \op D(T)$.
We take $\xi=1$ and so $H_V$
is generated by $y_3:= (1,0,l)$.
An element $x \in D(\la z \ra) \op D(T)$ belongs to $H_V^{\perp} \subset D(\la z \ra) \op D(T)$ if and only if
\begin{equation*}
  \frac{-x_1}{2s} + \frac{l x_3}{2s} \equiv 0 \bmod{ \ZZ};
\end{equation*}
or, equivalently, if
\begin{equation}\label{u=m=2s,k=0,orth}
  x_1 \equiv  l x_3 \bmod{ 2s}.
\end{equation}
The elements $y_1 = (l,0,1)$, $y_2 = (0,1,0)$ both satisfy \eqref{u=m=2s,k=0,orth} and so $H_V^{\perp} = \la y_1, y_2 \ra$.
To determine relations, if $a y_1 + b y_2 = c y_3$ for $a,b,c \in \NN$ then $b \equiv 0 \bmod{2r}$.
By Lemma \ref{conglemma}, $l^2 \equiv 1 \bmod{ s}$, implying  $2y_1 = 2ly_3$ and so $y_1 \bmod{H_V}$ is of order 1 or 2.
We have $y_1 \in H_V$ if and only if $l^2 \equiv 1 \bmod{2s}$, implying 
  \[
  H_V^{\perp}/H_V \cong
  \begin{cases}
    C_2 \op C_{2r} & \text{if $l^2 \equiv 1 + s \bmod{2s}$} \\
    C_{2r} & \text{if $l^2 \equiv 1 \bmod{2s}$}.
  \end{cases}
  \]
  The case $l^2 \equiv 1 + s \bmod{2s}$ cannot occur as, by considering $K^{u,m}_{r,s}(k,l) \otimes \ZZ_2$ and applying Theorem 1.9.1 and  Corollary 1.9.3 of \cite{Nikulin}, we obtain the contradiction that $K^{u,m}_{r,s}(k,l)$ is of even rank. 
  Therefore, $H_V^{\perp}/H_V \cong C_{2r}$ and  
  \[
  q_V(a)=\frac{a^2}{2r}  \bmod{ 2\ZZ}
  \]
  for $a \in C_{2r} \cong \ZZ/(2r)\ZZ$.

  \noindent $\textbf{u=m=s, k=0, l \text{even}.}$ 
  The discriminant form of $D(\la z \ra) \op D(T)$ is given by
\[
  \frac{-a^2}{s} +  \frac{b^2}{2r} + \frac{c^2}{2s}  \bmod{ 2\ZZ}
  \]
for $(a,b,c) \in D(\la z \ra) \op D(T)$.
  We take $\xi = 1$ and $y_3:=(1,0,l)$ as a generator of $H_V$.
  An element $x \in D(\la z \ra) \op D(T)$ belongs to $H_V^{\perp}$ if and only if
\[
  \frac{-x_1}{s} + \frac{l x_3}{2s} \equiv 0 \bmod{ \ZZ};
  \]
or equivalently, as $l=2l_1$, if
  \begin{equation*}
    x_1 \equiv l_1 x_3 \bmod{ s}.
  \end{equation*}
  Therefore, $H_V^{\perp} = \la y_1, y_2 \ra$ where $y_1:=(l_1, 0, 1)$ and $y_2 = (0, 1, 0)$.
To determine relations, suppose $a y_1 + b y_2 = c y_3$ where $a,b,c \in \NN$,
  $a$ is taken modulo $2s$,
  $b$ is taken modulo $2r$
  and $c$ is taken modulo $s$.
  Then $b = 0$ and $a \equiv cl \bmod{2s}$, implying $a$ is even.
  By Lemma \ref{conglemma}, $l^2 \equiv 2 \bmod{ 2s}$ and if $a=2$ and $c=l$ then $a y_1 = c y_3$.
  Therefore, $y_1$ and $y_2$ modulo $H_V$ are of order $2$ and $2r$, respectively and satisfy no non-trivial relations, implying $H_V^{\perp}/H_V \cong C_2 \op C_{2r}$; however, this case cannot occur for rank reasons, as for $u=m=2s$, $k=0$, $l$ odd.

  \noindent $\textbf{u=2m, m=s, k=0, l \text{even}.}$ 
  The form on $D(\la z \ra) \op D(T)$ is as in \eqref{u=m=2s,k=0,form} and $l=2l_1$.
  We take $\xi=2$ and $y_4 := (2,0,l) \in D(\la z \ra) \op D(T)$ as a generator of $H_V$. 
  An element $x \in D(\la z \ra) \op D(T)$ belongs to $H_V^{\perp}$ if and only if
\[
  \frac{-x_1}{s} + \frac{l_1 x_3}{s} \equiv 0 \bmod{ \ZZ};
  \]
or, equivalently, if
  \begin{equation*}
    x_1 \equiv l_1 x_3 \bmod{ s}.
  \end{equation*}
  As $2l_1^2 \equiv 2 \bmod{ 2s}$ then
  $H_V^{\perp} = \la y_1, y_2, y_3 \ra$ where
  $y_1 = (l_1, 0, 1)$,
  $y_2=(0,1,0)$ and 
  $y_3=(s,0,0)$.
  Suppose $a y_1 + b y_2 + c y_3 = d y_4$ where
  $a$ is taken modulo $2s$,
  $b$ is taken modulo $2r$,
  $c$ is taken modulo $2$ and
  $d$ is taken modulo $s$.
  Then $b=0$ and $(a-dl)(l_1,0,1) = c(s,0,0)$, implying $a \equiv dl \bmod{ 2s}$ and $c \equiv 0 \bmod{ 2}$.
  As $l$ is even then $a=:2a_1$ is even and $a \equiv 2dl_1 \bmod{ 2s}$ is equivalent to $a_1 \equiv d l_1 \bmod{ s}$, which always admits a solution as $2l_1^2 \equiv 2 \bmod{ 2s}$ implies $\opn{gcd}(l_1, s)=1$.
  Therefore, $y_1$, $y_2$ and $y_3$ modulo $H_V$ are of order $2$, $2r$ and $2$, respectively and satisfy no non-trivial relations.
  Hence, $H_V^{\perp}/H_V \cong C_2 \op C_2 \op C_{2r}$ and
  \[
  q_V(a,b,c) \equiv 
  \frac{(1-l_1^2)a^2}{2s}
  -
  \frac{(1-l_1^2)l_1 ab}{2s}
  + \frac{sb^2}{2}
  +
  \frac{c^2}{2r} \bmod{ 2 \ZZ}
  \]
  for $(a,b,c) \in C_2 \op C_2 \op C_{2r}$.
  
  \noindent $\textbf{u=m=2s, k=r, l \text{odd}.}$
  The form on $D(\la z \ra) \op D(T)$ is as in \eqref{u=m=2s,k=0,form}.
  We take $\xi = 1$ and $y_3:=(1,r,l)$ as a generator for $H_V$.
  We have $x \in H_V^{\perp}$ if and only if
  \begin{equation*}
    \frac{-x_1}{2s} + \frac{r x_2}{2r} + \frac{l x_3}{2s} \equiv \bmod{ \ZZ};
  \end{equation*}
  or, equivalently, if 
  \begin{equation}\label{u=m=2s,k=r,orth}
    \begin{cases}
      x_1 \equiv l x_3 \bmod{ 2s} & \text{if $x_2$ even} \\
      x_1 \equiv l x_3 + s x_2 \bmod{ 2s} & \text{if $x_2$ odd}.
    \end{cases}
  \end{equation}
  The elements $y_1:=(l,0,1)$ and $y_2:=(s,1,0)$ both satisfy \eqref{u=m=2s,k=r,orth} and one checks that all $x \in H_V^{\perp}$ satisfy $x = x_3 y_1 + x_2 y_2$, implying $H_V^{\perp} = \la y_1, y_2 \ra$.
By Lemma \ref{conglemma}, $2y_1 = (2l,0,2)=2l(1,r,l) = 2l y_3$ and so $y_1$ modulo $H_V$ is of order $1$ or $2$. 
  If $y_1 = a y_3$ for $a \in \NN$ then $ar \equiv 0 \bmod{ 2r}$ and so $a$ is even; on the other hand, $a \equiv l \bmod{ 2s}$ implies $a$ is odd; therefore, $y_1$ modulo $H_V$ is of order $2$. 
  Now suppose $a y_2 = b y_3$ for $a, b \in \NN$ taken modulo $2r$ and $2s$, respectively. 
  By Lemma \ref{conglemma}, $2l^2 \equiv 2 \bmod{ 2s}$ and as $l$ is odd then $\opn{gcd}(l, 2s)=1$.
  Therefore $b=0$ and $a=0$, implying $y_2$ modulo $H_V$ is of order $2r$.
  
  We show that $y_1 + a y_2 = b y_3$ if and only if $l^2 - 1 \equiv rs \bmod{ 2s}$.
  Suppose that 
  $l-as-b \equiv 0 \bmod{ 2s}$,
  $a-br \equiv 0 \bmod{ 2r}$ and
  $1-bl \equiv 0 \bmod{ 2s}$.
  As $\opn{gcd}(l,2s)=1$ then $b \equiv l^{-1} \bmod{ 2s}$.
  As $l$ is odd then $b$ is odd and so $a \equiv r \bmod{ 2r}$.
  Therefore $l-rs - l^{-1} \equiv 0 \bmod{2s}$ and $l^2 - rs - 1 \equiv 0 \bmod{ 2s}$.
  The converse is immediate and so, 
  \begin{equation*}
    H_V^{\perp}/H_V \cong
    \begin{cases}
      C_{2r} & \text{if $l^2 - 1 \equiv rs \bmod{ 2s}$}\\
      C_2 \op C_{2r} & \text{if $l^2 - 1 \not \equiv rs \bmod{ 2s}$}.
    \end{cases}
  \end{equation*}
  As for the case of $u=m=2s$, $k=0$, $l$ odd, the case $l^2 - 1 \not \equiv rs \bmod{ 2s}$ cannot occur for rank reasons.
  In the remaining case, $q_V$ is given by 
  \[
  q_V(a) = \frac{(1-sr)a^2}{2r} \bmod{ 2 \ZZ}
  \]
  for $a \in C_{2r}$.

  \sloppy
  \noindent  $\textbf{u=m=2s, k=r, l \text{even}.}$ 
  The quadratic form on $D(\la z \ra) \op D(T)$ is as in  \eqref{u=m=2s,k=0,form}.
  We take $\xi=1$ and $y_3 := (1,r,l)$ as a generator of $H_V$.
  As for the case of $u=m=2s$, $k=r$, $l$ odd, $H_V^{\perp} = \la y_1, y_2 \ra$ where $y_1:=(l,0,1)$ and $y_2:=(s,1,0)$.
  By Lemma \ref{conglemma}, $l^2 \equiv 1 \bmod{ s}$ and so $2l^2 \equiv 2 \bmod{ 2s}$.
  Therefore, $2y_1 = 2l y_3$ and so $y_1$ modulo $H_V$ is of order 1 or 2. 
  If $y_1 = b y_3$ for $b$ taken modulo $2s$ then $bl \equiv 1 \bmod{ 2s}$, implying the contradiction $\opn{gcd}(l, 2s)=1$ and so $y_1$ modulo $H_V$ is of order 2. 
  Suppose $a y_2 = b y_3$ for $a$ and $b$ taken modulo $2r$ and  $2s$, respectively.
  Then $bl \equiv 0 \bmod{ 2s}$ and, as $\opn{gcd}(l, s)=1$, we have $b \equiv 0 \bmod{ s}$.
  If $b$ is even then $a=0$; 
  if $b$ is odd then, as $l^2 \equiv 1 \bmod{ s}$ and $l$ is even, $s$ is odd.
  As $as \equiv b \bmod{ 2s}$ then $a \equiv b \bmod{ 2}$ and so $a$ is odd;
  and as $a \equiv br \bmod{ 2r}$ and $b$ is odd, then $a \equiv r \bmod{ 2r}$, implying $r$ is odd.
  Hence, $y_2$ modulo $H_V$ is of order $2r$ if $r$ is even and order $r$ if $r$ is odd.
  If $y_1 + a y_2 = b y_3$ for $a,b \in \NN$ then $1-bl \equiv 0 \bmod{ 2s}$, which is a contradiction as $l$ is even.
  Therefore, there are no non-trivial relations between $y_1$ and $y_2$ modulo $H_V$, implying 
  \begin{equation*}
    H_V^{\perp}/H_V \cong
    \begin{cases}
      C_2 \op C_{2r} & \text{if $r$ is even} \\
      C_2 \op C_r & \text{if $r$ is odd.}
    \end{cases}
    \end{equation*}
  As for $u=m=2s$, $k=0$, $l$ odd, the case of even $r$ cannot occur for rank reasons.
  Therefore, if $r$ is odd, then  $q_V$ is given by
\[
  q_V(a,b) \equiv
  \frac{(-l^2 + 1)a^2}{2s}
  +
  \frac{l(l^2 - 1)ab}{2s}
  +
  \frac{(1-sr)b^2}{2r} \bmod{ 2 \ZZ}
  \]
  for $(a,b) \in C_2 \op C_r$.
\end{proof}
 \section{Extension Criteria}\label{ExtensionSec}
We now prove the extension criteria used in Theorem \ref{lwcft}, starting with an exercise on toric varieties.
For a finite group $G \subset \opn{GL}(n, \CC)$, we let $X_G$ denote the quotient $\CC^n/G$.
\begin{lemma}\label{toriclemma}
  Take coordinates
  $(y_1, y_2, y_3, y_4)$
  for $\CC^4$ and suppose $g \in \opn{GL}(4, \CC)$ is given by
  $\frac{1}{6}(1,1,0,2)$,
  $\frac{1}{6}(0,1,1,2)$
  or
  $\frac{1}{6}(1,1,1,2)$. 
  Let 
  \[
  \Omega = h(y_1, y_2, y_3, y_4) 
  \frac{(dy_1 \wedge \ldots \wedge dy_4)^{\otimes k}}{(y_1 \ldots y_4)^{\otimes k}}
  \]
  be a $g$-invariant differential form for $\CC^4$.
  If $h$ vanishes to order $\geq 3k$ along $y_2 =0$ and to order $\geq k$ along $y_i$ for all $i \neq 2$, then $\Omega$ extends to a pluricanonical form on a desingularisation of $X_{\la g \ra}$.  
\end{lemma}
\begin{proof}
  \noindent $\bold{\frac{1}{6}(1,1,0,2)}$
  Let $g = \frac{1}{6}(1,1,0,2)$. 
  As in \cite{Tai},  $X_{\la g \ra}$ can be realised as a toric variety $X(\sigma)$ as follows.
  Let $N$ be a lattice with basis
  $p_1=(1,0,0,0)$,
  $p_2=(0,1,0,0)$,
  $p_3=(0,0,1,0)$,
  $p_4=(0,0,0,1)$
  and let $\overline{N}$ be an overlattice of $N$ such that $\overline{N}/N$ is generated by $\frac{1}{6}(1,1,0,2) \in N \otimes \QQ$.
  On the basis $\{ p_i \}$ the lattice $\overline{N}$ is spanned by
  $q_1=\frac{1}{6}(1,1,0,2)$,
  $q_2=(0,1,0,0)$,
  $q_3=(0,0,1,0)$,
  $q_4=(0,0,0,1)$;
  and the cone $\sigma$, given by
  \[
    \sigma =
    \left \{
    \sum_{i=1}^4 \lambda_i p_i \mid \lambda_i \geq 0
    \right \}
    \subset \overline{N} \otimes \RR,
  \]
  is generated by
  $(6,-1,0,-2)$,
  $(0,1,0,0)$,
  $(0,0,1,0)$,
  $(0,0,0,1)$
  on the basis $\{q_i\}$.
  The toric variety $X(\sigma)$ can be resolved by subdividing $\sigma$ to obtain a fan $\Sigma$ whose rays are given by
  $u_1=(6,-1,0,-2)$,
  $u_2=(0,1,0,0)$,
  $u_3=(0,0,1,0)$,
  $u_4=(0,0,0,1)$,
  $u_5=(3,0,0,-1)$,
  $u_6=(1,0,0,0)$
  on the basis $\{q_i\}$;
  and whose maximal cones are given by
  $\{u_1, u_3, u_4, u_6 \}$,
  $\{u_1, u_3, u_5, u_6 \}$,
  $\{u_2, u_3, u_4, u_6 \}$,
  $\{u_2, u_3, u_5, u_6 \}$.
  (We performed the calculation using the package \texttt{NormalToricVarieties} for the computer algebra system  \texttt{Macaulay2}, but the result can be easily checked by hand using standard toric geometry \cite{Fulton}.)
  
  As in \cite{Tai}, if 
  \begin{align*}
    \Omega
    & = h(y_1, \ldots, y_4) \frac{ (dy_1 \wedge \ldots \wedge dy_4)^{\otimes k}}{(y_1 \ldots y_4)^k} \\
    & = h(y_1^*, \ldots, y_4^*) \frac{ (dy_1^* \wedge \ldots \wedge dy_4^*)^{\otimes k}}{(y_1^* \ldots y_4^*)^k}, 
  \end{align*}
  where $y_1^*, \ldots, y_4^*$ are coordinates defined by the rays of a maximal cone of $\Sigma$,
  then $\Omega$ extends to a pluricanonical form on $X(\Sigma)$ if $h$ vanishes to order $\opn{ord}_{y_j^*} h \geq k$ along each of the divisors $y_i^*=0$.
If $y_i^*$ is given by the ray
  $v = \sum \lambda_j p_j$ where $\lambda_j \in \QQ$, $\lambda_i \geq 0$, then
  \[
    \opn{ord}_{y_i^*} h =
    \sum \lambda_j \opn{ord}_{y_j} h. 
  \]
On the basis $\{p_i\}$ we have
  $u_1 = (1,0,0,0)$,
  $u_2 = (0,1,0,0)$,
  $u_3 = (0,0,1,0)$,
  $u_4 = (0,0,0,1)$,
  $u_5=\frac{1}{2}(1,1,0,0)$ and
  $u_6 = \frac{1}{6}(1,1,0,2)$.
  From the statement of the lemma,  $\opn{ord}_{y_2}(h) \geq 3k$ and $\opn{ord}_{y_i}(h) \geq k$ for $i \neq 2$ and we calculate
  $\opn{ord}_{y_1^*}(h) \geq k$, 
  $\opn{ord}_{y_2^*}(h) \geq 3k$, 
  $\opn{ord}_{y_3^*}(h) \geq k$, 
  $\opn{ord}_{y_4^*}(h) \geq k$, 
  $\opn{ord}_{y_5^*}(h) \geq 2k$ 
  and 
  $\opn{ord}_{y_6^*}(h) \geq k$.
  \\
  \\
\noindent $\bold{\frac{1}{6}(0,1,1,2)}$
We proceed as for $g=\frac{1}{6}(0,1,1,2)$. 
For $\{p_i\}$ defined previously, we take generators 
$q_1 = \frac{1}{6}(0,1,1,2)$,
$q_2=(1,0,0,0)$,
$q_3=(0,0,1,0)$ and
$q_4 = (0,0,0,1)$ for $\overline{N} \supset N$; and set $X_{\la g \ra} = X(\sigma)$ where the cone $\sigma$ is generated by the rays
$(0,1,0,0)$,
$(6,0,-1,-2)$,
$(0,0,1,0)$ and
$(0,0,0,1)$ on the basis $\{q_i\}$. 
We resolve $X(\sigma)$ by subdividing $\sigma$ to obtain the fan $\Sigma$ with with rays
$u_1=(0,1,0,0)$,
$u_2=(6,0,-1,-2)$,
$u_3=(0,0,1,0)$,
$u_4=(0,0,0,1)$,
$u_5=(3,0,0,-1)$ and
$u_6=(1,0,0,0)$ on the basis $\{q_i\}$;
and maximal cones
$\{u_1, u_2, u_4, u_6\}$,
$\{u_1, u_2, u_5, u_6 \}$,
$\{u_1, u_3, u_4, u_6 \}$ 
and 
$\{u_1, u_3, u_5, u_6 \}$.
We have
$u_1=(1,0,0,0)$,
$u_2=(0,1,0,0)$,
$u_3=(0,0,1,0)$,
$u_4=(0,0,0,1)$,
$u_5=\frac{1}{2}(0,1,1,0)$ and
$u_6 = \frac{1}{6}(0,1,1,2)$ on the basis $\{p_i\}$;
from which we calculate 
$\opn{ord}_{z_i^*}(h) \geq k$ for all $i$ and
$\opn{ord}_{z_2^*}(h) \geq 3k$.

\noindent  $\bold{\frac{1}{6}(1,1,1,2)}$
Similarly, for $g=\frac{1}{6}(1,1,1,2)$, we realise $X_{\la g \ra}$  as the toric variety $X(\sigma)$ where  $\overline{N}$ is generated by 
$q_1=\frac{1}{6}(1,1,1,2)$,
$q_2=(0,1,0,0)$,
$q_3=(0,0,1,0)$,
$q_4=(0,0,0,1)$ on the basis $\{p_i\}$, with $\{p_i\}$ is as before;
and the cone $\sigma$ is generated by the rays 
$(6,-1,-1,-2)$,
$(0,1,0,0)$,
$(0,0,1,0)$,
$(0,0,0,1)$
on $\{q_i\}$.
We resolve $X(\sigma)$ by subdividing $\sigma$, obtaining a fan $\Sigma$ with rays
$u_1=(6,-1,-1,-2)$,
$u_2=(0,1,0,0)$,
$u_3=(0,0,1,0)$,
$u_4=(0,0,0,1)$,
$u_5=(3,0,0,-1)$,
$u_6=(1,0,0,0)$
on $\{q_i\}$;
and maximal cones
$\{u_1, u_2, u_4, u_6 \}$,
$\{u_1, u_2, u_5, u_6 \}$,
$\{u_1, u_3, u_4, u_6 \}$,
$\{u_1, u_3, u_5, u_6 \}$,
$\{u_2, u_3, u_4, u_6 \}$ 
and 
$\{u_2, u_3, u_5, u_6 \}$.
On the basis $\{p_i\}$ the rays $u_i$ are given by
$u_1 = (1,0,0,0)$,
$u_2 = (0,1,0,0)$,
$u_3 = (0,0,1,0)$,
$u_4 = (0,0,0,1)$,
$u_5 = \frac{1}{2}(1,1,1,0)$ 
and 
$u_6 = \frac{1}{6}(1,1,1,2)$,
from which we conclude
$\opn{ord}_{z_i^*}(h) \geq k$ for all $i$ and
$\opn{ord}_{z_2^*}(h) \geq 3k$.
\end{proof}
\begin{proposition}\label{intsingextprop}
  Suppose $\ec$ is taken so that for each
  primitive sublattice $2U \subset L$ there exists a embedding $2U \op \la v \ra \hookrightarrow L \in \ec$ for $v \in (2U)^{\perp}$.
  Then the differential form $\Omega(f)$ of Theorem \ref{lwcft} extends to a pluricanonical form on $\widetilde{\fc}(\Gamma)$ above the singular locus of $\fc(\Gamma)$.
\end{proposition}
\begin{proof}
  By \cite[Proposition 3.2]{Tai}, a $G$-invariant pluricanonical form on $\CC^n$ extends to a pluricanonical form on a desingularisation of $X_G$ if and only if it extends to a pluricanonical form on a desingularisation of $X_{\la g \ra}$ for each $g \in G$.
  We use this criterion to show that $\Omega(F)$ extends to $\widetilde{\fc}(\Gamma)$ by studying $\Omega(F)$ in  a neighbourhood of a non-canonical singularity. 
  By Lemma \ref{DeltaLemma}, non-canonical singularities of  $\fc(\Gamma)$ correspond to the fixed loci of $\Delta_{\iota}(g_6, I)$ or $\Delta_{\iota}(I, g_6) \in \Gamma$ for a primitive embedding $\iota:2U \hookrightarrow L$.
  Here we consider the case of $g:= \Delta_{\iota}(g_6, I)$, with $\Delta_{\iota}(I, g_6)$ following by an identical argument.
  
  For an arbitrary primitive embedding $2U \subset L$, let $B:=\iota(2U)^{\perp} \subset L$ and, for $v$ as in the statement of the lemma, suppose $u \in B$ is primitive and orthogonal to $v$.
  We begin by defining an isomorphism $\hc \rightarrow \dc_L$, following \cite[p.508]{Handbook}.
  Let $C^+(U \op B)$ be a connected component of the cone
  \[
  C(B \op U) =
  \{
  x \in (B \op U) \otimes \RR \mid (x,x)>0
  \},
  \]
  and let $\hc$ be the tube domain 
  \[
    \hc := (B \op U) \otimes \RR + i C^+(B \op U).
  \]
For $i=1,2$, let $\{e_i, f_i \}$ be standard bases for the two copies of $U$ in $\iota(2U)$ and let $(x_1, x_2, x_3, x_4) \mapsto x_1v + x_2u + x_3 e_1 + x_4 f_1 \in \hc \subset (B \op U) \otimes \CC$ define coordinates for $\hc$.
  We define the isomorphism $\hc \rightarrow \dc_L$  by 
  \[
  \hc \ni \underline{x}:=(x_1, x_2, x_3, x_4) \mapsto [x_1v + x_2 u + x_3 e_1 + x_4 f_1 + e_2 + \frac{1}{2}x_6 f_2 ] \in \dc_L \subset \PP(L \otimes \CC)
  \]
  where $x_6 := -(2x_3 x_4 + x_1^2v^2 + 2x_1x_2(u,v) + x_2^2u^2)$.
  If
\[
  h = \bpm a & b \\ c & d \epm \in \opn{SL}(2, \ZZ),
  \]
  then 
  \begin{equation}\label{GammaAction}
  \Delta_{\iota}(h, I):
  \underline{x} \mapsto
  \left(
  \frac{x_1}{a-bx_4}, \frac{x_2}{a-bx_4}, x_3 - \frac{b(x_1^2v^2 + x_2^2 u^2)}{2(a-bx_4)}, \frac{dx_4 - c}{a-bx_4}
  \right ).
  \end{equation}
  As we fixed a component of the cone $C(U \op B)$, if $\Delta_{\iota}(g_6, I) \underline{x} = \underline{x}$ then $x_4 = e^{\pi i /3}=:\xi$, $x_3$ is free and $x_1=x_2=0$.

  Following the approach of \cite{Cartan} (and the explanation of \cite[p.331]{HuybrechtsK3}), we introduce local coordinates on $\hc$ linearising the action of an order $n$ element $g \in \Gamma$.  
  For a fixed point $\underline{w} \in \hc$ of $g$, let $\underline{x}' := \underline{x} - \underline{w}$ and $(g_1, \ldots, g_4):=g(\underline{x'})$.
  If  
  \[
  \textbf{d}_0(g) :=  \left. \left ( \frac{\partial g_i}{\partial x_j'} \right ) \right \vert_{\underline{x}'=0}
  \]
  then
  \[
  \underline{y}(\underline{x}'):=\frac{1}{n} \sum_{i=1}^n
 \textbf{d}_0(g)^{-i} g^i(\underline{x}') 
 \]
 defines local coordinates $\underline{y}=(y_1, y_2, y_3, y_4)$ around $\underline{w}$ on which $g$ acts by $\textbf{d}_0(g)$.
 Therefore, if $g=\Delta_{\iota}(g_6, I)$ then, for a sixth root of unity $\xi$, 
 \begin{equation}\label{LinearCoords}
 \textbf{d}_0(g) = \opn{diag}(\xi, \xi, 1, \xi^2)
  \end{equation}
  and from \eqref{GammaAction} and \eqref{LinearCoords}, 
 \begin{equation}\label{yQcoords}
 \underline{y} = \left ( x_1' Q_1(x_4'), x_2' Q_2(x_4'),  Q_3(\underline{x'}), Q_4(\underline{x'})\right )
 \end{equation}
 where all $Q_i$ are holomorphic and neither $Q_1$ nor $Q_2$ vanish at zero\footnote{
 To obtain the terms
 \[
 Q_1(x_4')=Q_2(x_4')=\frac{1}{3}\left(1-\frac{1}{\xi^2(x_4' + \xi)} + \frac{1}{\xi^4(x_4'+\xi-1)}\right)
 \]
 in \eqref{yQcoords} with minimal calculation, one can simply determine $\underline{y}$ for $\Delta(g_3, I)$ and note that $\Delta(-I, I)$ acts by $-1$ on $y_1$, $y_2$ but fixes $y_3$, $y_4$, implying $\underline{y}$ are also linear coordinates for $\Delta(g_6, I)$.}.
 As the modular form $f(\underline{x'})$ vanishes to order $2k$ along $x_2'=0$ then $f(\underline{y})$ vanishes to order $2k$ along $y_2=0$.
 By Lemma \ref{toriclemma}, $\Omega(f)$ extends to $\widetilde{\fc}(\Gamma)$ above a fixed point of $\Delta_{\iota}(g_6, I)$, from which the result follows. 
\end{proof}
\begin{proposition}\label{bcextprop}
  Suppose $\ec$ is taken so that for each
  boundary curve $F$ there exists an embedding $u^{\perp} \hookrightarrow L \in \ec$ for a vector $u$ in the lattice $B$ of \eqref{bcsplitbasis}.
Then the differential form $\Omega(f)$ of Theorem \ref{lwcft} extends to a pluricanonical form on $\widetilde{\fc}(\Gamma)$  above all boundary curves.
\end{proposition}
\begin{proof}
  We proceed as in Proposition \ref{intsingextprop}.
  By Lemma \ref{bcsingclass}, it suffices to consider points in the boundary fixed by $g \in G(F)$ acting as $g = \frac{1}{6}(1,1,1,2)$ or $\frac{1}{6}(0,1,1,2)$.
 In either case, we take the basis $\{f_i\}_{i=1}^6$ of $L \otimes \QQ$ defined in Lemma \ref{bcsplitlem}, on which $g$ is represented by 
\[
 \begin{array}{c c c}
      g =
      \bpm
      U & 0 & W' \\
      0 & I & 0 \\
      0 & 0 & Z
      \epm
      \in N(F)_{\ZZ}
      & \text{ for } &
      Z= \bpm
      a & b \\ c & d
      \epm \in \opn{SL}(2, \ZZ).
    \end{array}
  \]
  If $(z, \underline{w}, \tau)$ are the coordinates for $\dc_L(F)$ defined in  \eqref{DLFcoord} and  $r=\opn{exp}_e(z)$ (where $e$ is as in \eqref{Adeltae}) then $(r, \underline{w}, \tau)$  are  coordinates for $(\dc_L(F)/U(F)_{\ZZ})_{\{\sigma_i\}}$.
  By \eqref{gaction}, for either $g$ fixing $P=(0, \underline{w}_0, \tau_0)$, $\tau_0$ is an elliptic point of $\HH^+$,
  $(c \tau_0 + d)^{-1}=\xi$
  and
  $\underline{w}_0 = 0$, 
  where $\xi=e^{\pi i /3}$. 
  As in Proposition \ref{intsingextprop}, we define coordinates $\underline{y}$ on which $g$ acts linearly.
  From \eqref{gaction}, 
  \[
  \textbf{d}_0(g) =
  \bpm
  * & 0 & * \\
  0 & \frac{-c}{c \tau+d} I & 0 \\
  0 & 0 & * 
  \epm,
  \]
  and we obtain coordinates $\underline{y}=(y_i)$  where $y_2 = w_1$ and $y_3 = w_2$.
  In particular (by selecting an appropriate basis of $B$), we can assume that  $y_2=0$ is a local equation for the closure of $\dc_L(u)$ in a neighbourhood of $P$.
  By applying a further transformation of the form
  \[
  \bpm
  * & 0 & 0 \\
  0 & I & 0 \\
  * & 0 & *
  \epm,
  \]
  we obtain coordinates on which $g$ acts as $\opn{diag}(*, \xi, \xi, \xi^2)$.
  As in Proposition \ref{intsingextprop}, the result follows from Lemma \ref{toriclemma}.

\end{proof}
 \section{Low-weight cusp forms}\label{lwcfsec}
We now prove the existence of low-weight cusp forms for $\Gamma$ by studying dimension formulae for vector-valued modular forms, following a suggestion of the anonymous referee.
In this section, we will assume that $L=2U \op \la -2(n+1) \ra \op \la -2d \ra$ where both $n,d>0$.
Our results should be compared with the work of Ma on $\opn{O}(D(L))$-invariant vector-valued modular forms \cite{MaGauss}.
\subsection{Vector-valued modular forms}
The \emph{metaplectic group} $\opn{Mp}(2, \ZZ)$ is a double cover of $\opn{SL}(2, \ZZ)$ whose elements are given by pairs of the form $(M, \phi(\tau))$, where
\[
M = \bpm a & b \\ c & d \epm \in \opn{SL}(2, \ZZ)
\]
and $\phi(\tau)$ is a holomorphic function on the upper-half plane $\HH^+$ satisfying $\phi(\tau)^2 = c \tau + d$ \cite{autograss}.
The metaplectic group is generated by the elements 
\begin{center}
  \begin{tabular}{ccc}
  $T = \left ( \bpm 1 & 1 \\ 0 & 1 \epm, 1 \right )$ & \text{and} &
  $S = \left ( \bpm 0 & -1 \\ 1 & 0 \epm, \sqrt{\tau} \right )$
  \end{tabular}
\end{center}
 and the centre of $\opn{Mp}(2, \ZZ)$ is generated by  
\[
Z:=S^2 = \left ( \bpm -1 & 0 \\ 0 & -1 \epm, \sqrt{-1} \right ).
\]
\begin{defn}{\hspace{1sp}\cite{autograss}}
Suppose $\rho:\opn{Mp}(2, \ZZ) \rightarrow \opn{GL}(n, \CC)$ is a representation of $\opn{Mp}(2, \ZZ)$ and let $l \in \frac{1}{2} \ZZ$.
A function $f: \HH^+ \rightarrow \CC^n$ that is holomorphic on $\HH^+$ and at cusps is said to be a \textbf{vector-valued modular form} of type $\rho$ and weight $l$ if it satisfies 
\[
f(M \tau) = \rho(M, \phi)\phi(\tau)^{2l} f(\tau)
\]
for  all $(M, \phi) \in \opn{Mp}(2, \ZZ)$.
We denote the space of all such forms by $M_l(\rho)$ and the subspace of cusp forms by $S_l(\rho)$.
\end{defn}
The dimension of $M_k(\rho)$ and $S_k(\rho)$ can often be calculated using Theorem \ref{vvdimform}.
\begin{thm}[{\hspace{1sp}\cite{SkoruppaThesis}\cite[p.23]{EholzerSkoruppa}}]\label{vvdimform}
  Let $\rho:\opn{Mp}(2, \ZZ) \rightarrow \opn{GL}(m, \ZZ)$ be a representation with finite image such that $\rho(\xi_4^2 I, \xi_4) = \xi_4^{-2k}$ for all fourth roots of unity $\xi_4$.
  If $k \in \frac{1}{2}\ZZ$, then 
  \begin{align}\label{vvdim}
  \opn{dim} S_{2-k}(\rho) - \opn{dim} M_k(\overline{\rho})  =
  &
  \frac{1-k}{12} m - \frac{1}{4} \opn{Re}(e^{\pi i k/2} \opn{tr} \overline{\rho}(S))  \\
   & - \frac{2}{3 \sqrt{3}} \opn{Re}(e^{\pi i (2k+1)/6} \opn{tr}(\overline{\rho}(ST))) \nonumber \\
      & - \frac{1}{2} \alpha(\overline{\rho}) + \sum_{j=1}^m \BB_1(\lambda_j) \nonumber
\end{align}
  where $\lambda_j \in \CC$ are the eigenvalues of $\overline{\rho}(T)$;
  $\alpha(\overline{\rho})$ is the number of $\lambda_j$ such that $e^{2 \pi i \lambda_j} =1$; and 
\[
\BB_1(x) =
\begin{cases}
  x' - 1/2 & \text{if $x \in x' + \ZZ$ for $0<x'<1$} \\
  0 & \text{if $x \in \ZZ$.}
  \end{cases}
\]
\end{thm}
\begin{defn}
The \textbf{Weil representation} $\rho_M$ is defined for a lattice $M$  of signature $(b_+, b_-)$ by
\begin{align}
  \rho_A(T)(\bold{e}_x) & := e((x,x)/2)\bold{e}_x \label{WeilTdef}\\
  \rho_A(S)(\bold{e}_x) & := \frac{\xi_8^{b_- - b_+}}{\vert D(M) \vert^{1/2}} \sum_{y \in D(M)} e(-(x,y)) \bold{e}_y, \label{WeilSdef}
\end{align}
where $\bold{e}_x \in \CC[D(M)]$ is the basis element of the group algebra $\CC[D(M)]$ corresponding to $x \in D(M)$ and $e(y):=e^{2 \pi i y}$ for $y \in \CC$.
\end{defn}
If $M$ is a lattice of signature $(2,m)$ with $m \geq 3$ and $2U \subset M$ then there exists an injective $\opn{O}(D(L))$-equivariant linear map (which,  as explained in \cite{Ma}, is given by the Borcherds-Gritsenko lift \cite{GritsenkoAbelianK3, autograss}) 
\begin{equation}\label{cusplift}
S_l(\rho_M) \rightarrow S_k(\widetilde{\opn{O}}^+(M)),
\end{equation}
where $k = l + m/2 - 1$ and $l \in \frac{1}{2} \ZZ$ satisfies $l \equiv m/2 \bmod{ \ZZ}$.

If $\Gamma^* := \la \Gamma, \widetilde{\opn{O}}^+(L) \ra$ then $\widetilde{\opn{O}}^+(L) \trianglelefteq \Gamma^*$ and, by Proposition \ref{modgp}, $G:=\Gamma^*/\widetilde{\opn{O}}^+(L) \cong C_2$ is generated by the element $\sigma_{\underline{w}}$.
If $H = \la \rho(-I, i), \sigma_{\underline{w}} \ra$ then the restriction of $\rho_L$ to the subspace $S_2(\rho_L)^H \subset S_2(\rho_L)^G$  satisfies the conditions of Theorem \ref{vvdimform}, which we use in Theorem \ref{cuspthm} and Corollary \ref{weight3cuspformcor} to obtain conditions for $S_2(\rho_L)^H \neq 0$ and so, by  \eqref{cusplift}, for $S_3(\Gamma) \neq 0$. 
\subsection{The representation $\rho_L^{inv}$}
Let $\rho_L^{inv}$ be the restriction of $\rho_L$ to $\CC[D(L)]^H$.
\begin{deflem}\label{Hinvsdeflem}
    If $\CC[D(L)] \cong C_{2(n+1)} \op C_{2d}$ let 
    \[
    \ic := \{(a,b) \in \ZZ^2 \mid 0 < a < n+1 \text{ and } 0 \leq b \leq d \}
    \] 
    and
    \[
    \nu_{(a,b)} =
    \begin{cases}
      \bold{e}_{(a,b)}
      - \bold{e}_{(-a,-b)}
      + \bold{e}_{(a,-b)}
      - \bold{e}_{(-a,b)} & \text{if $0<a<n+1$ and $0<b<d$ } \\
      \bold{e}_{(a,b)} - \bold{e}_{(-a, -b)} & \text{if $b = 0$ or $d$}.
    \end{cases}
    \]
    Then $\{\nu_{(a,b)} \mid (a,b) \in \ic \}$ is a basis for $\CC[D(L)]^H$.
  \end{deflem}
  \begin{proof}
    Immediate from definition.
    \end{proof}
\begin{lemma}\label{rhoinvTlemma}
    $\rho_L^{inv}(T) \nu_{(a,b)} = e((a,b)^2/2) \nu_{(a,b)}$.
  \end{lemma}
  \begin{proof}
    Immediate from \eqref{WeilTdef} as $(a,b)^2 = (\pm a, \pm b)^2$.
  \end{proof}
  \begin{lemma}\label{rhoinvSSTlemma}
    On the basis of Definition/Lemma \ref{Hinvsdeflem},
   if $(a,b) \in \ic$ then the $\nu_{(a,b)}$-th coordinate of $\rho_L^{inv}(S)\nu_{(a,b)}$ is given by
    \[
    \lambda \times 
    \begin{cases}
      \sum_{\delta_1, \delta_2 = \pm 1}
      \delta_1 e \left ( \frac{\delta_1 a^2}{2(n+1)} + \frac{\delta_2 b^2}{2d} \right )
      & \text{if $0<b<d$} \\
      e \left ( \frac{a^2}{2(n+1)} + \frac{b^2}{2d} \right )
      - e \left ( \frac{-a^2}{2(n+1)} + \frac{b^2}{2d} \right ) 
      & \text{if $b = 0$ or $d$;}
 \end{cases}
    \]
    and the $\nu_{(a,b)}$-th coordinate of $\rho_L^{inv}(ST) \nu_{(a,b)}$ is given by
    \[
    \lambda \times 
    \begin{cases}
      \sum_{\delta_1, \delta_2 = \pm 1}
      \delta_1 e \left ( \frac{(2\delta_1 - 1) a^2}{4(n+1)} + \frac{(2\delta_2 - 1) b^2}{4d} \right )
      & \text{if  $0<b<d$} \\
      e \left ( \frac{a^2}{4(n+1)} + \frac{b^2}{4d} \right )
      - e \left ( \frac{-3a^2}{4(n+1)} + \frac{b^2}{4d} \right )
      & \text{if $b = 0$ or $d$,}
 \end{cases}
 \]
    where $\lambda = \frac{i}{2\sqrt{d(n+1)}}$.
  \end{lemma}
  \begin{proof}
    Routine calculation using \eqref{WeilSdef} and Definition/Lemma \ref{Hinvsdeflem}.
  \end{proof}
\subsection{Bounds for traces}
We now bound $\opn{tr}\rho_L^{inv}(S)$ and $\opn{tr}\rho_L^{inv}(ST)$.
For the convenience of the reader, we begin by summarising some standard results on quadratic Gauss sums \cite{GaussJacobiSums}. 
For $a,b,c \in \ZZ$ and $b>0$, we let $\GG(a,b,c)$ denote the generalised quadratic Gauss sum 
\[
\GG(a,b,c):=\sum_{m=0}^{c} e \left ( \frac{am^2}{b} \right )
\]
and we use $\GG(a,b)$ to denote the classical quadratic Gauss sum $\GG(a,b,b-1)$.
We will often use the elementary property 
\begin{equation}\label{NonCoprimeGaussSum} 
\GG(a,b) = (a,b) \GG \left ( \frac{a}{(a,b)}, \frac{b}{(a,b)} \right )
\end{equation}
for $(a,b) \neq 1$ and $b>0$.
\begin{thm}[{\cite{Gauss, GaussJacobiSums}}]\label{GaussSumThm}
If $(a,b)=1$ then 
\[ 
\GG(a,b) =
\begin{cases}
  0 & \text{if $b \equiv 2 \bmod{ 4}$} \\
  \delta_b \sqrt{b} \left ( \frac{a}{b} \right ) & \text{if $b$ is odd} \\
  (1+i) \delta_a^{-1} \sqrt{b} \left ( \frac{b}{a} \right ) & \text{if $a$ is odd and $4 \vert b$,}
  \end{cases}
\] 
where
$\delta_x = 1$ if $x \equiv 1 \bmod{4}$, $\delta_x=i$ if $x \equiv 3 \bmod{4}$
and $( \frac{.}{n} )$ is the Jacobi symbol.
\end{thm}
\begin{thm}[{\emph{(Reciprocity for generalised quadratic Gauss sums} \cite[Theorem 1.2.2]{GaussJacobiSums})}]\label{reciprocitythm}
If $a,b \in \ZZ \backslash \{0 \}$ where $b>0$ and $ab$ is even, then
\[ \sum_{m=0}^{b-1} e^{\pi i \frac{am^2}{b}} = e^{\opn{sgn}(a) \pi i / 4} \left \vert \frac{b}{a} \right \vert^{1/2}  \sum_{m=0}^{\vert a \vert - 1} e^{- \pi i \frac{b m^2}{a}}.
\]
\end{thm}
\begin{lemma}\label{agausssumlemma}
  If $a$ and $d$ are positive integers then 
  \[
\GG(a, 4d, d-1)
  =
  \frac{1}{2}
  \left (
  1 - e \left (\frac{a d}{4} \right ) +   \xi_8 \GG(-d, a) \sqrt{\frac{2d}{a}}  \right ).
  \]
\end{lemma}
\begin{proof}
By Theorem \ref{reciprocitythm}, 
  \[
  \sum_{m=0}^{2d - 1} e^{\pi i (a m^2/2d)} = \xi_8  \sqrt{2d/a} \sum_{m=0}^{a-1} e^{-2\pi i dm^2/a}.
  \]
  As
  \begin{align*}
    \sum_{m=0}^{2d-1} e^{\pi i (am^2/2d)}
    = & \sum_{m=0}^{d-1} e^{\pi i (am^2/2d)} + \sum_{m=1}^{d} e^{\pi i (a(2d-m)^2 / 2d)} \\
    = & -1 + e^{a \pi i d / 2} + 2  \sum_{m=0}^{d-1} e^{\pi i (am^2/2d)}
  \end{align*}
  then
  \[
  \sum_{m=0}^{d-1} e \left ( \frac{am^2}{4d} \right ) = \frac{1}{2}
  \left (
  1 - e \left (\frac{a d}{4} \right ) + \xi_8 \sqrt{2d/a} \sum_{m=0}^{a-1} e \left (\frac{-d m^2}{a} \right )
  \right )
  \]
  and the result follows.
\end{proof}
\begin{lemma}\label{trrhoSlemma}
  If both $d, n>0$ then $\opn{Re} \opn{tr} \rho_L^{inv}(S) = 0$. \end{lemma}
\begin{proof}
  By Lemma \ref{rhoinvSSTlemma},
  \begin{align*}
    \lambda^{-1} \opn{tr}\rho_L^{inv}(S) =
    &     \left(  \sum_{a=1}^n e \left (\frac{a^2}{2(n+1)} \right ) - e \left (\frac{-a^2}{2(n+1)} \right ) \right ) \biggl \{  2 + 2e(d/2)  \\
     & +\sum_{b=1}^{d-1} e \left ( \frac{b^2}{2d} \right ) + e \left ( \frac{-b^2}{2d} \right ) \biggr \},
  \end{align*}
  which is the product of a purely imaginary term and a purely real term.
\end{proof}
\begin{lemma}\label{trrhoSTlemma}
  We have
  \begin{align*}
  \vert \opn{tr} \rho_L^{inv}(ST) \vert  
  \leq \frac{1}{2} & \left ( 2/\sqrt{d(n+1)} + 1/\sqrt{2d} + \sqrt{3/2d} \right ) \\
  &  \times \left ( 4/\sqrt{d(n+1)} + 1/\sqrt{2(n+1)} + \sqrt{3/2(n+1)} \right ).
\end{align*}
\end{lemma}
\begin{proof}
 By Lemma \ref{rhoinvSSTlemma} we have,
 \begin{align*}
   \opn{tr} \rho_L^{inv}(ST)
   = &
   \frac{i}{2 \sqrt{d(n+1)}}
   \Biggl \{
   \left (
   \sum_{\delta_1, \delta_2 = \pm 1}
   \sum_{a=1}^n \sum_{b=1}^{d-1}
   \delta_1 e \left ( \frac{(2\delta_1 - 1) a^2}{4(n+1)} + \frac{(2\delta_2 - 1) b^2}{4d} \right )
   \right ) \\
   & + ( 1 + e(d/4))
   \sum_{a=1}^{n}
   e \left ( \frac{a^2}{4(n+1)}  \right )
   - e \left ( \frac{-3a^2}{4(n+1)}  \right )
   \Biggr \} \\
   = &
   \frac{i}{2 \sqrt{d(n+1)}}
   \Biggl \{
   \sum_{a=1}^n
   \left (
   e \left ( \frac{a^2}{4(n+1)} \right )
   -
   e \left ( \frac{-3a^2}{4(n+1)} \right )
  \right )
  \Biggr \} \\
  & \times
  \Biggl \{
  ( 1 + e(d/4)) + 
  \sum_{b=1}^{d-1}
  \left (
  e \left ( \frac{b^2}{4d} \right ) +  e \left ( \frac{-3b^2}{4d} \right ) \right ) 
  \Biggr \} \\
  = &  
   \frac{i}{2 \sqrt{d(n+1)}}
    \Biggl \{
   \sum_{a=1}^n
   \left (
   e \left ( \frac{a^2}{4(n+1)} \right )
   -
   e \left ( \frac{-3a^2}{4(n+1)} \right )
  \right )
  \Biggr \} \\
  & \times 
  \Biggl \{
  \GG(1, 4d, d-1) + \GG(-3, 4d, d-1) + e(d/4) - 1
  \Biggr \}.
 \end{align*}
 By Theorem \ref{GaussSumThm} and Lemma \ref{agausssumlemma},
 \begin{center}
   \begin{tabular}{c c c}
     $\vert \GG(1, 4d, d-1) - 1 \vert \leq 1 + \sqrt{d/2} $
     & and &
     $\vert \GG(3, 4d, d-1) -1 \vert \leq 1 + \sqrt{3d/2}$.
     \end{tabular}
   \end{center}
 Therefore, 
 \begin{align*}
   \vert \opn{tr} \rho_L^{inv}(ST) \vert
   \leq &
   \frac{1}{2 \sqrt{d(n+1)}}
   \vert \GG(1, 4(n+1), n) - \GG(-3, 4(n+1), n) \vert \\
   & \times      \vert \GG(1, 4d, d-1) + \GG(-3, 4d, d-1) + e(d/4) - 1 \vert \\
\leq &
     \frac{1}{2 \sqrt{d(n+1)}} \left (2 + \sqrt{(n+1)/2} + \sqrt{3(n+1)/2} \right ) \left (4 + \sqrt{d/2} + \sqrt{3d/2} \right ), \end{align*}
 from which the result follows.
\end{proof}
\begin{lemma}\label{aisolemma}
  In \eqref{vvdim}, we have 
  $\vert \alpha(\rho_L^{inv}) \vert \leq  ( 1+ 2^{\nu(d)}  ) ( n+2  )$.  
\end{lemma}
\begin{proof}
  By Lemma \ref{rhoinvTlemma}, it suffices to count solutions to
  \begin{equation}\label{isoeq}
    \frac{a^2}{4(n+1)} + \frac{b^2}{4d} \equiv 0 \bmod{ \ZZ}
  \end{equation}
  for $(a,b) \in \ic$.
  If \eqref{isoeq} is satisfied for fixed $a$ then $b^2 \equiv -a^2 d(n+1)^{-1} \bmod{d}$, which has  at most $1+2^{\nu(d)}$ solutions.
  The bound then follows by noting there are at most $n+2$ choices for $a$.
\end{proof}
\subsection{Bounds for $\sum_{j=1}^m \BB_1(\lambda_j)$}
We now obtain bounds for the sum $\sum_{j=1}^m \BB_1(\lambda_j)$ in \eqref{vvdim}, in terms of 
\[
\BB(\mu, d) := \sum_{x=0}^{d} \BB_1 \left ( \mu - \frac{x^2}{4d} \right ),
\]
where $\mu \in \QQ$ and $d \in \ZZ_{>0}$.
\begin{lemma}\label{Budlemma}
  If 
  $\mu = \mu_1/\mu_2$ 
  where 
  $(\mu_1, \mu_2)=1$ 
  and 
  $\mu_2>0$, 
  let 
  $N:= \opn{lcm}(\mu_2, 4d)$
  and
  $k = N/4d$.
  Then, 
  \[
  \vert \BB(\mu, d) \vert
  \leq
  \vert \BB_1(\mu - d/4) \vert
  + \frac{1}{8 \pi} \left (
  9 + 4k + 8 \opn{log}(N) + 2^{4+\nu(4d)}\sqrt{d} ( 2k(1+ \opn{log}(4d))+1)
  \right ).
\]
\end{lemma}
\begin{proof}
  We closely follow \cite[Lemma 5]{BruinierPicard}. 
As the function $\BB_1$ is $1$-periodic with Fourier expansion
  \[
  \BB_1(x) = -\frac{1}{2 \pi i} \sum_{\substack{n \in \ZZ \\ n \neq 0 }} \frac{e(nx)}{n} 
  \]
  then 
  \begin{align*}
    \BB(\mu, d)
    = & - \frac{1}{2 \pi i} \sum_{x=0}^{d} \sum_{\substack{n \in \ZZ \\ n \neq 0 }} \frac{1}{n} e \left ( n \mu - \frac{nx^2}{4d} \right ) \\
    =& \BB_1 \left (\mu - \frac{d}{4} \right ) - \frac{1}{2 \pi i} \sum_{x=0}^{d-1} \sum_{\substack{n \in \ZZ \\ n \neq 0 }} \frac{1}{n} e \left ( n \mu - \frac{nx^2}{4d} \right ), \\
    \intertext{and under the change of index $n \mapsto -n$,}
    = & \BB_1\left ( \mu - \frac{d}{4} \right ) + \frac{1}{2 \pi i}  \sum_{\substack{n \in \ZZ \\ n \neq 0 }} \frac{1}{n} e(-n \mu) \GG(n, 4d, d-1). \\
 \intertext{As $\BB_1(\mu, d) \in \RR$, then}
    \BB(\mu, d)
  = & \BB_1\left ( \mu - \frac{d}{4} \right ) + \frac{1}{2 \pi} \sum_{\substack{n \in \ZZ \\ n \neq 0}} \frac{1}{n}
  \opn{Im} \left ( e(-n \mu) \GG(n, 4d, d-1) \right ), \\
  \intertext{and as $\GG(-n, 4d, d-1) = \overline{\GG(n, 4d, d-1)}$ then}
  \BB(\mu, d)
  = & \BB_1\left ( \mu - \frac{d}{4} \right ) + \frac{1}{\pi} \sum_{n=1}^{\infty} \frac{1}{n} \opn{Im} \left ( e(-n \mu) \GG(n, 4d, d-1) \right ).
  \end{align*}
  By definition, $e((n+N)\mu) \GG((n+N), 4d, d-1) = e(n \mu) \GG(n, 4d, d-1)$ and $\opn{Im}(e(N \mu) \GG(N, 4d, d-1)) = 0$.
  Therefore, 
  \begin{align*}
    \BB(\mu, 2d) - \BB_1\left ( \mu - \frac{d}{4} \right ) 
    = & \frac{1}{\pi } \sum_{m=0}^{\infty} \sum_{j=1}^{N-1} \frac{1}{Nm + j} \opn{Im} \left ( e(-j \mu) \GG(j, 4d, d-1) \right ) \\
    = & \frac{1}{2 \pi} \sum_{m=0}^{\infty} \sum_{j=1}^{N-1} \biggl (
    \frac{1}{Nm + j} \opn{Im} \left ( e(-j \mu) \GG(j, 4d, d-1) \right ) \\
    & +
    \frac{1}{N(m+1) - j} \opn{Im} \left ( e(j \mu) \GG(-j, 4d, d-1) \right ) \biggr ) \\
    = & \frac{1}{2 \pi} \sum_{m=0}^{\infty} \sum_{j=1}^{N-1}
    \left ( \frac{1}{Nm + j } - \frac{1}{N(m+1) - j} \right )
    \opn{Im} \left ( e(-j \mu) \GG(j, 4d, d-1) \right ).
    \end{align*}
We obtain the expression 
  \[
  \BB(\mu, d) =  \lambda_1 + \lambda_2 + \BB_1\left ( \mu - \frac{d}{4} \right ),
  \]
  where
  \begin{align*}
  \lambda_1 := &
    \frac{1}{2 \pi} \sum_{j=1}^{N-1} \left ( \frac{1}{j} - \frac{1}{N-j} \right)
    \opn{Im}( e(-j \mu) \GG(j, 4d, d-1)) \\
\lambda_2 := &
    \frac{1}{2 \pi} \sum_{m=1}^{\infty} \sum_{j=1}^{N-1}
    \left ( \frac{1}{Nm + j }  - \frac{1}{N(m+1) - j} \right )
    \opn{Im} (e(-j \mu) \GG(j, 4d, d-1)).
  \end{align*}
  We now bound $\lambda_1$ and $\lambda_2$.
  As $\GG(j, 4d, d-1) = \GG( N+j, 4d, d-1)$, then 
  \begin{align*}
    \vert \lambda_1 \vert \leq &
    \frac{1}{2 \pi} \sum_{j=1}^{N-1} \frac{1}{j} \vert \GG(j, 4d, d-1) \vert
    +
    \frac{1}{2 \pi} \sum_{j=1}^{N-1} \frac{1}{N-j} \vert \GG(j, 4d, d-1) \vert \\
    \leq & \frac{1}{\pi} \sum_{j=1}^{N-1} \frac{1}{j} \vert \GG(j, 4d, d-1) \vert.
  \end{align*}
  By Lemma \ref{agausssumlemma},
  \begin{align}
    \vert \GG(j, 4d, d-1) \vert
    \leq & 1 + \frac{1}{2}\vert \GG(-d, j)  \vert \left \vert \frac{2d}{j} \right \vert^{1/2} \nonumber \\
    \leq & 1 + \frac{1}{2} (d,j) \left \vert \GG \left ( \frac{-d}{(d,j)}, \frac{j}{(d,j)} \right ) \right \vert \left \vert \frac{2d}{j} \right \vert^{1/2} \nonumber \\
    \intertext{and by Theorem \ref{GaussSumThm},}
    \vert \GG(j, 4d, d-1) \vert \leq & 1 + \sqrt{d(d,j)}, \label{Gj4dBound}
    \end{align}
  (where the last two lines follow from \eqref{NonCoprimeGaussSum}).
  Therefore,
  \[
    \vert \lambda_1 \vert
    \leq  \frac{1}{\pi} \sum_{j=1}^{N-1} \frac{1}{j} + \frac{\sqrt{d}}{\pi} \sum_{j=1}^{N-1} \frac{ (j, 4d)^{1/2}}{j}, 
    \]
    and by noting that  $\sum_{n=1}^N \frac{1}{n} \leq 1 + \opn{log}(N)$, we obtain
    \[
    \vert \lambda_1 \vert
    \leq
    \frac{1 + \opn{log}(N)}{\pi} +
    \frac{k-1}{2 \pi} + 
    \frac{k \sqrt{d}}{\pi} \sum_{j=1}^{4d - 1} \frac{(4d, j)^{1/2}}{j}.
    \]
    By elementary considerations, each element of $\{1, \ldots, 4d-1 \}$ can be expressed uniquely as $xy$ where $x \vert 4d$, $x \neq 4d$ and $1 \leq y \leq 4d/x$ with $(y, 4d/x)=1$.
    Therefore, 
    \begin{align*}
      \vert \lambda_1 \vert \leq &
      \frac{1 + \opn{log}(N)}{\pi}
      + \frac{k-1}{2 \pi}
      + \frac{k \sqrt{d}}{\pi}
      \sum_{\substack{x \vert 4d \\ x \neq 4d}}
      \sum_{\substack{y = 1 \\ (y, 4d/x) = 1}}^{4d/x}
      \frac{\sqrt{x}}{xy} \\
      \leq &
      \frac{1 + \opn{log}(N)}{\pi}
      + \frac{k-1}{2 \pi}
      + \frac{k \sqrt{d}}{\pi}
      \sum_{\substack{x \vert 4d \\ x \neq 4d}} ( 1 + \opn{log}(4d/x) ) x^{-1/2} 
       \\
      \leq &
      \frac{1 + \opn{log}(N)}{\pi}
      + \frac{k-1}{2 \pi}
      + \frac{k \sqrt{d}}{\pi} ( 1 + \opn{log}(4d)) \sigma_{-1/2}(4d),  
    \end{align*}
    where $\sigma_r(d)$ is the divisor function
    \[
    \sigma_r(n) = \sum_{d \vert n} d^r.
    \]
By considering values at prime powers, we have $\sigma_{-1/2}(4d) \leq 2^{\nu(4d) + 2}$.
    Therefore,
    \[
    \vert \lambda_1 \vert
    \leq
    \frac{1 + \opn{log}(N)}{\pi}
    + \frac{k-1}{2 \pi}
    + \frac{4k \sqrt{d}}{\pi } (1 + \opn{log}(4d)) 2^{\nu(4d)}.
    \]
    We now bound $\lambda_2$.
    By definition,
    \begin{align*}
      \vert \lambda_2 \vert
      \leq &
      \frac{1}{2 \pi} \sum_{m=1}^{\infty} \sum_{j=1}^{N-1}
      \left ( \frac{N-2j}{(Nm+j)(N(m+1) - j)} \right )
      \vert \GG(j, 4d, d-1) \vert. \\
      \intertext{As $\vert N-2j \vert < N$ and $N^2m(m+1) + Nj - j^2 \geq N^2m(m+1)$, then}
      \vert \lambda_2 \vert \leq &
      \frac{1}{2 \pi N} \sum_{m=1}^{\infty} \sum_{j=1}^{N-1} \frac{\vert \GG(j, 4d, d-1) \vert }{m(m+1)}, \\
      \intertext{and as $1 = \sum_{m=1}^{\infty} \frac{1}{m(m+1)}$ then,}
      \vert \lambda_2 \vert
      \leq &
      \frac{1}{2 \pi N} \sum_{j=1}^{N-1} \vert \GG(j , 4d, d-1) \vert. \\
      \intertext{By \eqref{Gj4dBound},}
      \vert \lambda_2 \vert \leq &
      \frac{N-1}{2 \pi N} + \frac{(k-1)d}{2 \pi N}
      + \frac{d}{2 \pi N} \sum_{j=1}^{4d-1} (d,j)^{1/2} \\
      \leq & \frac{1}{2 \pi N} + \frac{\sqrt{d}}{2 \pi N} \left ( (k-1) \sqrt{d} + k \sum_{j=1}^{4d-1}(d,j)^{1/2} \right ). \\
      \intertext{As for $\lambda_1$, }
      \vert \lambda_2 \vert
      \leq & \frac{1}{2 \pi N} + \frac{\sqrt{d}}{2 \pi N} \left ( (k-1)\sqrt{d} + k \sum_{\substack{x \vert 4d \\ x \neq 4d}} \sum_{\substack{y = 1 \\ (y,4d/x) = 1}}^{4d/x} \sqrt{x} \right ) \\
      \leq & \frac{1}{2 \pi N} + \frac{\sqrt{d}}{2 \pi N} \left ( (k-1)\sqrt{d} + 4dk + \sum_{\substack{x \vert 4d \\ x \neq 4d}} \frac{1}{\sqrt{x}} \right ) \\
      \leq & \frac{1}{2 \pi N} + \frac{\sqrt{d}}{2 \pi N} \left ( (k-1)\sqrt{d} + 4dk \sigma_{-1/2}(4d) \right ) \\
      \leq & \frac{1}{8 \pi} \left ( 5 + 4d^{1/2}\sigma_{-1/2}(4d) \right ).
    \end{align*}
    Therefore,
    \[
    \vert \BB(\mu, d) \vert
    \leq
    \vert \BB_1(\mu - d/4) \vert
    + \frac{1}{8 \pi} \left (
    9 + 4k + 8 \opn{log}(N) + 2^{4+\nu(4d)}\sqrt{d} ( 2k(1+ \opn{log}(4d)) + 1))
    \right ).
    \]
\end{proof}
\subsection{Bounds for low-weight cusp forms}
We now show that low-weight cusp forms exist for all $n>0$ when $d \gg 0$, obtaining an explicit bound for $d$ when $n=2$.
\begin{thm}\label{cuspthm}
  If $n$ is fixed and $d \gg 0$ then $S_3(\Gamma) \neq 0$.
\end{thm}
\begin{proof}
  By Theorem \ref{vvdimform} and Definition/Lemma \ref{Hinvsdeflem},  with $\rho=\rho^{inv}_L$, $k = 0 $ and $m = n(d+1)$, we have 
  \begin{align*}
    \opn{dim} S_2(\rho_L^{inv})
    = & \frac{n(d+1)}{12} - \frac{1}{4} \opn{Re}( \opn{tr} \overline{\rho}^{inv}_L(S))
    - \frac{2}{3 \sqrt{3}} \opn{Re}( e^{\pi i /6} \opn{tr}(\overline{\rho}^{inv}_L(ST))) \\
    & - \frac{1}{2} a(\overline{\rho}^{inv}_L) + \sum_j \BB_1(\lambda_j) + \opn{dim}M_0(\overline{\rho}_L^{inv}).
  \end{align*}
  By Lemma \ref{trrhoSlemma} $\opn{Re}( \opn{tr} \overline{\rho}^{inv}_L(S)) = 0$,
  by Lemma \ref{trrhoSTlemma} $\opn{tr}(\overline{\rho}^{inv}_L(ST))) = O(\sqrt{d})$,
  by Lemmas \ref{aisolemma} and \ref{2nulemma} $a(\overline{\rho}^{inv}_L) = O(d^{\epsilon})$
  and by Lemmas \ref{Budlemma} and \ref{2nulemma} $\sum_j \BB_1(\lambda_j) = O(d^{1/2 + \epsilon}\opn{log}(d))$. 
Therefore, for $\epsilon<1/2$, we have
  \[
    \opn{dim} S_2(\rho_L^{inv})
    = \frac{n(d+1)}{12} +\opn{dim} M_0(\overline{\rho}_L^{inv}) +  O(d^{1/2 + \epsilon} \opn{log}(d))
    \]
    and the result follows from \eqref{cusplift}.
\end{proof}
\begin{cor}\label{weight3cuspformcor}
  Let $n=2$.
Then, for all $d>2209016321562407$ and prime $d>11943004$,  $S_3(\Gamma) \neq 0$. 
\end{cor}
\begin{proof}
  By Lemma \ref{Budlemma} and \ref{2nulemma},  
  \begin{align*}
    \vert \BB(\mu, d) \vert
    \leq 
    \frac{1}{8 \pi} \left ( (32)2^{\nu(d)}d^{1/2}(12 \opn{log}(2) + 6 \opn{log}(d) + 7) + 8\opn{log}(12) + 8\opn{log}(d) + 21 \right ) + \frac{1}{2}.
  \end{align*}
    If $d>5$ then, from Lemma \ref{trrhoSTlemma}, $\vert \opn{Re}(e^{\pi i / 12} \opn{tr}(\overline{\rho}(ST)) \vert \leq (3 \sqrt{3}) / 2$.
    In \eqref{vvdim},
    \[ 
    \sum_j \BB_1(\lambda_j) \leq \vert \BB(1/12, d) \vert + \vert \BB(1/3, d) \vert
    \] 
    and by Theorem \ref{vvdimform} and Lemma \ref{trrhoSlemma} and \ref{aisolemma} we have,
  \begin{align}\label{S2bound}
    \opn{dim}(S_2(\rho_L^{inv})) 
    \geq &
    - \frac{1}{4 \pi}(1200 d^{2/3}(12 \opn{log}(2) + 6\opn{log}(d) + 7) + 8\opn{log}(12) + 8\opn{log}(d) + 21)
  \nonumber \\
  &
  + \frac{d}{6}
  - 75 d^{1/6}
  - \frac{23}{6}.
  \end{align}
  We then locate the roots of \eqref{S2bound} using a computer (setting $2^{\nu(4d)}=2$ in the case of prime $d$). 
\end{proof}
 \section{Hirzebruch-Mumford volumes}\label{HMSec}
Let $L$ be an indefinite lattice of signature $(2,n)$ and let  $\Lambda \subset \opn{O}(L \otimes \QQ)$ be an  arithmetic subgroup.
 The Hirzebruch-Mumford proportionality principle \cite{HMvol} states that   
\[
\opn{dim} M_k(\Lambda, \chi) = \opn{Vol}_{HM}(\Lambda)k^m + O(k^{m-1}),
\]
where $\opn{Vol}_{HM}(\Lambda)$ is a constant known as the \emph{Hirzebruch-Mumford volume} of $\Lambda$.
In this section, we calculate the Hirzebruch-Mumford volumes of $\opn{O}(L)$ and $\opn{O}(K^{u,m}_{r,s}(k,l))$ in terms of local densities, following the approach of \cite{HMvolapps}. 
Throughout, we will use $\nu(x)$ to denote the number of prime divisors of $x \in \NN$ counted without multiplicity.
\begin{thm}[{\cite[Theorem 2.1]{HMvolapps}}]\label{GHSHMvol}
  If $L$ is an indefinite lattice of signature $(2, m)$ then 
  \[
  \opn{Vol}_{HM}(\opn{O}(L)) = \frac{2}{\vert \opn{Sg}(L) \vert} \vert \opn{det} L \vert^{(m+3)/2} \prod_{k=1}^{m+2} \pi^{-k/2} \Gamma(k/2) \prod_p \alpha_p(L)^{-1},
  \] 
  where $\alpha_p(L)$ is the local density of the lattice $L$ at the prime $p$ and $\vert \opn{Sg} (L) \vert$ denotes the number of spinor genera in $\opn{gen}(L)$.
\end{thm}
As noted in \cite{HMvolapps} (adopting the convention of \cite{Ma}), if $\Lambda \subset \opn{O}(L)$ is of finite index and $L$ contains a copy of the hyperbolic plane $U$ then,
\begin{equation}\label{subgpHMvol}
  \opn{Vol}_{HM}(\Lambda) =
  \vert \la \opn{O}(L), -I \ra :\la \Lambda, -I \ra \vert \opn{Vol}_{HM}(\opn{O}(L)).
  \end{equation}
\subsection{Local densities}
  We recall the calculation of the local densities $\alpha_p(L)$ for an indefinite lattice $L$ of signature $(2,m)$, following \cite{HMvolapps}. These results can also be found in \cite{Kitaoka}.

Suppose 
\begin{equation}\label{jordandecomp}
  L \otimes \ZZ_p =\bigoplus_{j \in \ZZ} L_j
\end{equation}
is the Jordan decomposition of the $\ZZ_p$-lattice $L \otimes \ZZ_p$ where, if $n_j:=\opn{rank}L_j$, $N_j:=p^{-j} L_j \in \opn{GL}(\ZZ_p, n_j)$ if $n_j>0$ and $N_j:=0$ otherwise.

\noindent {$\bf{p>2.}$}
If $p$ is an odd prime then $\alpha_p(L)$ is given by
\[
\alpha_p(L) = 2^{s-1} p^w P_p(L) E_p(L),
\]
where $s$ is the number of non-zero factors in \eqref{jordandecomp}, 
\begin{equation*}
  P_p(L) := \prod_j P_p([n_j/2]) 
\end{equation*}
where 
\[
P_p(n) := 
  \begin{cases}
    \prod_{i=1}^n (1 - p^{-2i}) & \text{if $n \in \NN \backslash \{0\}$} \\
    1 & \text{if $n=0$} \\
  \end{cases}
\]
and $[n_j/2]$ denotes the integer part of $n_j/2$.
The term $w$ is given by
\[
w = \sum_j j n_j \left ( \frac{(n_j + 1)}{2} + \sum_{k>j} n_k \right ) 
\]
and $E_p(L)$ is defined by
\begin{equation*}
  E_p(L) = \prod_{L_j \neq 0} (1 + \chi(N_j/pN_j)p^{-n_j/2})^{-1}
\end{equation*}
where, for a finite quadratic space $W$ over $\FF_p$,
\[
\chi(W) =
\begin{cases}
  0 & \text{if $\opn{dim}(W)$ is odd} \\
  1 & \text{if $W$ is a hyperbolic space} \\
  -1 & \text{otherwise.}
  \end{cases}
\]
\noindent {$\bf{p=2}$.}
The formula needs to be modified when $p=2$. 
Over $\ZZ_2$, the lattices $N_j$ decompose as the orthogonal direct sum
\[
N_j = N_j^{even} \op N_j^{odd}
\]
where $N_j^{even}$ is even, $N_j^{odd}$ is odd and $\opn{rank}N_j^{odd} \leq 2$.
We let $E_j := \frac{1}{2}(1 + \chi(N_j^{even})2^{- \opn{rank} N_j^{even}/2})$ if both $N_{j-1}$ and $N_{j+1}$ are even and $N_j^{odd} \not \cong \la \epsilon_1 \ra \op \la \epsilon_2 \ra$ where $\epsilon_1 \cong \epsilon_2 \bmod{ 4}$; otherwise, we let $E_j:=1/2$.
We define 
\[
  q_j := 
  \begin{cases}
    0 & \text{if $N_j$ is even} \\
    n_j & \text{if $N_j$ is odd and $N_{j+1}$ is even} \\
    n_j + 1 & \text{if $N_j$ and $N_{j+1}$ are both odd,} \\
  \end{cases}
  \]
  and set $q:=\sum_j q_j$. 
  Then, if   
\[
  P_2(L) := \prod_j P_2(\opn{rank} N_j^{even}/2)
\]
and 
\[
  E_2(L) := \prod_j E_j^{-1},
\]
the local density $\alpha_2(L)$ is given by
\begin{equation*}
  \alpha_2(L) = 2^{m+1+w-q} P_2(L) E_2(L).
\end{equation*}
The calculation of $E_j$ can often be simplified by noting that \cite{Kitaoka} 
\[
  E_j =
  \begin{cases}
    1 & \text{if $N_{j-1} = N_j = N_{j+1} = \{0\}$} \\
    1 & \text{if $N_{j-1}=N_j=\{0\}$ and $N_{j+1}$ is even} \\
    1 & \text{if $N_{j-1}$ is even and $N_j = N_{j+1} = \{0\}$} \\
    \frac{1}{2} & \text{if $N_{j-1} = N_j = \{0\}$ and $N_{j+1}$ is odd} \\
    \frac{1}{2} & \text{if $N_{j-1}$ is odd and $N_j=N_{j+1} = \{0\}$}.
  \end{cases}
\]
\subsection{The Hirzebruch-Mumford volume of $\opn{O}(L)$}
\begin{proposition}\label{HMvolOL}
  If $L = 2U \op \la -6 \ra \op \la -2d \ra$, then
  \[
    \opn{Vol}_{HM}(\opn{O}(L))
    =  \frac{(12d)^{5/2}\pi^{3}}{137781} L(3, \left( \frac{-12d}{*} \right) )^{-1} 2^{-\nu(12d)} 
    \vert 12d \vert^{-1}_2 \vert 12 d \vert^{-1}_3 \alpha_2(L)^{-1} \alpha_3(L)^{-1} \prod_{p \vert 12d} (1-p^{-6}),
  \]
  where expressions for $\alpha_2(L)$ and $\alpha_3(L)$ are given in \eqref{alpha_2L} and \eqref{alpha_3L}.
\end{proposition}
\begin{proof}
  For prime $p$, let $L_p:=L \otimes \ZZ_p$ and define $a$ by $\vert 12 d \vert_p = p^{-a}$.
  
  \noindent $\mathbf{p=2.}$
  If $p=2$ then $D(L_2)=C_2 \op C_{2^{a-1}}$, $q_2(x,y) = -3x^2/2 - dy^2 / 2^{2a-3} \bmod{ 2 \ZZ_2}$ for $(x,y) \in D(L_2)$ and
  \[
  L_2 = 2U \op \la \theta_1 \ra (2) \op \la \theta_{a-1} \ra (2^{a-1}), 
  \]
  where $\theta_1, \theta_{a-1} \in \ZZ_p^*$.
  \begin{enumerate}[(i)]
  \item  If $a=2$ then $s=2$, $q=2$, $w=3$ and $P_2(L) = P_2( \opn{rank} N_0^{even}/2) P_2( \opn{rank} N_1^{even}/2)$.
  As $\opn{norm} N_1 = \ZZ_2$ then $P_2(\opn{rank} N_1^{even}/2) = P_2(0)$ and $P_2( \opn{rank} N_0^{even}/2) = P_2(2)$.
  Moreover, $E_i=1$ when $i<0$ or $i>2$,
  $E_0 = 1/2$, 
  $E_1 = 1$ if $d \not \equiv 3 \bmod{ 4}$,
  $E_1 = 1/2$ if $d \equiv 3 \bmod{ 4}$
  and $E_2=1/2$. 
  Therefore,
  \[
    E_2(L) = 
    \begin{cases}
      4 & \text{if $d \not \equiv 3 \bmod{ 4}$} \\
      8 & \text{if $d \equiv 3 \bmod{ 4}$}.
    \end{cases}
  \]
\item
  If $a>2$ then
  $w=a+1$, 
  $q=3$ if $a=3$, 
  $q=2$ if $a \geq 4$   
  and
  $P_2(L) = P_2(2) P_2(0)^2$.
  \item If $a=3$ then $E_{i} = 1$ if $i<-1$ or $i>3$,
  $E_{-1} = 1$,
  $E_0 = 1/2$,
  $E_1 = 1/2$,
  $E_2 = 1/2$,
  $E_3 = 1/2$ 
  and $E_2(L) = 16$.
  \item If $a=4$ then $E_{i} = 1$ if $i<-1$ or $i>4$,
  $E_{-1} = 1$, 
  $E_0 = 1/2$, 
  $E_1 = 1$, 
  $E_2 = 1/2$, 
  $E_3 = 1$,
  $E_4 = 1/2$  
  and $E_2(L) = 8$.
  \item If $a=5$ then $E_{i} = 1$ if $i<-1$ or $i>5$,
  $E_{-1} = 1$, 
  $E_0 = 1/2$, 
  $E_1 = 1$, 
  $E_2 = 1/2$, 
  $E_3 = 1/2$, 
  $E_4 = 1$, 
  $E_5 = 1/2$  
  and $E_2(L) = 16$.
\item
  If $a \geq 6$ then $E_2(L) = E_{-1}^{-1} E_0^{-1} E_1^{-1} E_2^{-1} E_{a-2}^{-1} E_{a-1}^{-1} E_a^{-1}$ and
  $E_{-1} = 1$, 
  $E_0 = 1/2$, 
  $E_1 = 1$, 
  $E_2 = 1/2$, 
  $E_{a-2} = 1/2$,  
  $E_{a-1} = 1 $, 
  $E_a = 1/2$  
  and $E_2(L) = 16$.
\end{enumerate}
  Therefore,
  \begin{equation}\label{alpha_2L}
    \alpha_2(L) =
    \begin{cases}
      256 P_2(2) & \text{if $a=2$ and $d \not \equiv 3 \bmod{ 4}$} \\
      512 P_2(2) & \text{if $a=2$ and $d \equiv 3 \bmod{ 4}$} \\
      2^{7+a} P_2(2) & \text{if $a=3$ or $4$} \\
      2^{8+a} P_2(2) & \text{if $a \geq 5$}.
    \end{cases}
  \end{equation}
  
  \noindent $\mathbf{p=3.}$
  Let $p=3$.
  If $a=1$ then 
  $D(L_3)=C_3$ and $q_3(x) = -2x^2/3 \bmod{\ZZ_3}$ for $x \in D(L_3)$;
  if $a=2$ then $D(L_3) = C_3 \op C_3$ and $q_3(x,y) = -2x^2/3   -2y^2d/3^2 \bmod{ 2 \ZZ_3}$ for $(x,y) \in D(L_3)$;  
  if $a>2$ then $D(L_3)=C_3 \op C_{3^{a-1}}$ and $q_3(x,y) = -2x^2/3  -2y^2 d/3^{2a-2} \bmod{\ZZ_3}$ for $(x,y) \in D(L_3)$.
  By Theorem 1.9.1 and  Corollary 1.9.3 of \cite{Nikulin},
  \begin{equation*}
    L_3 =
    \begin{cases}
      \la 1 \ra^{\op 4} \op \la \theta_0 \ra \op \la \theta_1 \ra (3) & \text{if $a=1$} \\
      \la 1 \ra^{\op 3} \op \la \theta_0 \ra \op \la \theta_1 \ra (3) \op \la \theta_1' \ra (3) & \text{if $a=2$} \\
      \la 1 \ra^{\op 3} \op \la \theta_0 \ra \op \la \theta_1 \ra (3) \op  \la \theta_{a-1} \ra (3^{a-1}) & \text{if $a>2$}
    \end{cases}
  \end{equation*}
  for $\theta_1$, $\theta_1'$, $\theta_a \in \ZZ_p^*$.
\begin{enumerate}[(i)]
\item  If $a=1$ then $s=2$, $w=1$, $P_3(L) = P_3(2)$ and $E_3(L)=1$ as $\chi(N_0)=\chi(N_1)=0$.
  \item  If $a=2$ then $s=2$, $w=3$ and $P_3(L) = P_3(2) P_3(1)$.
    As $N_0$ is obtained from $2U$ then $N_0$ is hyperbolic and $\chi(N_0)=1$. 
    As 
\[
    \chi(N_1)= 
    \left ( - \frac{4.3^{-1}.d}{3} \right )
  \] 
  then $E_3(L) = (1+3^{-2})^{-1} (1 + ( \frac{-4.3^{-1}d}{3} )3^{-1} )^{-1}$.
\item If $a>2$ then $s=3$, $w=a+1$ and $P_3(L) = P_3(2)$.
  As $N_0$ is hyperbolic, then $\chi(N_0)=1$ and $E_3(L) = (1 + 3^{-2})^{-1}$.
\end{enumerate}
Therefore, 
\begin{equation}\label{alpha_3L}
    \alpha_3(L) =
    \begin{cases}
      6 P_3(2) & \text{if $a=1$}\\
      54 P_3(2) P_3(1) (1 + 3^{-2})^{-1} (1 + ( \frac{-4.3^{-1} d }{3} ) 3^{-1})^{-1} & \text{if $a=2$} \\
      2^2 3^{a+1} P_3(2) (1+3^{-2})^{-1} & \text{if $a>2$.} 
    \end{cases}
  \end{equation}
\noindent $\mathbf{p>3.}$ If $p>3$ then, by Proposition 1.7.1 of \cite{Nikulin}, $D(L_p) = C_{p^a}$ with discriminant form $q_p(x)=-2dp^{-2a}x^2$ $\bmod{ \ZZ_p}$ for $x \in D(L_p)$.
  By Theorem 1.9.1 and Corollary 1.9.3 of \cite{Nikulin},
  \begin{equation*}
    L_p =
    \begin{cases}
      \la 1 \ra^{\op 5} \op \la \theta_0 \ra & \text{if $a=0$}\\
      \la 1 \ra^{\op 4} \op \la \theta_0 \ra \op \la \theta_a \ra (p^a) & \text{if $a>0$} \\
    \end{cases}
  \end{equation*}
  where $\theta_0$, $\theta_a \in \ZZ_p^*$.
  If $a=0$ then $s=1$, $P_p(L)=P_p(3)$, $w=0$ and 
  $E_p(L) = (1 + \chi(N_0)p^{-3})^{-1}$.
  By the classification of non-degenerate quadratic forms over $\FF_p$ (e.g. \cite[p.63]{Dieudonne}), $L_p / pL_p$ is hyperbolic if and only if
  \begin{equation*}
    \left ( \frac{\theta_0}{p} \right ) =  \left (\frac{-12d}{p} \right ) = 1,
  \end{equation*}
  where $\left ( \frac{*}{p} \right)$ denotes the Legendre symbol, 
  implying 
  \begin{equation*}
    \chi(N_0) =  \left ( \frac{-12d}{p} \right ).
  \end{equation*}
  If $a>0$ then $s=2$, $P_p(L)=P_p(2)$, $w=a$ and $\chi(N_0) = \chi(N_a) = 0$.
  Therefore, 
  \[
  E_p(L) =
  \begin{cases}
    (1 + \left ( \frac{- 12d}{p} \right) p^{-3})^{-1} & \text{if $a=0$} \\
    1 & \text{if $a>0$}
  \end{cases}
  \]
  and
\[ 
  \alpha_p(L) =
  \begin{cases}
    P_p(3)(1 + \left ( \frac{- 12d}{p} \right ) p^{-3} )^{-1} & \text{if $a = 0$ } \\
    2 p^a P_p(2) & \text{if $a>0$.}
  \end{cases}
\] 
We now calculate $\prod \alpha_p^{-1}$.
  For prime $p$, let
\begin{center}
\begin{tabular}{c c c }
    $\beta_p(L):=P_p(3)(1 + \left ( \frac{-12d}{p} \right )p^{-3} )^{-1}$ & and &
    $\gamma_p(L) := 2p^a P_p(2)$.
\end{tabular}
\end{center}
    Then,
  \begin{equation*}
    \prod_p \alpha_p^{-1}  =  \frac{\gamma_2(L) \gamma_3(L)}{\alpha_2(L) \alpha_3(L)} \left ( \prod_p \beta_p(L)^{-1} \right ) \left ( \prod_{p \mid 12d} \frac{\beta_p(L)}{\gamma_p(L)} \right )
  \end{equation*}
  and
  \begin{align*}
    \prod_p \beta_p(L)^{-1} & =
    \prod_p (1-p^{-2})^{-1} \prod_p (1 - p^{-4})^{-1} \prod_p (1-p^{-6})^{-1} \prod_p \left( 1 + \left ( \frac{-12d}{p} \right )p^{-3} \right) \\
& = \zeta(2) \zeta(4) \zeta(6) L(3, \left (\frac{-12d}{*} \right ))^{-1}.\\
\intertext{Using the familiar values $\zeta(2) = \pi^2/6$, $\zeta(4) = \pi^4 / 90$ and $\zeta(6) = \pi^6/945$, we have }
    \prod_p \beta_p(L)^{-1} & = \frac{\pi^{12}}{510300} L(3, \left ( \frac{-12d}{*} \right ) )^{-1},
  \end{align*}
  where $L(s, \left( \frac{D}{*} \right) )$ is the Dirichlet series defined by $p \mapsto \left( \frac{D}{p} \right )$. 
  Similarly,
  \begin{align*}
    \prod_{p \mid 12d} \frac{\beta_p(L)}{\gamma_p(L)}
    & = \prod_{p \mid 12d} \frac{P_p(3)}{2 p^a P_p(2)}  \\
    & = 2^{-\nu(12d)} (12d)^{-1} \prod_{p \vert 12d} \left ( 1 - p^{-6} \right ).
  \end{align*}
  As 
  \begin{align*}
    \frac{\gamma_3(L)}{\alpha_3(L)} \frac{\gamma_2(L)}{\alpha_2(L)}
    & = 2^2 \vert 12d \vert^{-1}_3 \vert 12d \vert^{-1}_2 P_3(2) P_2(2) \alpha_3(L)^{-1} \alpha_2(L)^{-1} \\
    & = \frac{200}{81} \vert 12d \vert^{-1}_3 \vert 12d \vert^{-1}_2 \alpha_3(L)^{-1} \alpha_2(L)^{-1}, 
  \end{align*}
  then 
  \begin{align*}
  \prod_p \alpha_p(L)^{-1}
    = & \left ( \frac{\pi^{12}}{510300} \right ) \left ( \frac{200}{81} \right ) L(3, \left( \frac{-12d}{*} \right) )^{-1} 2^{-\nu(12d)} (12d)^{-1}
     \vert 12d \vert^{-1}_2  \vert 12 d \vert^{-1}_3   \\
     & \times \alpha_2(L)^{-1} \alpha_3(L)^{-1} \prod_{p \vert 12 d} \left (1-p^{-6} \right ).
  \end{align*}
     As noted in \cite{HMvolapps}, by the Legendre duplication formula \eqref{LegDupForm} (e.g. \cite[p.73]{Davenport}), 
  \begin{equation}\label{LegDupForm}
    \Gamma(z) \Gamma(z + 1/2) = 2^{1-2z} \pi^{1/2} \Gamma(2z),
  \end{equation}
  we have 
  \begin{equation}\label{PiGammaProd}
    \prod_{k=1}^6 \pi^{-k/2}\Gamma(k/2) = \frac{3 }{4 \pi^9}.
  \end{equation}
  Hence, by Theorem \ref{GHSHMvol}, 
  \begin{align*}
    \opn{Vol}_{HM}(\opn{O}(L))
    = &  \frac{3 (12d)^{7/2} }{2 \pi^9}  \prod_p \alpha_p(L)^{-1} \\
    = & (12d)^{7/2} \left ( \frac{3}{2 \pi^9} \right )
    \left (\frac{\pi^{12}}{510300} \right )
    \left ( \frac{200}{81} \right ) L(3, \left( \frac{-12d}{*} \right) )^{-1} 2^{-\nu(12d)} (12d)^{-1} \\
    &  \times \left( \prod_{p \vert 12 d} (1-p^{-6}) \right )
    \vert 12 d \vert^{-1}_3 \vert 12d \vert^{-1}_2 \alpha_3(L)^{-1} \alpha_2(L)^{-1} \\
    = & \frac{(12d)^{5/2}\pi^{3}}{(137781)  2^{\nu(12d)} \vert 12 d \vert_3 \vert 12d \vert_2} L(3, \left( \frac{-12d}{*} \right) )^{-1} \alpha_3(L)^{-1} \alpha_2(L)^{-1} \prod_{p \vert 12 d} (1-p^{-6}).
  \end{align*}
\end{proof}
\subsection{Hirzebruch-Mumford volumes of $\opn{O}(K^{u,m}_{r,s}(k,l))$}
We now calculate the Hirzebruch-Mumford volumes of $\opn{O}(K^{u,m}_{r,s}(k,l))$ for the lattices $K^{u,m}_{r,s}(k,l)$ in Proposition \ref{KgenusProp}. 
\begin{proposition}\label{KHMvolProp}
  Let $K$ denote the lattice $K^{u,m}_{r,s}(k,l)$ of Proposition \ref{KgenusProp} and suppose $N_0$ denotes the unimodular factor of $K \otimes \ZZ_p$ in \eqref{jordandecomp}.
  Then, 
\[
  \begin{array}{l l}
\opn{Vol}_{HM}(\opn{O}(K^{2s, 2s}_{r,s}(0,l)))  = \frac{r^2}{160} 2^{-\nu(2r)} (\alpha_2(K)^{-1} \vert 2r \vert_2^{-1}) \prod_{p \vert 2r} (1 + p^{-2})
    & \text{if $l$ is odd} \\[7pt]
    \opn{Vol}_{HM}(\opn{O}(K^{2s, s}_{r,l}(0,l)))  \leq \frac{r^2}{6}2^{-\nu(8r)} \prod_{p \vert 8r} (1 + \chi(N_0)p^{-2}) & \text{if $l$ is even} \\[7pt]
    \opn{Vol}_{HM}(\opn{O}(K^{2s,2s}_{r,s}(r,l)))  = 
    \frac{r^2}{160} 2^{-\nu(2r)} (\alpha_2(K)^{-1} \vert 2r \vert_2^{-1}) \prod_{p \vert 2r}(1 + \chi(N_0)p^{-2})
    & \text{if $l$ is odd} \\[7pt]
    \opn{Vol}_{HM}(\opn{O}(K^{2s, 2s}_{r,l}(r,l)))  = \frac{r^2}{3600}2^{-\nu(2r)} \prod_{p \vert 2r} (1 + \chi(N_0)p^{-2}) & \text{if $l$ is even.}
  \end{array}
  \]
\end{proposition}
\begin{proof}
We use $N_i$ to denote unimodular $\ZZ_p$-lattices and, for prime $p$, we define $a$ by $\vert \opn{det} K \vert_p = p^{-a}$.
  We note that in each case $\opn{Sg}(K)$ contains exactly one class, which follows from \cite[Theorem 1.13.2]{Nikulin}.

  \noindent $\mathbf{K^{2s,2s}_{r,s}(0,l)},$ $\mathbf{l}$ \textbf{odd}. 
  Let $K = K^{2s, 2s}_{r,s}(0,l)$ where $l$ is odd.
  If $p=2$ then, by Theorem 1.9.1 and  Corollary 1.9.3 of \cite{Nikulin},
  \begin{equation*}
    K_2 = 2U \op \la \theta_a \ra (2^a), 
  \end{equation*}
  implying $m=3$, $w=a$ and $q=1$.
  \begin{enumerate}[(i)]
  \item If $a=1$ then
  $E_j=1$ for $j<-1$ or $j>3$,
  $E_{-1}=1$,
  $E_0 = 1/2$,
  $E_1 = 1$,
  $E_2 = 1/2$ and
  $E_3 = 1$.
\item  If $a=2$ then
  $E_j=1$ for $j<-1$ or $j>3$, 
  $E_{-1}=1$,
  $E_0=\frac{1}{2}(1+2^{-2})$,
  $E_1=1/2$,
  $E_2=1$ and
  $E_3=1/2$. 
\item If $a=3$ then
  $E_j = 1$ for $j<-1$ or $j>5$,
  $E_{-1}=1$,
  $E_0 = \frac{1}{2}(1 + 2^{-2})$,
  $E_1=1$,
  $E_2=1/2$,
  $E_3=1$,
  $E_4=1/2$ and
  $E_5=1$. 
\item   If $a>3$ then
  $E_j = 1$ for $j <-1$ or $j>a+1$,
  $E_{-1}=1$,
  $E_0 = \frac{1}{2}(1 + 2^{-2})$,
  $E_1 = 1$,
  $E_j = 1$ for $1<j<a-1$,
  $E_{a-1} = 1/2$,
  $E_a = 1$ and
  $E_{a+1} = 1/2$. 
\end{enumerate}
  Therefore,
  \[ 
  E_2(K) =
    \begin{cases}
      4 & \text{if $a=1$} \\
      8(1+2^{-2})^{-1} & \text{if $a>1$} 
    \end{cases}
\]
    and     
  \begin{equation}\label{K2s2sr0l,odd}
    \alpha_2(K) =
    \begin{cases}
      64 P_2(2) & \text{if $a=1$} \\
      2^{6+a} P_2(2) (1+2^{-2})^{-1} & \text{if $a>1$.}
    \end{cases}
  \end{equation}
  If $p \neq 2$ then, by Theorem 1.9.1 and  Corollary 1.9.3 of \cite{Nikulin},
  \begin{equation}\label{Kp}
    K_p =
    \begin{cases}
      \la 1 \ra^{\op 4} \op \la \theta_0 \ra & \text{if $a =0$} \\
      \la 1 \ra^{\op 3} \op \la \theta_0 \ra \op \la \theta_a \ra (p^a) & \text{if $a>0$}.
    \end{cases}
  \end{equation}
  If $a=0$ then $s=1$, $w=0$, $P_p(K) = P_p(2)$ and $E_p(K)=1$.
  If $a>0$ then $s=2$, $w=a$, $P_p(K) = P_p(2)$ and $E_p(K)=(1 + p^{-2})^{-1}$ (noting that $K = 2U \op \la -2r \ra$ by \cite[Corollary 1.13.4]{Nikulin}).
  Therefore,
  \[
    \alpha_p(K) =
    \begin{cases}
      P_p(2) & \text{if $a=0$} \\
      2 p^a P_p(2) (1 + p^{-2})^{-1} & \text{if $a>0$.}
    \end{cases}
  \]
  As
  \begin{align*}
    \prod_p \alpha_p(K)^{-1}
    & =  \left (\frac{ 2 \alpha_2(K)^{-1} \vert 2r \vert_2^{-1}}{1+2^{-2}} \right )
    \left (\prod_{p \vert 2r} 2^{-1} p^{-a} (1+p^{-2}) \right ) \left (\prod_p P_p(2)^{-1} \right )P_2(2) \\
    & = \frac{2 \zeta(2) \zeta(4)}{1+2^{-2}} 2^{-\nu(2r)} (2r)^{-1} (\alpha_2(K)^{-1} \vert 2r \vert_2^{-1}) P_2(2) \prod_{p \vert 2r} (1 + p^{-2}) \\
    & = \frac{\pi^6}{480} 2^{-\nu(2r)} (2r)^{-1} (\alpha_2(K)^{-1} \vert 2r \vert_2^{-1}) \prod_{p \vert 2r} (1 + p^{-2})
  \end{align*}
  and, as in \eqref{PiGammaProd}, 
  \begin{equation}\label{gamma5}
    \prod_{k=1}^5 \pi^{-k/2} \Gamma(k/2) = \frac{3}{8 \pi^{6}},
  \end{equation}
  then
  \begin{align*}
    \opn{Vol}_{HM}(\opn{O}(K))
& = \frac{r^2}{160} 2^{-\nu(2r)} (\alpha_2(K)^{-1} \vert 2r \vert_2^{-1}) \prod_{p \vert 2r} (1 + p^{-2}).
  \end{align*}
    
  \noindent $\mathbf{K^{2s,s}_{r,s}(0,l),}$  \textbf{$l$ even.}
  Let $K=K^{2s,s}_{r,s}(0,l)$ where $l$ is even.
  If $p=2$ then, by Corollary 1.9.3 of \cite{Nikulin},
  \begin{equation*}
    K_2 =
    \begin{cases}
      \begin{array}{llll}
        U \op N_1(2) & \text{or} & V \op N_1(2) & \text{if $a=3$}  \\
        U \op N_1(2) \op N_{a-2}(2^{a-2}) & \text{or} & V \op N_1(2)  \op N_{a-2}(2^{a-2}) & \text{if $a>3$}
      \end{array}
    \end{cases}
  \end{equation*}
  where
  \begin{equation*}
    V =
    \begin{pmatrix}
      2 & 1 \\
      1 & 2
    \end{pmatrix}.
  \end{equation*}
  If $a=3$ then $q \leq 4$, $w = 6$, $P_2(K) \geq P_2(1)^2$ and $E_2(K) \geq 16/9$.
  If $a>3$ then $q \leq 5$, $w=a+3$, $P_2(K) \geq P_2(1)^2$ and $E_2(K) \geq 16/9$, implying $\alpha_2(K) \geq 2^a$.

  If $p \neq 2$ then $K_p$ is as in \eqref{Kp}.
  If $a=0$ then $s=1$, $w=0$, $P_p(K)=P_p(2)$ and $E_p(K)=1$;
  if $a>0$ then $s=2$, $w=a$, $P_p(K)=P_p(2)$ and $E_p(K) = (1 + \chi(N_0)p^{-2})^{-1}$. 
  Hence, 
  \[
  \alpha_p(K) =
  \begin{cases}
    P_p(2) & \text{if $a=0$}\\
    2p^a P_p(2)(1 + \chi(N_0) p^{-2})^{-1} & \text{if $a>0$}
  \end{cases}
  \]
  and  
  \begin{align*}
    \prod_p \alpha_p(K)^{-1} & \leq
    \left ( \prod_p P_p(2)^{-1} \right )
    \left ( \prod_{p \vert 8r} 2^{-1} \vert 8r \vert_p (1 + \chi(N_0) p^{-2}) \right )
    \left ( 2 \vert 8r \vert_2^{-1} P_2(2) (1 + \chi(N_0)2^{-2})^{-1} \right ) \vert 8r \vert_2  \\
& \leq
    (15/8) 2^{-\nu(8r)} (8r)^{-1}
    \left ( \prod_p P_p(2)^{-1} \right ) 
    \left ( \prod_{p \vert 8r} (1  + \chi(N_0)p^{-2} ) \right ).
  \end{align*}
  Therefore, by \eqref{gamma5} and the standard values for $\zeta(2)$ and $\zeta(4)$,
  \begin{align*}
    \opn{Vol}_{HM}(\opn{O}(K))
& \leq
    \frac{r^2}{6}2^{-\nu(8r)} \prod_{p \vert 8r} (1 + \chi(N_0)p^{-2}).
\end{align*}

\noindent $\mathbf{K^{2s,2s}_{r,s}(r,l),}$ \textbf{$l$ odd.} 
If $p = 2$ then by \cite[Corollary 1.9.3]{Nikulin},
\[
K_2 = 2U \op \la \theta_a \ra (2^a),
\]
implying $K_2$ and $\alpha_2(K)$ are as for $K^{2r,2s}_{r,s}(0, l)$ with $l$ odd. 
If $p \neq 2$ then $K_p$ and $\alpha_p(K)$ are as for $K_{r,s}^{2s,2s}(0, l)$ with $l$ even.
Therefore,
\begin{align*}
\prod_p \alpha_p(K)^{-1}
= & 
\left ( \prod_p P_p(2)^{-1} \right)
\left ( \prod_{p \vert 2r} 2^{-1}\vert 2r \vert_p (1+\chi(N_0)p^{-2}) \right ) \\
& \times 2 \vert 2r \vert_2^{-1} \alpha_2(K)^{-1}(1 + 2^{-2})^{-1} P_2(2). 
\end{align*}
Hence, by \eqref{gamma5},
\[
\opn{Vol}_{HM}(\opn{O}(K)) = \frac{r^2}{160} 2^{-\nu(2r)} \left ( \prod_{p \vert 2r} 1+ \chi(N_0)p^{-2} \right ) \alpha_2(K)^{-1} \vert 2r \vert_2^{-1}.
\]
\noindent $\mathbf{K^{2s,2s}_{r,s}(r,l)},$  \textbf{$l$ even.}
  Let $K = K^{2s, 2s}_{r,s}(r,l)$ where $l$ is even.
  If $p=2$ then, by Theorem 1.9.1, Corollary 1.9.3 and Proposition 1.8.1 and 1.82 of \cite{Nikulin}, 
  \begin{equation*}
    K_2 = 2U \op \la \theta_1 \ra (2)
  \end{equation*}
  (where we have used the fact that $r$ is odd).
  Then, as in the case of $K^{2s,2s}_{r,s}(0,l)$ with $l$ odd, $\vert 2r \vert_2=2^{-1}$, $E_2(K)=4$  and $\alpha_2(K) = 64 P_2(2)$.
  
  If $p \neq 2$ then, by Theorem 1.9.1 and  Corollary 1.9.3 of \cite{Nikulin}, $K_p$ is as in \eqref{Kp}.
  If $a=0$ then $s=1$, $w=0$, $P_p(K) = P_p(2)$ and $E_p(K)=1$.
  If $a>0$ then $s=2$, $w=a$, $P_p(K)=P_p(2)$ and $E_p(K)=(1+\chi(N_0)p^{-2})^{-1}$.
  Therefore,
  \begin{equation*}
    \alpha_p(K) =
    \begin{cases}
      P_p(2) & \text{if $a=0$} \\
      2 p^a P_p(2) (1+\chi(N_0)p^{-2})^{-1} & \text{if $a>0$,}
    \end{cases}
  \end{equation*}
  and so,  
  \begin{align*}
    \prod_p \alpha_p(K)^{-1}
    & = 
    \frac{1}{32 \vert 2r \vert_2 (1 + \chi(N_0)2^{-2})}
    \left ( \prod_p P_p(2)^{-1} \right )
    \left ( \prod_{p \vert 2r} 2^{-1} \vert 2r \vert_p
    (1 + \chi(N_0)p^{-2}) \right ) \\
    & =
    \frac{\zeta(2) \zeta(4)}{16 (1+2^{-2}) 2^{\nu(2r)} (2r)} \prod_{p \vert 2r} (1 + \chi(N_0)p^{-2})
  \end{align*}
  (where the denominator $(1 + \chi(N_0)2^{-2})$ is obtained from the decomposition of $K_2$).
  Hence, 
  \begin{align*}
    \opn{Vol}_{HM}(\opn{O}(K))
    & = \left ( \frac{6(2r)^3}{8 \pi^6} \right ) \left ( \frac{\zeta(2) \zeta(4)}{16 (1+2^{-2}) 2^{\nu(2r)} (2r)} \right ) \prod_{p \vert 2r} (1 + \chi(N_0)p^{-2}) \\
    & = \frac{r^2}{3600}2^{-\nu(2r)} \prod_{p \vert 2r} (1 + \chi(N_0)p^{-2}).
  \end{align*}     
\end{proof}
 \section{Bounds for Hirzebruch-Mumford Volumes}\label{hmbounds}
We now produce bounds for the Hirzebruch-Mumford volumes of \S\ref{HMSec}.
We will use these bounds in \S\ref{ObsSec} in connection with obstruction calculations.
In this section, $L$ is the lattice
\[
L=2U \op \la -2d \ra \op \la -6 \ra
\] 
(i.e. \eqref{ldef} with $n=2$) and $\Gamma \subset \opn{O}^+(L)$ is the modular group defined in \S\ref{modgp}.
\subsection{Bounds for $\opn{Vol}_{HM}(\opn{O}(L))$}
\begin{lemma}\label{LfBound}
  If $D \in \NN$, $s>1$ and $\chi$ is an arithmetic function taking values in $\{0, \pm 1\}$ then,   
\[
\begin{array}{ccc}
  \zeta(2s) \zeta(s)^{-1} \leq L(s, \chi) \leq \zeta(s) 
  & \text{and} & 
  \zeta(2s) \zeta(s)^{-1} \leq \prod_{p \vert D} (1 - \chi(p)p^{-s})^{-1} \leq \zeta(s).
\end{array}
\]
\end{lemma}
\begin{proof}
  If $\lambda$ is the Liouville function then, 
  \begin{equation*}
    L(s, \lambda)
    =
    \prod_p (1 + p^{-s})^{-1}
    \leq
    \prod_{p \vert D} (1 + p^{-s})^{-1}
    \leq
    \prod_{p \vert D} (1 - \chi(p)p^{-s})^{-1}
    \leq
    \prod_{p \vert D} (1 - p^{-s})^{-1}
    \leq
    \zeta(s).
  \end{equation*}
  As $L(s, \lambda)=\zeta(2s) \zeta(s)^{-1}$ \cite[p.335]{HardyWright} then, 
  \begin{equation*}
    \zeta(2s) \zeta(s)^{-1}
    \leq
    \prod_{p \vert D} (1 - \chi(p)p^{-s})^{-1}
    \leq
    \zeta(s),
  \end{equation*}
  from which the result follows.
\end{proof}
\begin{lemma}\label{Lbound}
  We have
  \begin{equation*}
    \opn{Vol}_{HM}(\opn{O}(L)) \geq \frac{(12d)^{5/2} }{ (248832) \zeta(3) \pi^3 2^{\nu(12d)}}.
  \end{equation*}
\end{lemma}
\begin{proof}
  By Proposition \ref{HMvolOL},
  \[
  \opn{Vol}_{HM}(\opn{O}(L)) = 
    \frac{(12d)^{5/2} \pi^3}{137781}
    L \left (3, \left ( \frac{-12d}{*} \right ) \right )^{-1}
    2^{-\nu(12d)}  \vert 12d \vert^{-1}_3 \vert 12d \vert^{-1}_2 \alpha_3(L)^{-1} \alpha_2(L)^{-1}
    \prod_{p \vert 12d} (1-p^{-6}),   
  \]
  where
  \begin{equation*}
    \alpha_2(L) =
    \begin{cases}
      45 \vert 12 d \vert_2^{-1} 
      & \text{if $a=2$ and $d \not \equiv 3 \bmod{ 4}$ } \\
      90 \vert 12 d \vert_2^{-1} 
      & \text{if $a=2$ and $d \equiv 3 \bmod{ 4}$} \\
      90 \vert 12 d \vert_2^{-1} 
      & \text{if $a=3$ or $a=4$} \\
      180 \vert 12 d \vert_2^{-1} 
      & \text{if $a \geq 5$}
    \end{cases}
  \end{equation*}
  and
  \begin{equation*}
    \alpha_3(L) =
    \begin{cases}
      2 \vert 12 d \vert_3^{-1} P_3(2) & \text{if $a=1$} \\
      6 \vert 12 d \vert_3^{-1} P_3(2) P_3(1) (1+3^{-2})^{-1} ( 1 + \left ( \frac{-4.3^{-1} d}{p} \right ) 3^{-1})^{-1} &
      \text{if $a=2$} \\
      12 \vert 12 d \vert_3^{-1} P_3(2) (1+3^{-2})^{-1} & \text{if $a>2$}.
    \end{cases}
  \end{equation*}
  As $P_2(2) = 45/64$, $P_3(1) = 8/9$ and $P_3(2) = 640/729$, then  
  \begin{equation*}
    \alpha_2(L)^{-1} \geq \frac{\vert 12 d \vert_2}{256 P_2(2)} = \frac{\vert 12 d \vert_2}{180}
  \end{equation*}
  and
  \begin{equation*}
    \alpha_3(L)^{-1} \geq \frac{27 \vert 12 d \vert_3}{256}, 
  \end{equation*}
  implying 
  \begin{equation*}
    \vert 12d \vert_3^{-1} \vert 12d \vert_2^{-1} \alpha_3(L)^{-1} \alpha_2(L)^{-1} \geq \frac{3}{5120}.
  \end{equation*}
  By Lemma \ref{LfBound},
  \begin{equation*}
    \begin{array}{ccc}
      L(3, \left( \frac{-12d}{*} \right ) )^{-1} \geq \zeta(3)^{-1}
      & \text{and}
      & \prod_{p \vert 12d} (1-p^{-6}) \geq \zeta(6)^{-1},
    \end{array}
  \end{equation*}
  from which the result follows.
\end{proof}
\subsection{Bounds for $\opn{Vol}_{HM}(\Gamma)$.}
\begin{lemma}\label{GammaBound}
We have 
\begin{equation*}
      \begin{array}{c c c}
        \opn{Vol}_{HM}(\Gamma) \geq  \dfrac{\gamma(12d)^{5/2} }{ 124416 \zeta(3) \pi^3 }
        & \text{where} &
        \gamma =
        \begin{cases}
          1/4 & \text{if $(6,d)=1$} \\
          1 & \text{if $(6,d)=6$} \\
          1/2 & \text{otherwise.}
        \end{cases}
        \end{array}
    \end{equation*}
\end{lemma}
  \begin{proof}
    By Proposition \ref{modgp}, $\Gamma = \widetilde{\opn{SO}}^+(L) \rtimes \la \sigma_v \ra$, implying 
    \[
    \vert \opn{O}(L) : \widetilde{\opn{SO}}^+(L) \vert  
    = 
    2 \vert \opn{O}(L) : \Gamma \vert
    =
    2 \vert \opn{O}(L) : \widetilde{\opn{O}}^+(L) \vert.
    \]
    By Lemma 3.2 of \cite{HMvolapps},
    $\vert \opn{O}^+(L) : \widetilde{\opn{O}}^+(L) \vert
    =
    \vert \opn{O}(q_L ) \vert
    $
    and as $ -I \not \in \Gamma$ then,
    \begin{align*}
      \vert \la \opn{O}(L), -I \ra : \la \Gamma, -I \ra \vert
      &=
      \frac{1}{2} \vert \opn{O}(L) : \Gamma \vert \\
      &=
      \vert \opn{O}^+(L) : \widetilde{\opn{O}}^+(L) \vert \\
      & = \vert \opn{O}(q_L) \vert.
    \end{align*}
    We now produce a lower bound for $\vert \opn{O}(q_L) \vert$.
    If $(a,b) \in D(L) \cong C_6 \op C_{2d}$ then
    \[
    q_L(a,b) \equiv
     -\frac{a^2}{6} -  \frac{b^2}{2d}  \bmod{ 2\ZZ}.
     \]
     As in Lemma \ref{OqkLemma}, there are $2^{\nu(d)}$ solutions to
     \begin{equation}\label{gammab1}
       x^2 \equiv 1 \bmod{ 4d}
     \end{equation}
      and  2 solutions to
    \begin{equation}\label{gammab2}
      y^2 \equiv 1 \bmod{ 12}, 
    \end{equation}
    where $x \in \ZZ / (2d) \ZZ$ and $y \in \ZZ / 6 \ZZ$.
    As each solution to \eqref{gammab1} and \eqref{gammab2} is a generator for $\ZZ/(2d) \ZZ$ or $\ZZ/6 \ZZ$ then
    $(0,1) \mapsto (0,x)$
    and
    $(1,0) \mapsto (y,0)$
    define elements of $\opn{O}(q_L)$, implying  $\vert \opn{O}(q_L) \vert \geq 2^{\nu(d) + 1}$.
    By Lemma \ref{Lbound}, 
    \begin{equation*}
      \opn{Vol}_{HM}(\opn{O}(L)) \geq
      \left (
      \frac{(12d)^{5/2}}{124416 \zeta(3) \pi^3}
      \right )
      2^{-\nu(12d)} 
    \end{equation*}
    and the result follows from \eqref{subgpHMvol}, with $\gamma$  as in the statement of the Lemma.
\end{proof}
\subsection{Bounds for $\opn{Vol}_{HM}(\opn{O}(K))$}
As in Proposition \ref{KgenusProp}, we will often use $K$ to denote the lattice $K_{r,s}^{u,m}(k,l)$.
\begin{lemma}\label{Kbound}
  For $K^{u,m}_{r,s}(k,l)$ as in Proposition \ref{KgenusProp}, 
\[
  \begin{array}{l l }
    \opn{Vol}_{HM}(\opn{O}(K^{2s, 2s}_{r,s}({0,l})))  \leq \dfrac{r^2 }{240 \pi^2} 2^{-\nu(2r)} & \text{if $l$ is odd} \\
    [15pt]\opn{Vol}_{HM}(\opn{O}(K^{2s, s}_{r,s}({0,l})))  \leq \dfrac{5r^2}{2 \pi^2}2^{-\nu(2r)} & \text{if $l$ is even}\\
    [15pt]
    \opn{Vol}_{HM}(\opn{O}(K^{2s, 2s}_{r,s}(r,l)))  \leq \dfrac{r^2 }{240 \pi^2} 2^{-\nu(2r)}  & \text{if $l$ is odd} \\
    [15pt]
    \opn{Vol}_{HM}(\opn{O}(K^{2s, 2s}_{r,s}(r,l)))  \leq \dfrac{r^2 }{240 \pi^2} 2^{-\nu(2r)} & \text{if $l$ is even.}
  \end{array}
  \]
\end{lemma}
\begin{proof}
  \begin{enumerate}[(i)]
    \item If $l$ is odd and $K = K^{2s, 2s}_{r,s}(0,l)$ or $K^{2s,2s}_{r,s}(r,l)$ then, by Proposition \ref{KHMvolProp},
  \begin{equation*}
    \opn{Vol}_{HM}(\opn{O}(K))
    =
    \frac{r^2}{160} 2^{-\nu(2r)}
    (\alpha_2(K)^{-1} \vert 2r \vert^{-1}_2)
    \prod_{p \vert 2r} ( 1+ \chi_p p^{-2}), 
  \end{equation*}
  where $\chi_p = \pm 1$.
  By \eqref{K2s2sr0l,odd}, 
$\alpha_2(K)^{-1} \leq \frac{2}{45} \vert 2r \vert_2$ and 
  by Lemma \ref{LfBound},
  \[
  \prod_{p \vert 2r} (1+ \chi_p p^{-2}) \leq \zeta(2) \zeta(4)^{-1} = 15 \pi^{-2}, 
  \] 
  implying 
  \[
    \opn{Vol}_{HM}(\opn{O}(K))
\leq \frac{r^2 }{240 \pi^2} 2^{-\nu(2r)}.
\]
\item  If $l$ is odd and $K=K^{2s,s}_{r,s}(0,l)$ then, by Proposition \ref{KHMvolProp} and Lemma \ref{LfBound},
  \begin{equation*}
    \opn{Vol}_{HM}(\opn{O}(K)) \leq
    \frac{r^2}{6} 2^{-\nu(8r)}\prod_{p \vert 8r} (1 + \chi(N_0)p^{-2}) \leq 
    \frac{5r^2}{2 \pi^2}2^{-\nu(2r)}.
  \end{equation*}
\item Similarly, if $l$ is even and $K = K^{2s, 2s}_{r,s}(r,l)$ then, by Proposition \ref{KHMvolProp} and Lemma \ref{LfBound}, 
  \[
    \opn{Vol}_{HM}(\opn{O}(K))
    =
    \frac{(2r)^2}{14400} 2^{-\nu(2r)}
    \prod_{p \vert 2r} (1 + \chi(N_0) p^{-2})
\leq
\frac{r^2}{240 \pi^2} 2^{-\nu(2r)}.
  \]
  \end{enumerate}
\end{proof}
\subsection{Hirzebruch-Mumford volumes and elliptic obstructions}
We now produce bounds for Hirzebruch-Mumford volumes of groups related to elliptic obstructions.
\begin{lemma}\label{II-III-HMVol}
  Let $K^{II} = U \op U(3) \op \la -2r \ra$ and $K^{III}=U \op P \op \la -2r \ra$ where $P$ is an indefinite binary quadratic form satisfying $D(P)=C_9$.
  If $K=K^{II}$ or $K^{III}$ then 
  \[
  \opn{Vol}_{HM}(\opn{O}(K))
  \leq \frac{3 r^2}{5 \pi^2}   2^{-\nu(18r)}.
  \]
\end{lemma}
\begin{proof}
We use $\theta_0, \theta_a, \theta_a', \theta_a'', \theta_a'''$ to denote elements of $\ZZ_p^*$ and $D(L)_p$ to denote the $p$-component of the abelian group $D(L)$.
If $K=K^{II}$ or $K^{III}$, $p=2$ and $x \in D(K)_2 \cong \ZZ/ (2^a) \ZZ$, then
  \begin{equation*}
    q_K(x)_2 =
    \left ( \frac{\theta_2 x^2}{2^a} \right ) \bmod{ 2\ZZ_2}.
  \end{equation*}
  Therefore, by Proposition 1.8.1 of \cite{Nikulin},
  \begin{equation*}
    \begin{array}{ccc}
      K_2 \cong U \op V \op \la \theta_a \ra (2^a) & \text{or} & K_2 \cong U \op U \op \la \theta_a \ra(2^a),
    \end{array}
  \end{equation*}
  where
  \begin{equation*}
    V =
    \begin{pmatrix}
      2 & 1 \\
      1 & 2
    \end{pmatrix}.
  \end{equation*}
 Similarly, if $p=3$ then
  \begin{equation*}
    K_3^{II} =
    \begin{cases}
      \la 1 \ra^{\op 2} \op \la \theta_0 \ra \op (\la \theta'_3 \ra \op \la \theta_3 '' \ra)(3) & \text{if $a=2$}\\
      \la 1 \ra \op \la \theta_0 \ra \op ( \la \theta_3' \ra \op \la \theta_3'' \ra \op \la \theta_3 ''' \ra)(3)  & \text{if $a=3$} \\
      \la 1 \ra \op \la \theta_0 \ra \op (\la \theta_3' \ra \op \la \theta_3'' \ra)(3) \op \la \theta_{3^{a-2}} \ra (3^{a-2}) & \text{if $a>3$}
    \end{cases}
  \end{equation*}
  and 
  \begin{equation*}
    K_3^{III} =
    \begin{cases}
      \la 1 \ra^{\op 3} \op \la \theta_0 \ra \op \la \theta_a \ra (9) & \text{if $a=2$} \\
      \la 1 \ra^{\op 2} \op \la \theta_0 \ra \op \la \theta_3 \ra (3) \op \la \theta_9 \ra (9) & \text{if $a=3$} \\
      \la 1 \ra^{\op 2} \op \la \theta_0 \ra \op (\la \theta_9' \ra \op \la \theta_9'' \ra)(9) & \text{if $a=4$} \\
      \la 1 \ra^{\op 2} \op \la \theta_0 \ra \op (\la \theta_9 \ra)(9) \op (\la \theta_{3^{a-2}} \ra )(3^{a-2} ) & \text{if $a>4$}. 
    \end{cases}
  \end{equation*}
  If $K=K^{II}$ or $K^{III}$ and $p>3$ then (by direct calculation or \cite[Proposition 1.8.1]{Nikulin}),
  \begin{equation*}
    q_K(x)_p = \left (
    \frac{\theta_a x^2}{p^a} \right ) \bmod{ 2\ZZ_p}, 
  \end{equation*}
  where $x \in D(K)_p \cong \ZZ/(p^a) \ZZ$, implying 
  \begin{equation*}
    K_p =
    \begin{cases}
      \la 1 \ra^{\op 4} \op \la \theta_0 \ra & \text{if $a=0$} \\
      \la 1 \ra^{\op 3} \op \la \theta_0 \ra \op \la \theta_a \ra (p^a) & \text{if $a>0$.}
    \end{cases}
  \end{equation*}
  
  We now calculate the terms $\alpha_p(K)$. 
Suppose $K=K^{II}$ or $K^{III}$ and let $p=2$.
  Then $q=1$, $w=a$ and $P_2(K) = P_2(2)$.
  \begin{enumerate}
  \item  If $a=1$ then
    $E_{-1} = 1$,
    $E_0 = 1/2$,
    $E_1 = 1$,
    $E_2 = 1/2$
    and $E_2(K) \geq 4$. 
  \item If $a=2$ then
    $E_{-1} = 1$,
    $E_0 \leq 5/8$,
    $E_1 = 1/2$,
    $E_2 = 1$,
    $E_3 = 1/2$
    and 
    $E_2(K) \geq 32/5$.
  \item If $a=3$ then
    $E_{-1} = 1$,
    $E_0 \leq 5/8$,
    $E_1 = 1$,
    $E_2=1/2$,
    $E_3=1$,
    $E_4 = 1/2$
    and $E_2(K) \geq 32/5$.
  \item If $a \geq 4$ then
    $E_{-1} = 1$,
    $E_0 \leq 5/8$,
    $E_1=1$,
    $E_{a-1}=1/2$,
    $E_a=1$,
    $E_{a+1}=1/2$
    and 
    $E_2(K) \geq 32/5$.
  \end{enumerate}
  Therefore, $E_2(K) \geq 4$ and
  \[
  \alpha_2(K) \geq 2^{5+a} P_2(2).
  \]

  Suppose $K=K^{II}$. \begin{enumerate}
  \item If $a=2$ then $s=2$, $w=3$, $P_3(K) = P_3(1)^2$ and $E_3(K)=3/4$.
  \item If $a=3$ then $s=2$, $w=6$, $P_3(K)=P_3(1)^2$ and $E_3(K) = 3/4$.
  \item If $a>3$ then $s=3$, $w=a+3$, $P_3(K) = P_3(1)^2$ and $E_3(K) = 9/16$.
  \end{enumerate}
  Therefore, $s \geq 2$, $w \geq a+1$, $P_3(K) \geq P_3(1)^2$ and $E_3(K) \geq 9/16$.

  Suppose $K=K^{III}$. \begin{enumerate}
  \item If $a=2$ then $s=2$, $w=2$, $P_3(K) = P_3(2)$ and $E_3(K) \geq 9/10$.
  \item If $a=3$ then $s=3$, $w=4$, $P_3(K) =P_3(1)$ and $E_3(K) = 1$.
  \item If $a=4$ then $s=2$, $w=6$, $P_3(K) = P_3(1)^2$ and $E_3(K) \geq 3/4$.
  \item If $a>4$ then $s=3$, $w = a+2$, $P_3(K) = P_3(1)$ and $E_3(K) =1$.
  \end{enumerate}
  Therefore,
  $s \geq 2$, $w \geq a$, $P_3(K) \geq P_3(1)^2$ and $E_3(K) \geq 3/4$.
  Hence, for $K = K^{II}$ or $K^{III}$,  \[
  \alpha_3(K)
\geq 2^{-3} 3^{a+2} P_3(1)^2.
     \]
Similarly, if $p>3$, then
  \begin{equation*}
    \alpha_p(K) \geq
    \begin{cases}
      P_p(2) & \text{if $a=0$} \\
      2p^a P_p(2) (1 + p^{-2})^{-1} & \text{if $a>0$}
    \end{cases}
  \end{equation*}
  (as in the case of $K_{r,s}^{2s, 2s}(0, l)$ with $l$ odd in Theorem \ref{KHMvolProp}).
  
  Therefore,
  \begin{align*}
    \prod_p \alpha_p(K)^{-1}
    = & 
    \left ( \prod_{\substack{p \\ p \neq 2,3 }} \alpha_p(K)^{-1} \right )
    \alpha_2(K)^{-1} \alpha_3(K)^{-1} \\
    \leq &
    \left (  \prod_p P_p(2)^{-1} \right ) 
    \left ( \prod_{p \mid 18r } 2^{-1} \vert 18r \vert_p (1 + p^{-2} ) \right )
    \\
    & \times  \left ( \frac{16 P_3(2)}{9 P_3(1)^2 (1+3^{-2})} \right) \left( \frac{1}{2^4 (1+2^{-2})}\right ) \\
\leq & \frac{4}{45}
    \left ( \prod_p P_p(2)^{-1} \right )
    \left ( \prod_{p \mid 18r } 2^{-1} \vert 18r \vert_p (1+p^{-2}) \right ).
  \end{align*}
  From which we obtain,
  \begin{align*}
    \opn{Vol}_{HM}(\opn{O}(K))
    & \leq
    \left ( \frac{108 r^2}{5 \pi^6} \right )
    \zeta(2) \zeta(4) 2^{-\nu(18r)} \prod_{p \vert 18r} (1+p^{-2}), \\
    \intertext{and by Lemma \ref{LfBound},}
    \opn{Vol}_{HM}(\opn{O}(K))
    & \leq \left ( \frac{3 r^2}{5 \pi^2} \right ) 2^{-\nu(18r)}. 
  \end{align*}
\end{proof}
 \section{Obstruction calculations}\label{ObsSec}
In this section we study the obstruction spaces $\opn{RefObs}_{(4-a)k}(\Gamma)$ and $\opn{EllObs}_{(4-a)k}(\Gamma, \ec)$, defined in \eqref{RefObsdef} and \eqref{EllObsdef}.
We begin by characterising the spaces in Proposition \ref{refobsprop} and \ref{ellobsprop};
we exhibit a collection of vectors $\ec$ satisfying the conditions of Proposition \ref{intsingextprop} and \ref{bcextprop} in Lemma \ref{eclem}; 
we study the growth of $\opn{RefObs}_{(4-a)k}(\Gamma)$ and $\opn{EllObs}_{(4-a)k}(\Gamma, \ec)$ in Propositions \ref{alphabound}, \ref{betaIbound}, \ref{betaIIbound} and \ref{betaIIIbound}; and examine their codimension in Proposition \ref{mfbounds2}, following \cite{HMvol}. Throughout, we assume that $L=2U \op \la -6 \ra \op \la -2d \ra$ and use $x \Vert y$ to denote $x \vert y$ and $(x, y/x)=1$.
\subsection{Reflective Obstructions}\label{RefObsSec}
\begin{proposition}[{\cite[Proposition 4.1]{HMvol}}]\label{refobsprop}
  If $0<a<4$ is integral and $k>0$ is even, then
\[
  \opn{RefObs}_{k(4-a)}(\Gamma) \subset \bigoplus_{z \in \rc / \Gamma} \bigoplus_{i=0}^{k/2-1} M_{k(4-a) + 2i}(\Gamma \cap \opn{O}^+(z^{\perp})),
\]
where $\rc/ \Gamma$ is a set of representatives for the $\Gamma$-orbits of
  \[
  \rc := \{ z \in L \mid \text{$z$ is primitive and } \pm \sigma_z \in \Gamma \}.
  \]
\end{proposition}
  \begin{definition}
By Hirzebruch-Mumford proportionality, we define $\alpha(a)$ by
\begin{equation}\label{alphadef}
  \opn{dim} \bigoplus_{z \in \rc / \Gamma} \bigoplus_{i=0}^{k/2-1} M_{k(4-a) + 2i}(\Gamma \cap \opn{O}^+(z^{\perp})) = \alpha(a) k^4 + O(k^3).
\end{equation}
\end{definition}
  We bound $\alpha(a)$ in Proposition \ref{alphabound}, starting with bounds for $\rc/\Gamma$ in Lemma \ref{JjLemma}. 
\begin{lemma}\label{JjLemma}
  If $L=2U \op \la -2r \ra \op \la -2s \ra$, let $J(K)$ denote the number of $[z] \in \rc/\Gamma$ such that $z^{\perp} \cong K$.
\begin{enumerate}
  \item If $l^2 \equiv 1 \bmod{ s}$, $l$ is odd and $K=K^{2s, 2s}_{r,s}(0, l)$, then
    $J(K) \leq 2^{\nu(s)+1}$ if $s$ is odd;
    $J(K) \leq 2^{\nu(s)}$ if $2 \Vert s$;
    $J(K) \leq 2^{\nu(s)+1}$ if $4 \Vert s$; and
    $J(K) \leq 2^{\nu(s) + 2}$  if $8 \vert s$.  
  \item If $2 l_1^2 \equiv 2 \bmod{ 2s}$, $l=2l_1$ is even and $K= K^{2s,s}_{r,s}(0,l)$, then
    $J(K) \leq 2^{\nu(s)}$ if $s$ is odd;
    $J(K) \leq 2^{\nu(s)-1}$ if $2 \Vert s$;
    $J(K) \leq 2^{\nu(s)}$ if $4 \Vert s$; and 
    $J(K) \leq 2^{\nu(s) + 1}$ if $8 \vert s$.
  \item If $l^2 \equiv 1 \bmod{ s}$, $l$ is even and $K= K^{2s, 2s}_{r,s}(r,l)$ then $J(K) \leq 2^{\nu(s)+1}$.
  \item If $l^2 \equiv 2 \bmod{ 2s}$, $l$ is odd and $K=K^{2s,2s}_{r,s}(r,l)$, then
    $J(K) \leq 2^{\nu(s)+1}$ if $s$ is odd;
    $J(K) \leq 2^{\nu(s)}$ if $2 \Vert s$;
    $J(K) \leq 2^{\nu(s)+1}$ if $4 \Vert s$; and
    $J(K) \leq 2^{\nu(s) + 2}$  if $8 \vert s$.  
  \end{enumerate}
\end{lemma}
\begin{proof}
  By Proposition \ref{KgenusProp}, if $z \in \rc/\Gamma$ with $u:=-(z,z)$, $m:=\opn{div}(z)$ and $z^*=z/\opn{div}(z) = (k,l) \in D(L)$ then $z^{\perp} \cong K^{u,m}_{r,s}(k,l)$.
 As $\widetilde{\opn{SO}}^+(L) \subset \Gamma$ then,  by the Eichler criterion, it suffices to bound the number of $l \in \ZZ/(2s) \ZZ$ satisfying the conditions of Lemma \ref{conglemma}/Corollary \ref{qrinvs}.
 We count solutions as in \cite[Lemma 3.3]{HMvolapps}.
  \begin{enumerate}[(i)]
  \item \label{Bj1} Suppose $l^2 \equiv 1 \bmod{ s}$ where $l$ is odd and $z^{\perp} \cong K^{2s,2s}_{r,s}(0,l)$.
    We count solutions to $(l-1)(l+1) \equiv 0 \bmod{ p_i^{a_i}}$ where $\{p_i^{a_i} \}$ are the distinct prime power divisors of $s$.
    As $m=\opn{div}(z) =2s$ is the order of $(0,l) \in D(L)$ then $(l,2s)=1$ and so, if $p_i$ is odd, then $p_i$ cannot divide both $l-1$ and $l+1$, otherwise $p_i \vert l$ and $(l,s) \neq 1$.
    Therefore, $l-1 \equiv 0 \bmod{ p_i^{a_i}}$ or $l+1 \equiv 0 \bmod{ p_i^{a_i}}$.
    If $2^{a} \Vert s$ then $l^2 - 1 \equiv 0 \bmod{ 2^a}$ has exactly one solution if $a=0$ or $1$ and two solutions if $a=2$.
    If $a \geq 3$ then only one of $4 \vert l-1$ or $4 \vert l+1$ can occur (otherwise  $2 \vert (l,s)$) and so $l \equiv \pm 1 \mod{ 2^a}$ or $l \equiv \pm 1 + 2^{a-1} \bmod{ 2^a}$.
    Therefore, by the Chinese remainder theorem, $l^2 - 1 \equiv 0 \bmod{ s}$ has
    $2^{\nu(s)}$ solutions if $s$ is odd;
    $2^{\nu(s)-1}$ solutions if $2 \Vert s$;
    $2^{\nu(s)}$ solutions if $4 \Vert s$;
    and
    $2^{\nu(s)+1}$ solutions if $8 \vert s$.
    If $(l,s)=1$ and $s$ is even then $(l,2s)=(l+s, 2s)=1$.
    If $(l,s)=1$ and $s$ is odd, then precisely one of $(l,2s)=1$ and $(l+s, 2s)=2$ is satisfied.
    We therefore obtain a bound for the solutions to $l^2 \equiv 1 \bmod{ s}$ with $l \in \ZZ/(2s)\ZZ$ and $(l, 2s)=1$.
\item Suppose $l=2l_1$ where $2 l_1^2 \equiv 2 \bmod{ 2s}$ and $z^{\perp} \cong K^{2s,s}_{r,s}(0,l)$.
    As $l$ is uniquely determined by $l_1 \bmod{ s}$, we count solutions to $l_1^2 \equiv 1 \bmod{ s}$ as in \eqref{Bj1}.
    There are
    $2^{\nu(s)}$ solutions if $s$ is odd;
    $2^{\nu(s)-1}$ solutions if $2 \Vert s$;
    $2^{\nu(s)}$ solutions if $4 \Vert s$;
    and
    $2^{\nu(s) + 1}$ solutions if $8 \vert s$.
  \item Suppose $l^2 \equiv 1 \bmod{ s}$ with $l$ even and $z^{\perp} \cong K^{2s,2s}_{r,s}(r, l)$.
    As $l$ is even then $s$ is odd.
    As in \eqref{Bj1}, there are  $2^{\nu(s)}$ solutions to $l^2 \equiv 1 \bmod{ s}$ and  $2^{\nu(s)+1}$ solutions if $l$ is taken in $\ZZ / (2s) \ZZ$.
  \item As in \eqref{Bj1}.
  \end{enumerate}
\end{proof}
To use the Hirzebruch-Mumford bounds of \S\ref{hmbounds}, we will also need bounds for $\vert \opn{O}(q_K) \vert$ where $K \cong z^{\perp}$ for $z \in \rc$.
\begin{lemma}\label{OqkLemma}
  If $K$ is a lattice and $D(K) \cong C_{2r}$ then
  $\vert \opn{O}(q_K) \vert \leq 2^{\nu(r)}$. 
\end{lemma}
\begin{proof}
  Suppose $x \in \ZZ/(2r) \ZZ$ is a generator for $D(K)$. 
  By Proposition 1.8.1 of \cite{Nikulin}, the discriminant form 
  \begin{equation*}
    q_K(x)=    \frac{x^2 \theta}{2r} \bmod{2 \ZZ}
  \end{equation*}
  where $\theta \in \ZZ$ and $(\theta, 2r)=1$. 
As $(\theta, 2r)=1$, the condition $q_K(y) = q_K(x)$ is equivalent to
  \begin{equation}\label{qkcong}
    x^2 \equiv y^2  \bmod{4r}.
  \end{equation}
Suppose $4r = p_1^{a_1} \ldots p_n^{a_n}$ where $p_i$ are distinct primes, $p_1 = 2$ and $a_i \neq 0$.
  By the Chinese remainder theorem, \eqref{qkcong} is satisfied if and only if
\[ 
  (x-y)(x+y) \equiv 0 \bmod{p_i^{a_i}} 
  \] 
$\forall p_i^{a_i} \Vert 4r$.
  If $p>2$ and $p \mid (x-y)$ and $p \mid (x+y)$ then $p \vert 2x$, contradicting $(x, 2r) = 1$.
  Therefore, precisely one of $x-y \equiv 0 \bmod{ p_i^{a_i}}$ or $x+y \equiv 0 \bmod{p_i^{a_i}}$ is true for each $i$.
  Let $p=2$.
  If $a_1=2$ then $x \equiv 1 \mod{2}$;
  if $a_1=3$, then (as $x$ is odd) $2 \Vert x-y$ and $4 \Vert x+y$; or $4 \Vert x-y$ and $2 \Vert x+y$,
  implying $x \equiv \pm y \bmod{ 4}$;
  if $a>3$, then $4$ cannot divide both $x-y$ and $x+y$ otherwise $4 \vert 2x$, implying  $x$ is even.
  Therefore, $2^{a_1 - 1} \Vert x+y$ or $2^{a_1 - 1} \Vert x-y$ and so $x \equiv \pm y \bmod{2^{a_1 - 1}}$. 
  As $D(K)$ is cyclic, then any element of $\opn{O}(q_k)$ is defined by its image on a generator of $D(K)$.
  Therefore, by counting solutions to \eqref{qkcong} for $x \in \ZZ / (2r) \ZZ$, we conclude that $\vert \opn{O}(q_K) \vert \leq 2^{\nu(r)}$.
\end{proof}
\begin{lemma}\label{c2c2c2crLemma}
  If $K$ is a lattice such that $D(K) \cong C_2^{\op 3} \op C_r$ where $r$ is odd then $\vert \opn{O}(q_K) \vert \leq 2^{\nu(r) + 3}$.
\end{lemma}
\begin{proof}
  Let $q_K^2$ and $q_K^r$ be the restriction of the discriminant form $q_K$ to $C_2^{\op 3} \subset D(K)$ and $C_r \subset D(K)$, respectively.
  As $r$ is odd then, as explained in \cite[\S1]{Nikulin}, 
$q_K = q_K^2 \op q_K^r$ 
  and so  $\opn{O}(q_K) \cong \opn{O}(q_K^2) \times \opn{O}(q_K^r)$.
  As $\opn{O}(q_K^2) \subset \opn{Aut}(C_2^{\op 3})$ and $\vert \opn{Aut}(C_2^{\op 3}) \vert = 8$ then $\vert \opn{O}(q_K) \vert \leq 8 \vert \opn{O}(q_K^r) \vert$.
  By \cite[Proposition 1.8.1]{Nikulin}, if $x \in \ZZ/r \ZZ$ then
  \begin{equation}\label{qkr}
    q_K^r(x) = \frac{x^2 \theta}{r} \bmod{ 2 \ZZ}
  \end{equation}
  where $(\theta, r)=1$.
  To bound $\vert \opn{O}(q_K^r) \vert$, we count solutions to $q_K^r(x)=q_K^r(y)$ satisfying $(x,r)=(y,r)=1$ for fixed $y$.
  As $(\theta, r)=1$ then \eqref{qkr} is equivalent to  $(x-y)(x+y) \equiv 0 \bmod{ 2r}$ or to 
  \begin{equation}\label{qkrpi}
    (x-y)(x+y) \equiv 0 \bmod{ p_i^{a_i}}
  \end{equation}
  for all primes $p_i$ satisfying $p_i^{a_i} \Vert 2r$.
  If the prime $p$ satisfies $p \vert 2r$, $p \vert x-y$ and $p \vert x+y$ then $p \vert 2x$, which is a contradiction unless $p=2$, implying \eqref{qkrpi} has exactly 2 solutions for each $p_i \neq 2$.
  If $p_i=2$ and $p_i^{a_i} \Vert 2r$ then $a_i=1$ and \eqref{qkrpi} has exactly one solution.
  Therefore, \eqref{qkr} has not more than $2^{\nu(r)}$ solutions, from which the result follows.
\end{proof}
We will need the following Lemma in \S\ref{ellobssec}.
\begin{lemma}\label{Oqk2-3-lem}
  If $K = U \op U(3) \op \la -2r \ra$ then $\vert \opn{O}(q_K) \vert \leq (624) 2^{\nu(2r)}$.
  If $K = U \op P \op \la -2r \ra$, where $P$ is an indefinite binary quadratic form satisfying $D(P)=C_9$, then $\vert \opn{O}(q_K) \vert \leq (72) 2^{\nu(2r)}$.
\end{lemma}
\begin{proof}
  If $K = U \op U(3) \op \la -2r \ra$ then $D(K) \cong C_3 \op C_3 \op C_{2r}$.
  The discriminant form $q_K$ is given by 
  \begin{equation}\label{KIIform}
  q_K(a,b,c) = \frac{2ab}{3} - \frac{c^2}{2r} \bmod{ 2 \ZZ}
  \end{equation}
  for $(a,b,c) \in D(K)$ and we define $x,y,z \in D(K)$ by $x=(1,0,0)$, $y=(0,1,0)$ and $z=(0,0,1)$.
If $g \in \opn{O}(q_K)$ then $gx$ and $gy$ belong to the  unique subgroups $C_3^{\op 2} \subset D(K)$  if $(2r,3)=1$ and  $C_3^{\op 3} \subset D(K)$ if $(2r, 3) \neq 1$.
  In each case, there are no more than $\vert \opn{Aut}(C_3^{\op 2}) \vert = 48$ or $\vert \opn{Aut}(C_3^{\op 3}) \vert = 624$ possible images for $(x,y)$.
  By considering \eqref{KIIform} as in Lemma \ref{OqkLemma}, there are at most $2^{\nu(2r)}$ images for $gz \in \la gx, gy \ra^{\perp} \cong C_{2r}$, implying $\vert \opn{O}(q_K) \vert \leq (624)2^{\nu(2r)}$.
Similarly, if $K = U \op P \op \la -2r \ra$, then  $D(K) \cong C_9 \op C_{2r}$ and 
  \begin{equation}\label{KIIIform}
  q_K(a,b) = \frac{\theta a^2}{9} - \frac{b^2}{2r} \bmod{ 2\ZZ}
  \end{equation}
  for $(a,b) \in D(K)$ and $\theta \in \ZZ$. 
  We define $x',y' \in D(K)$ by $x'=(1,0)$ and $y'=(0,1)$. 
  There are at most $6$, $18$ or $72$ elements of order 9 in $D(K)$ if $(2r,9)=1$, $3$ or $9$, respectively.
  By \eqref{KIIIform}, there are at most $2^{\nu(2r)}$ images for $y'$ in $(gx')^{\perp}$ and the result follows.
\end{proof}
\begin{proposition}\label{alphabound}
  Suppose that $2^{\nu(x)} \leq C_{\epsilon} x^{\epsilon}$ for some $\epsilon, C_{\epsilon}>0$.
  If $d$ is odd then 
  \[
  \alpha(a)
  \leq
  ( \alpha_1 d^2 +  \alpha_2 d^{\epsilon} ) C(a), 
  \]
  where
  $\alpha_1 = (108199)(675 \pi^2)^{-1}$,
  $\alpha_2 = (216263)(300 \pi^2)^{-1}C_{\epsilon}$
  and
  \begin{equation}\label{C(a)}
C(a) = \frac{1}{8}(4(4-a)^3 + 6(4-a)^2 + 4(4-a) + 1).
  \end{equation}
  \end{proposition}
\begin{proof}
Let
\[
I(K) := \opn{Max} \{ \vert \la \opn{O}(z^{\perp}), -I \ra : \la \opn{O}(z^{\perp}) \cap \Gamma , -I \ra \vert \mid z \in \rc \} 
\]
and 
\[
J(K) := \vert \{ z \in \rc / \Gamma \mid z^{\perp} \cong K \} \vert,  
\]
where $K=z^{\perp} \subset L$ for $z \in \rc$. 
As $C(a)$ in \eqref{C(a)}
is the $k^4$-coefficient of the polynomial
\[
\sum_{j=0}^{k/2 - 1} ((4-a)k + 2j)^3
\]
then $E(K):=I(K) J(K) \opn{Vol}_{HM}(\opn{O}(K))$ is an upper bound for the contribution from $[z] \in \rc/\Gamma$ with $z^{\perp} \cong K$ to $\alpha(a)/C(a)$ in \eqref{alphadef}.
We produced bounds for $J(K)$ in Lemma \ref{JjLemma}.
To bound $I(K)$ we will show that
\[
\vert \la \opn{O}(K), -I \ra : \la \opn{O}(K) \cap \Gamma, -I \ra \vert \leq \opn{O}(q_K)
\]
and apply Lemma \ref{OqkLemma} and \ref{c2c2c2crLemma}.

By Proposition \ref{modgp}, $\widetilde{\opn{SO}}^+(K) \subset \opn{O}(K) \cap \Gamma$, and so, 
\[
  \vert \la \opn{O}(K), -I \ra : \la \opn{O}(K) \cap \Gamma, -I \ra \vert
  \leq
  \vert \opn{O}(K) : \opn{O}(K) \cap \Gamma \vert
  \leq
  \vert \opn{O}(K) : \widetilde{\opn{SO}}^+(K) \vert.
\]
As
$\vert \opn{O}(L) : \opn{O}^+(L) \vert \leq 2$,
$\vert \widetilde{\opn{O}}^+(L) : \widetilde{\opn{SO}}^+(L) \vert \leq 2$
and
$\vert \opn{O}^+(L) : \widetilde{\opn{O}}^+(L) \vert \leq \vert \opn{O}(q_L) \vert$ 
then
\begin{equation}\label{OKStOKbound}
\vert \opn{O}(K) : \widetilde{\opn{SO}}^+(K) \vert
  \leq
  4 \vert \opn{O}(q_K) \vert.
\end{equation}

Therefore, by Lemma \ref{Kbound}, 
\begin{align*}
E(K_{d,3}^{6,3}(0,l))  & \leq \dfrac{160 d^2}{\pi^2}, 
  & & \text{   } & 
E(K_{d,3}^{6,6}(0,l))  &\leq \dfrac{64 d^2}{675 \pi^2},  \\[20pt]
E(K_{3,d}^{2d,d}(0,l)) &\leq \dfrac{(720)(2^{\nu(d)})}{\pi^2},   
& & \text{   } & 
    E(K_{3,d}^{2d,2d}(0,l)) &\leq \dfrac{(32)(2^{\nu(d)})}{75 \pi^2}. \\
\intertext{If $l$ is odd,}
E(K_{d,3}^{6,6}(0,l)) &\leq \dfrac{2 d^2}{15 \pi^2}, 
&& \text{   } &
E(K_{3,d}^{2d,2d}(3,l)) &\leq \dfrac{(3)(2^{\nu(d)})}{10 \pi^2}, \\
\intertext{and if $l$ is even,}
E(K_{d,3}^{6,6}(0,l)) &\leq \dfrac{d^2}{15 \pi^2}, 
&& \text{   } &
    E(K_{3,d}^{2d,2d}(3,l)) &\leq \dfrac{(3)(2^{\nu(d)})}{20 \pi^2}.
\end{align*}
    Hence, 
\[
\alpha(a) \leq
\left (
\alpha_1 d^2 +
\alpha_2 d^{\epsilon}
\right )
C(a)
\]
where
  $\alpha_1 = (108199)(675 \pi^2)^{-1}$ and
  $\alpha_2 = (216263)(300 \pi^2)^{-1}C_{\epsilon}$.
\end{proof}
To produce the general type results of \S\ref{GTsec} we will need an explicit bound for $2^{\nu(x)}$.
\begin{lemma}\label{2nulemma}
  If $\epsilon>0$ then $2^{\nu(x)} = O(x^{\epsilon})$.
  In particular, $2^{\nu(x)} \leq (75/2)x^{1/6}$ for all $x \in \NN$.
\end{lemma}
\begin{proof}
  Let $m \in \NN$ be such that $1/m<\epsilon$.
  We proceed by induction on the number of primes $n$ in the prime-power decomposition of $x \in \NN$.
  If $p_i$ denotes the $i$-th prime, let $x=p_{i_1}^{a_1} \ldots p_{i_n}^{a_n}$ where $p_{i_j}$ are distinct and $a_i \in \ZZ_{>0}$.
  The condition 
$2^n \leq c x^{1/m}$ 
is equivalent to
  \begin{equation}\label{log2nucond}
    n \leq \frac{\opn{log}(c)}{\opn{log}(2)} + \frac{1}{m \opn{log}(2)} \sum_{j=1}^n a_j \opn{log}(p_{i_j})
  \end{equation}
  and if $N$ is minimal such that $\opn{log}(p_{N+1}) \geq m \opn{log}(2)$ and $c$ is chosen such that \eqref{log2nucond} is satisfied for all $n \leq N$, then \eqref{log2nucond} is satisfied for all $n > N$.
  If
  \[
  S(n):=\sum_{i=1}^n \opn{log}(p_i)
  \]
  then any $c$ satisfying 
  \[
  \left ( n - \frac{\opn{log}(C)}{\opn{log}(2)} \right )m \opn{log}(2) \leq S(n)
  \]
  for all $n \leq N$ suffices.
  The result follows.
\end{proof}
We note that Lemma \ref{2nulemma} improves on a well-known bound of Robin \cite{Robin} when $x$ is small.
\begin{cor}\label{alphacor}
  We have
  \[
  \alpha(a) \leq \left (
  \frac{108199}{675 \pi^2} d^2 + \frac{216263}{8 \pi^2} d^{1/6}
  \right ) C(a).
  \]
\end{cor}
\subsection{Elliptic Obstructions}\label{ellobssec}
We now produce bounds for $\opn{EllObs}_{k(4-a)}(\Gamma, \ec)$ and exhibit a set $\ec$ satisfying the conditions of Proposition \ref{intsingextprop} and \ref{bcextprop}.
We will use $\ec/\Gamma$ to denote a set of representatives for the orbits of $\ec$ for the action of $\Gamma$ and, where no confusion is likely to arise, we will use $K \in \ec / \Gamma$ to denote an embedding $(K \hookrightarrow L) \in \ec/\Gamma$.
\begin{lemma}\label{ellobsprop}
  The obstruction space $\opn{EllObs}_{k(4-a)}(\Gamma, \ec)$ is a subspace of
\[
  \bigoplus_{K \in \ec/\Gamma} \bigoplus_{i=0}^{2k-1} M_{(4-a)k + i}(\Gamma \cap \opn{O}(K)).
  \]
\end{lemma}
\begin{proof}
  As for Proposition 4.1 of \cite{HMvol}.
\end{proof}
\begin{lemma}\label{eclem}
  Suppose $2d$ is square-free.
  Then there exists a set $\ec$ satisfying the conditions of Proposition \ref{intsingextprop} and \ref{bcextprop}.
  The set $\ec$ is given by
  \[
  \ec =
  \ec^{I}  \sqcup
  \ec^{II} \sqcup
  \ec^{III},
  \]
  where $\ec^{I}$, $\ec^{II}$, $\ec^{III}$ are sets of primitive embeddings $\iota_K: K \hookrightarrow L$ satisfying the following conditions.
  \begin{enumerate}
  \item If $\iota_K \in \ec^{I}$ then $ K=2U \op \la -2r \ra$ and $0<r\leq 2 \sqrt{d}$.
  \item If $\iota_K \in \ec^{II}$ then $K=U \op U(3) \op \la -2r \ra$ and $0<r\leq (2 \sqrt{d})/3$.
  \item If $\iota_K \in \ec^{III}$ then  $K=U \op P \op \la -2r \ra$, $0<r\leq (2 \sqrt{d})/3$ and $P$ is an indefinite binary quadratic form with $D(P) = C_9$.
  \end{enumerate}
\end{lemma}
\begin{proof}
  We first show that $\ec$ satisfies the conditions of Proposition \ref{bcextprop}.
  We take the $\ZZ$-basis of $L$ given in \eqref{bcsplitbasis} and take $g \in N(F)_{\ZZ}$ as in Lemma \ref{bcsplitlem}.
  By \eqref{Adeltae} and \eqref{bcsplitbasis},  $(\opn{det} A)^2 \opn{det} B = \opn{det} L = 12 d$ and, as $2d$ is square-free, the pair $(\delta, \delta e)$ defining the matrix $A$ is given by
$(1,1)$, $(1,2)$, $(1,3)$ or $(1,6)$.
  By \eqref{nf} (as noted in \cite[Proposition 2.27]{GHSK3}), $Z \in \Gamma_0(e)$ where $\Gamma_0(e) \subset \opn{SL}(2, \ZZ)$ is a congruence subgroup.
  By considering $\opn{SL}(2, \ZZ)/\Gamma(2)$ and noting that $\Gamma(2)$ contains no torsion other than $\pm I$ (or by the results of \cite{sebbar}), the group $\Gamma_0(2)$ contains no 6-torsion, allowing us to assume that $(\delta, \delta e)=(1,1)$ or $(1,3)$.
  The lattice
$B$ in Lemma \ref{bcsplitlem} is negative definite  of determinant $(12d)(\delta^4 e^2)^{-1}$ and so admits a Gauss decomposition \cite[p.358]{SPLAG} of the form
  \[
  B
  =
  \begin{pmatrix}
    \alpha & \beta \\
    \beta & \gamma
  \end{pmatrix}
  \]
  where $ \alpha < - 2\beta \leq - \alpha \leq - \gamma$ and 
  $\beta \leq 0$ if $\alpha = \gamma$.
Therefore, there exists $t, v \in B$ such that $t \neq 0$, $v$ is primitive, $(t, v)=0$ and $\vert (v,v) \vert \leq 2 \sqrt{(4d)(\delta^4 e^2)^{-1}}$.
  If $(\delta, \delta e) = (1,1)$ then the Gram matrix of $B^{\perp} \subset L$ is given by
  \begin{equation*}
    B^{\perp} =
    \begin{pmatrix}
      0 & I \\
      I & D
    \end{pmatrix}
  \end{equation*}
  and, by the classification of unimodular lattices, $B^{\perp} \cong 2U$ and  $K = 2U \op \la -2r \ra$ where $0<r \leq 2\sqrt{d}$.
  Similarly, if $(\delta, \delta e) = (1,3)$ then $B$ is of signature $(2,2)$ with
  $D(B) = C_3 \op C_3$ or $C_9$.
  If $D(B) = C_3 \op C_3$ then, by the classification of 3-elementary lattices \cite[p.386]{SPLAG}, $B^{\perp} \cong U \op U(3)$. 
  If $D(B) = C_9$ then, by Corollary 1.13.5 of \cite{Nikulin}, $B^{\perp} = U \op P$ where $P$ is an indefinite binary quadratic form of determinant $9$.
  Therefore, $K = U \op U(3) \op \la -2r \ra$ or $K= U \op P \op \la -2r \ra$ with  $0<r \leq (2/3) \sqrt{d}$ and $\ec$ satisfies the conditions of Proposition \ref{bcextprop}.
  The case of Proposition \ref{intsingextprop} follows by a similar argument.
\end{proof}
\begin{defn}
  As for $\alpha$, we define $\beta$ by
\[
  \opn{dim} \bigoplus_{K \in \ec/ \Gamma} \bigoplus_{i=0}^{2k-1} M_{k(4-a) + i}(\Gamma \cap  \opn{O}(K))
  = \beta(a) k^4 + O(k^3).
  \]
We use
  $\beta^I(a)$,
  $\beta^{II}(a)$
  and
  $\beta^{III}(a)$
  to denote the total contribution to $\beta(a)$ from embeddings of
  $K \in \ec^I$,
  $\ec^{II}$ and
  $\ec^{III}$, respectively.
  We will use 
  $\beta_K(a)$
  (or, if no confusion is likely to arise, $\beta_K$)
  to denote the contribution to $\beta(a)$ from a fixed embedding
  $K \in \ec/\Gamma$.
\end{defn}
\begin{proposition}\label{betaIbound}
  For $\epsilon >0$ suppose $A_{\epsilon}$ satisfies $2^{\nu(i)} \leq A_{\epsilon} i^{\epsilon}$ for all $i \in \NN \backslash \{0\}$.
  If $2d$ is square-free, then    
\[
  \beta^I(a) \leq
  \left( \frac{(6144)  8^{\epsilon} A_{\epsilon}^3 D(a)}{5(3 + \epsilon) \pi^2} \right )d^{(3+ 3 \epsilon)/2} \left ( \opn{log}(4 \sqrt{d}) + 1 \right ),
\]
  where  
  \begin{equation}\label{Ddef}
    D(a) :=
    2(4-a)^3 +
    6(4-a)^2 +
    8(4-a) +
    4.
  \end{equation}
\end{proposition}
\begin{proof}
    We estimate
    \[
    \beta^I(a)
    :=
    \sum_{K \in \ec^I/\Gamma} \beta_K(a),
    \]
    starting by bounding the number of embeddings $K \hookrightarrow L$ up to $\Gamma$-equivalence. 
    Let $K \cong 2U \op \la -2r \ra \in \ec^I$.    
    By Proposition 1.15.1 of \cite{Nikulin}, the primitive embeddings $K \hookrightarrow L$ are determined by tuples $(H_K, H_L, \gamma; T, \gamma_T)$ where
    $H_K \subset D(K)$ and
    $H_L \subset D(L)$ are subgroups;
    $\gamma:q_K \vert_{H_K} \xrightarrow{\sim} q_L \vert_{H_L}$
    is an isomorphism of groups preserving the restriction of $q_K$ and $q_L$ to $H_K$ and $H_L$, respectively;
    $T$ is a lattice such that $gen(T) = (0,1,-\delta)$ where
    \begin{equation}\label{deltadef}
      \delta \cong (q_K \op (-q_L) \vert \Gamma_{\gamma}^{\perp})/\Gamma_{\gamma} 
    \end{equation}
    and $\Gamma_{\gamma}$ is the pushout of $\gamma$ in $D(K) \op D(L)$;
    and $\gamma_T:q_T \rightarrow (-\delta)$ is an isomorphism of finite quadratic forms.
    As explained in \cite[p.129-130]{Nikulin}, the embeddings defined by the  tuples $(H_K, H_L, \gamma; T, \gamma_T)$ and $(H_K', H_L', \gamma'; T', \gamma_T')$ are equivalent under $\widetilde{\opn{O}}(L)$ if and only if
    $\gamma = \gamma'$ and $\gamma_T = \pm \gamma_T'$.
Therefore, as 
    $\vert \widetilde{\opn{O}}(L):\widetilde{\opn{SO}}^+(L) \vert = 4$
    and $\widetilde{\opn{SO}}^+(L) \subset \Gamma$, there are at most four  $\Gamma$-equivalent embeddings $K \subset L$ for each tuple $(H_K, H_L, \gamma; T, \gamma_T)$, implying 
    \begin{equation}\label{betai1}
      \beta^I(a)
\leq
       4  \sum_{\substack{K \cong 2U \op \la -2r \ra \\ 0<r \leq 2 \sqrt{d}}}
      \sum_{(H_K, H_L, \gamma; T, \gamma_T)} \beta_K(a). 
    \end{equation}
    (Where we have adopted the convention that when $K$, $H_K$, $H_L$, $\gamma$, $T$ and $\gamma_T$ occur as indices in a summation, they depend on the summand to the left).

    We now count the groups $H_K$ and $H_L$.
    As $D(K)$ is cyclic, the subgroup $H_K$ is uniquely determined by its order $s$ and as
    $H_K \cong H_L \subset D(L)$ then $s \vert 6d$.
As $2d$ is square-free, then
    \[
    D(L) \cong
    \begin{cases}
      C_2^{\op 2} \op C_3^{\op 2} \op C_{d/3} & \text{if $3 \vert d$} \\
      C_2^{\op 2} \op C_3 \op C_d & \text{otherwise,} \\
    \end{cases}
    \]
    implying there are at most $1$, $3$, $4$ or $12$ embeddings of $H_L \cong C_s \subset D(L)$ if  $(s, 6)=1$, $2$, $3$ or $6$, respectively. 
Therefore, by \eqref{betai1},
    \begin{align}
      \beta^I(a)
      & \leq
      48 \sum_{\substack{s \mid 12d \\ s \leq 4 \sqrt{d} \\ H_K \cong C_s}}
      \sum_{\substack{K \cong 2U \op \la -2r \ra \\ 0<r \leq 2\sqrt{d}}}
      \sum_{\gamma, T, \gamma_T} \beta_K, \label{betai2}
    \end{align}
    where the constraint  $s \leq 4 \sqrt{d}$ is obtained by noting that $C_s \cong H_K \subset D(K) \cong C_{2r}$ and $r \leq 2 \sqrt{d}$ (which follows from Lemma \ref{eclem}).

    We now count the isomorphisms $\gamma$.
    As $\gamma:H_K \xrightarrow{\sim} H_L$ preserves the restrictions of $q_K$ and $q_L$, there are at most $\vert \opn{O}(H_L) \vert$ choices for $\gamma$ for a given $(H_K, H_L)$.
    To bound $\vert \opn{O}(H_L) \vert$, suppose $H_L \cong \la (x,y) \ra \subset D(L) \cong C_6 \op C_{2d}$ and let 
\[
    \frac{q_1}{q_2} := \frac{-dx^2 - 3y^2}{6d}
    \]
where $q_1, q_2 \in \ZZ$ are coprime.
    We show that
    \begin{equation}\label{thetas3s}
      \frac{q_1}{q_2} = \frac{\theta_s}{3s}
    \end{equation}
    for some $\theta_s \in \ZZ$.
    As $(s(x,y))^2 \equiv 0 \bmod{ 2 \ZZ}$ then $q_2 \vert s^2$.
    Suppose $p \vert q_2$ for prime $p$. 
    If $p \neq 3$ then $p \vert q_2 \vert s \vert 6d$ and, as $2d$ is square-free, $\vert q_2 \vert_p = p^{-1}$.
    Noting that $q_2 \vert s^2$, we have $\vert s \vert_p \leq p^{-1}$, implying $\vert 3s \vert_p = \vert s \vert_p \leq \vert q_2 \vert_p$. If $p=3$ then, as $2d$ is square-free, $\vert 6d \vert_3 \geq 3^{-2}$  and $\vert q_2 \vert_3 \geq 3^{-2}$ as $q_2 \vert 6d$.
    From $q_2 \vert s^2$ we have $\vert s \vert_3 \leq 3^{-1}$ and $\vert 3 s \vert_3 \leq 3^{-2} \leq \vert q_2 \vert_3$, 
    hence $q_2 \vert 3s$ and \eqref{thetas3s} follows.
    As $H_L$ is cyclic then  $\vert \opn{O}(H_L) \vert$ is equal to the number of  solutions of
    $(z^2 - 1) \theta_s \equiv 0 \bmod{ 6s}$
    where
    $z \in (\ZZ/s\ZZ)^*$.
    We can obtain an upper bound  by counting solutions to 
\[
    (z^2 - 1) \theta_s \equiv 0 \bmod{ s}
    \]
or, by noting that   
    \begin{equation}\label{alcm}
      s = \opn{lcm} \left( \frac{6}{(6,x)}, \frac{2d}{(2d,y)} \right )
    \end{equation}
    is square-free, by counting solutions to
    \begin{equation}\label{HKeqp}
      (z^2 -1) \theta_s \equiv 0 \bmod{ p}
    \end{equation}
    for each prime $p \vert s$.
    We first show that $(p, \theta_s) = 1$ where $p \neq 2, 3$ is prime. 
    As $s \vert 6d$, then 
    \begin{equation*}
      dx^2 + 3y^2 \equiv 3y^2 \bmod{ p}
    \end{equation*}
    where $p \neq 2,3$ and $p \vert s$.
    We must have $(p,y)=1$ otherwise, as $2d$ is square-free and by \eqref{alcm}, we obtain the contradiction $(p,s)=1$.
    As $\theta_s \vert dx^2 + 3y^2$ then $(p,  \theta_s)=1$ and 
\eqref{HKeqp}   has 2 solutions if $p \neq 2,3$;
    at most 2 solutions if $p=2$;
    and at most 3 solutions if $p=3$; 
    implying $\vert \opn{O}(H_L) \vert \leq 3. 2^{\nu(s) - 1}$.
Therefore, from \eqref{betai2},
    \begin{align}
      \beta^I(a)
      & \leq
      72 \sum_{\substack{s \mid 12d \\ s \leq 4 \sqrt{d} \\ H_K \cong C_s}} 2^{\nu(s)}
      \sum_{\substack{K \cong 2U \op \la -2r \ra \\ 0<r \leq 2 \sqrt{d}}}
      \sum_{T, \gamma_T} \beta_K. \label{betai3}
    \end{align}
    As $H_K \subset D(K) \cong C_{2r}$ then $s \vert 2r$, $0<r\leq 2 \sqrt{d}$ and, by \eqref{betai3},
    \begin{align}
      \beta^I(a)
      & \leq
      72 \sum_{\substack{s \mid 12d \\ s \leq 4 \sqrt{d}}} 2^{\nu(s)}
      \left (
      \sum_{\substack{i=1 \\ K \cong 2U \op \la -si \ra \\ 2 \vert si}}^{(4 \sqrt{d})/s}
      \sum_{T, \gamma_T} \beta_K
      \right ). \label{betai4}
    \end{align}

    We now count $T$ and $\gamma_T$.
    The lattice $T$ is of rank 1 and so uniquely determined by the data $(D(K), D(L), \gamma)$ defining $\delta$.
    By \eqref{deltadef}, for fixed $K$, we have $\opn{det} T \vert (12d)(2r)$ and so, as in Lemma \ref{OqkLemma}, there are at most $2^{\nu(6dr)}$ choices for $\gamma_T$ for a given $K$.
    Therefore, by \eqref{betai4},
    \begin{align}
      \beta^I(a)
      & \leq
      72 \sum_{\substack{s \mid 12d \\ s \leq 4 \sqrt{d}}} 2^{\nu(s)}
      \sum_{\substack{i=1 \\ K \cong 2U \op \la -si \ra \\ 2 \vert si }}^{(4 \sqrt{d})/s}
      2^{\nu(6dsi)} \beta_K. \label{betai5}
    \end{align}
    
    We now bound $\beta_K$. 
    As,
    \[
    \opn{Vol}_{HM}(\Gamma \cap \opn{O}(K))
    =
    \vert \la \opn{O}(K), -I \ra : \la \Gamma \cap \opn{O}(K), -I \ra \vert
    \opn{Vol}_{HM}(\opn{O}(K))
    \]
    and $\widetilde{\opn{SO}}^+(K) \subset \Gamma \cap \opn{O}(K)$
    then, by \eqref{OKStOKbound} and Lemma \ref{OqkLemma}, 
\[
    \vert \la \opn{O}(K), -I \ra :  \la \Gamma \cap \opn{O}(K), -I \ra \vert
    \leq
    4 \vert \opn{O}(q_K) \vert 
    \leq
    2^{\nu(r)+2}. 
\]
As $\opn{gen}(K)  = \opn{gen}(K^{2s, 2s}_{r,s}(0,l))$, we bound $\opn{Vol}_{HM}(\opn{O}(K))$ using Lemma \ref{Kbound}, obtaining
  \begin{equation}\label{BKIhm}
    \opn{Vol}_{HM}(\Gamma \cap \opn{O}(K)) \leq \frac{r^2}{60 \pi^2}.
  \end{equation}
  From the identities
  \[
  \begin{array}{c c c c}
    \sum_{k=1}^n k = n(n+1)/2, & \sum_{k=1}^n k^2 = n(n+1)(2n+1)/6 & \text{and}  & \sum_{k=1}^n k^3 = n^2(n+1)^2 / 4, 
  \end{array}
  \]
  we have
  $\beta_K \leq D(a) (16r^2/675 \pi^2)$   
  where
\[
  D(a) := 
    2(4-a)^3 +
    6(4-a)^2 +
    8(4-a) +
    4
    \]
is the $k^4$-coefficient of the polynomial
  \[
  \sum_{j=0}^{2k-1} ((4-a)k+j)^3.
  \]
  Then, from \eqref{betai5}, noting that $s \vert 6 d$ implies $\nu(6dsi) = \nu(6di)$ and by bounding $\beta_K$ as in \eqref{BKIhm}, we have 
\[
  \beta^I(a)
    \leq
    \frac{12}{5 \pi^2} D(a) 2^{\nu(3d)} 
    \sum_{\substack{s \mid 12d \\ s \leq 4 \sqrt{d}}} 2^{\nu(s)}
    \sum_{i=1}^{(4 \sqrt{d})/s}  s^2 i^2 2^{\nu(i)}.  
\]
As $2^{\nu(i)} \leq A_{\epsilon} i^{\epsilon}$ then, 
  \begin{align}
    \sum_{i=1}^{(4\sqrt{d})/s} 2^{\nu(i)} i^2
    & \leq
    \sum_{i=1}^{(4 \sqrt{d})/s} A_{\epsilon} i^{2 + \epsilon} \nonumber \\
& \leq
    \frac{A_{\epsilon}}{3+\epsilon} \frac{(8 \sqrt{d})^{3 + \epsilon}}{s^{3 + \epsilon}}, \label{betai7}
\end{align}
  (with \eqref{betai7} following by comparing with an integral and noting that $s \leq 4 \sqrt{d}$). 
  By \eqref{betai7},
  \begin{align*}
    \sum_{\substack{s \mid 12 d \\ s \leq 4 \sqrt{d}}}
    2^{ \nu(s)} 
    \left (
    \sum_{i=1}^{(4 \sqrt{d})/s} s^2 i^2 2^{\nu(i)}
    \right )
    & \leq
    \left (
    \frac{8^{3 + \epsilon} A_{\epsilon}}{3 + \epsilon}
    \right )
    d^{(3+\epsilon)/2}
    \sum_{\substack{s \leq 4 \sqrt{d} \\ s \mid 12 d }} \frac{2^{ \nu(s)}}{s^{1 + \epsilon}} \\
    & \leq
    \left (
    \frac{8^{3 + \epsilon} A_{\epsilon}^2 }{3 + \epsilon}
    \right )
    d^{(3+\epsilon)/2}
    \sum_{s=1}^{4 \sqrt{d}} s^{-1} \\
\intertext{and as $\sum_{i=1}^N i^{-1} \leq 1 + \opn{log}(N)$ then,}
    & \leq
    \left (
    \frac{8^{3 + \epsilon}A_{\epsilon}^2}{3 + \epsilon}
    \right )
    d^{(3+\epsilon)/2}
    \left ( \opn{log}(4 \sqrt{d}) + 1 \right ).
  \end{align*}
  We therefore obtain
\[ 
    \beta^I(a)
\leq  \left( \frac{(6144)  8^{\epsilon} A_{\epsilon}^3 D(a)}{5(3 + \epsilon) \pi^2} \right )d^{(3+ 3 \epsilon)/2} \left ( \opn{log}(4 \sqrt{d}) + 1 \right ).
    \]
\end{proof}
\begin{proposition}\label{betaIIbound}
  For $\epsilon>0$ suppose  $A_{\epsilon}$ satisfies $2^{\nu(i)} \leq A_{\epsilon} i^{\epsilon}$ for all $i \in \NN \backslash \{0\}$.
  If $2d$ is square-free, then  
  \[
  \beta^{II}(a) \leq
  \frac{3726508032 A_{\epsilon}^3 8^{\epsilon}}{5 (3 + \epsilon)} D(a) d^{3(1+\epsilon)/2} \left ( \opn{log} \left ( \frac{4 \sqrt{d}}{3} \right ) + 1 \right ),
\]
  where $D(a)$ is as in \eqref{Ddef}.
\end{proposition}
\begin{proof}
  Let $K \cong U \op U(3) \op \la -2r \ra \in \ec^{II}$.
  As in Proposition \ref{betaIbound},
  \begin{equation}\label{betaII0}
    \beta^{II}(a) \leq 4 
    \sum_{\substack{K \cong U \op U(3) \op \la -2r \ra \\ 0<r \leq (2 \sqrt{d})/3}}
    \sum_{(H_K, H_L, \gamma; T, \gamma_T)} \beta_K
  \end{equation}
  and we bound the number of tuples $(H_K, H_L, \gamma; T, \gamma_T)$ encoding $\Gamma$-equivalent embeddings of $K \in \ec^{II}$.
  We begin by characterising the groups $H_K$.
  Suppose
\begin{equation}\label{betaIIdisc}
  \begin{array}{ccc}
    D(K) = C_3^{\op 2} \op C_{3^{\mu}} \op C_{2r'} &
    \text{and} &
    D(L) = C_2^{\op 2} \op C_3 \op C_{3^{\delta}} \op C_{d'}
  \end{array}
\end{equation}
  where $3^{\mu} \Vert 2r$, $3^{\delta} \Vert 2d$ and $r':= 3^{-\mu} r$, $d':=3^{-\delta} d$.
As explained in \cite{PetrilloSubgroups}, if $G_1$ and $G_2$ are finite groups of coprime order then, by a theorem of Suzuki \cite{Suzuki} (see also \S4 of \cite{Schmidt}), subgroups of $G_1 \times G_2$ are of the form $G_1' \times G_2'$ where $G_1' \subset G_1$ and $G_2' \subset G_2$.
Therefore, 
  \begin{equation}\label{HKdec0}
    H_K = H_K^3 \op C_s
  \end{equation}
  where $H_K^3 \subset C_3^{\op 2} \op C_{3^{\mu}}$ and $C_s \subset C_{2r'}$ for some $s$.
  As $H_L \cong H_K$ then, by \eqref{betaIIdisc}, $H_K^3 \subset C_3 \op C_{3^{\delta}}$ and, as $2d$ is square-free, $\delta=0$ or $1$ and $H_K^3 \cong \{e \}$, $C_3$ or $C_3^{\op 2}$.
  
  We now bound the number of embeddings $H_K \subset D(K)$, where $H_K$ is as in \eqref{HKdec0} and $C_s$ is fixed.
  There is a single embedding $H_K \subset D(K)$ if $H_K^3 = \{e\}$; 
  at most 13 embeddings if $H_K^3=C_3$; 
  and at most 13 embeddings if $H_K^3 = C_3^{\op 2}$.
Therefore (by using an identical argument for $H_L \subset D(L)$ and noting there are 3 elements of order $2$ in $C_2^{\op 2}$), for fixed $s$ there are at most $3$ pairs $(H_K, H_L)$ if $H_K^3 = \{e \}$;
  $156$ pairs if $H_K^3 = C_3$;
  and $39$ pairs if $H_K^3 = C_3^{\op 2}$.

We now count the maps $\gamma$.
  For fixed $(H_K, H_L)$, there are at most $\vert \opn{O}(H_L) \vert$ choices for $\gamma: H_K \xrightarrow{\sim} H_L$. 
  As in Proposition \ref{betaIbound},
  $\vert \opn{O}(H_L) \vert \leq (3)2^{\nu(s)-1}$ if $H_K^3=\{e \}$;
  $\vert \opn{O}(H_L) \vert \leq (3)2^{\nu(s)}$ if $H_K^3 = C_3$;
  and $\vert \opn{O}(H_L) \vert \leq (9)2^{\nu(s) + 3}$ if $H_K^3 = C_3^{\op 2}$.
Hence, by \eqref{betaII0}, 
  \[
    \beta^{II}(a)
\leq
      13122
      \sum_{\substack{K=U \op U(3) \op \la -2r \ra \\ 0 < r \leq (2\sqrt{d})/3}}
      \sum_{H_K=H_K^3 \op C_s} 2^{\nu(s)}
      \sum_{(T, \gamma_T)} \beta_K. 
  \]
  As in Proposition \ref{betaIbound}, the lattice $T$ is uniquely determined by $(H_K, H_L, \gamma)$, implying there are at most $2^{\nu(6dr)}$ choices for $\gamma_T$ for a given $T$.
  If $H_K \cong H_K^3 \op C_s$ then $s \vert 2r$ and as  $H_K \cong H_L$ then $s \vert 2d$.
  As $r \leq (2 \sqrt{d})/3$, then   
  \begin{align*}
    \beta^{II}(a) 
    \leq
    13122
    \sum_{\substack{s \vert 2d \\ 0 < s \leq (4\sqrt{d})/3}}2^{\nu(s)}
    \sum_{\substack{i=1 \\ K = U \op U(3) \op \la -si \ra \\ 2 \vert si}}^{\frac{4 \sqrt{d}}{3s}} 2^{\nu(6dsi)} \beta_K.
  \end{align*}

  We now bound $\beta_K$.
  By \eqref{OKStOKbound} and Lemma \ref{Oqk2-3-lem},
  \[
    \vert \la \opn{O}(K), -I \ra : \la \Gamma \cap \opn{O}(K), -I \ra \vert
     \leq 4 \vert \opn{O}(q_K) \vert \\
     \leq (2496)2^{\nu(2r)}
  \]
  and by Lemma \ref{II-III-HMVol},
\begin{equation*}
    \beta_K(a) \leq \frac{7488}{5} D(a) r^2.
  \end{equation*}
  Therefore,  
  \begin{align*}
    \beta^{II}(a)
    \leq & \frac{98257536}{5} 2^{\nu(3d)} D(a) \left( \sum_{\substack{s \vert 2d \\ 0<s \leq (4 \sqrt{d})/3 }}  2^{\nu(s)} s^2  
    \sum_{i=1}^{(4 \sqrt{d})/3s} 2^{\nu(i)} i^2 \right ) 
  \end{align*}
    where $D(a)$ is as in \eqref{Ddef}.
  As in \eqref{betai7}, \begin{align*}\sum_{i=1}^{(4 \sqrt{d})/3s} 2^{\nu(i)} i^2
& \leq \frac{}{} \frac{A_{\epsilon}(8 \sqrt{d})^{3 + \epsilon}}{(3 + \epsilon) (3s)^{3 + \epsilon}}
  \end{align*}
and so 
\[
  \beta^{II}(a) \leq
  \frac{3726508032 A_{\epsilon}^3 24^{\epsilon}}{5 (3 + \epsilon)} D(a) d^{3(1+\epsilon)/2} \left ( \opn{log} \left ( \frac{4 \sqrt{d}}{3} \right ) + 1 \right ).
  \]
\end{proof}
\begin{proposition}\label{betaIIIbound}
  For $\epsilon>0$ suppose $A_{\epsilon}$ satisfies $2^{\nu(i)} \leq A_{\epsilon} i^{\epsilon}$ for all $i \in \NN \backslash \{0\}$.
  If $2d$ is square-free, then 
  \begin{equation*}
    \beta^{III} (a) \leq
    \frac{47775744 A_{\epsilon}^3 8^{\epsilon} D(a)}{5(3 + \epsilon)} d^{3(1+\epsilon)/2}\left ( \opn{log} \left ( \frac{4 \sqrt{d}}{3} \right ) + 1 \right ),
  \end{equation*}
  where $D(a)$ is as in \eqref{Ddef}.
\end{proposition}
\begin{proof}
  Our argument proceeds as in Proposition \ref{betaIIbound}. 
  Let $K \cong U \op P \op \la -2r \ra \in \ec^{III}$ where $P$ is an indefinite binary quadratic form of determinant 9, 
then 
  \begin{equation*}
    \beta^{III}(a) \leq
    4 \sum_{\substack{K \cong U \op P \op \la -2r \ra \\ 0<r \leq (2 \sqrt{d})/3}}
    \sum_{(H_K, H_L, \gamma; T, \gamma_T)} \beta_K,
  \end{equation*}
  and we begin by bounding the number of tuples $(H_K, H_L, \gamma; T, \gamma_T)$ for each $K$.
We have 
  \[
  H_K^3 \op C_s \cong H_K  \subset D(K)
\] 
  where  $H_K^3 = \{e \}$, $C_3$ or $C_3^{\op 2}$ and $s \vert 2r$.
  For fixed $C_s$  there are
  at most 3 pairs $(H_K, H_L)$ if $H_K^3 = \{ e \}$ or $C_3^{\op 2}$ and at most 48 pairs if $H_K^3 = C_3$.
  The number of $\gamma: H_L \xrightarrow{\sim} H_K$ preserving the restriction of $q_L$ and $q_K$ is bounded as in Proposition \ref{betaIbound}.
Hence, as in Proposition \ref{betaIIbound}, 
  \begin{align*}
    \beta^{III}(a)
    \leq & \frac{1259712}{5} D(a) \left( \sum_{\substack{s \vert 2d \\ 0<s \leq (4 \sqrt{d})/3 }}  2^{\nu(s)} s^2  
    \sum_{i=1}^{(4 \sqrt{d})/3s} 2^{\nu(6dsi)} i^2 \right ) \\
    \leq & \frac{47775744 A_{\epsilon}^3 8^{\epsilon} D(a)}{5(3 + \epsilon)} d^{3(1+\epsilon)/2}\left ( \opn{log} \left ( \frac{4 \sqrt{d}}{3} \right )+ 1 \right ) 
  \end{align*}
  where $D(a)$ is as in \eqref{Ddef}.
\end{proof}
 \section{General Type Results}\label{GTsec}
\begin{proposition}\label{mfbounds1}
    Suppose $2d$ is square-free and take $\ec$ as in Lemma \ref{eclem}.
    If  $a=2$ and $d>2.04134 \times 10^{28}$, 
    or $a=3$ and $d>2.5112 \times 10^{29}$ then 
    \begin{equation}\label{formbound}
      \opn{dim} M_{k(4-a)}(\Gamma) 
      - \opn{dim} \opn{RefObs}_{k(4-a)}(\Gamma)
        - \opn{dim} \opn{EllObs}_{k(4-a)}(\Gamma, \ec)
        > 0 
    \end{equation}
    for $k \gg 0$.
  \end{proposition}
  \begin{proof}
    By Hirzebruch-Mumford proportionality, \eqref{formbound} is satisfied for $k \gg 0$ if
    \begin{equation}\label{formbound0}
      \gamma (4-a)^4
      - \alpha(a)
      - \beta^I(a)
      - \beta^{II}(a)
      - \beta^{III}(a) > 0.
    \end{equation}
    As  $\opn{log}(d) \leq 6 d^{1/12}$ for $d>10^{10}$ and by using the bounds of Corollary \ref{alphacor}, Proposition  \ref{betaIbound}, \ref{betaIIbound} and \ref{betaIIIbound}, we obtain a polynomial in $d^{1/12}$ bounding \eqref{formbound0} from below.
    By Descartes' rule of signs \cite{Descartes}, this polynomial has at most one root, which we locate using a computer.
\end{proof}
  Because of the range of values taken by $\nu(d)/\opn{log}(d)$, the bounds obtained in Proposition \ref{mfbounds1} are quite large.
  However, as a consequence of the Erd\H{o}s-Kac theorem \cite{ErdosKac}, $\nu(d)/\opn{log}(d)$ is typically much smaller than the worst case assumed in Proposition \ref{mfbounds1}, allowing smaller bounds to be obtained by constraining $\nu(d)$.
\begin{proposition}\label{mfbounds2}
    Suppose $d$ is prime and take $\ec=\ec^I$ as in Lemma \ref{eclem}.
    If
    $a=2$ and $d>409679370162$, or 
    $a=3$ and $d>3468755940829$
    then
\[ 
    \opn{dim} M_{k(4-a)}(\Gamma) 
      - \opn{dim} \opn{RefObs}_{k(4-a)}(\Gamma)
        - \opn{dim} \opn{EllObs}_{k(4-a)}(\Gamma, \ec)
        > 0
        \]
for $k \gg 0$.
\end{proposition}
\begin{proof}
  We proceed as in Proposition \ref{mfbounds1}, noting that we can assume $\ec=\ec^I$ and using the improved bound
  \[
  \beta^I(a)
  \leq
  \frac{16128(8^{3+\epsilon})A_{\epsilon}D(a)d^{(3+\epsilon)/2}}{85 \pi^2 (3 + \epsilon)},
  \]
  which is obtained by setting $\nu(d)=1$ in \eqref{betai5}.
\end{proof}
  \begin{thm}\label{gtthm}
    The orthogonal modular varieties $\fc_L(\Gamma)$ are of general type for all odd square-free $d>2.5112 \times 10^{29}$ and for all prime $d>3468755940829$\footnote{It was questioned in \cite{ConwaySloaneMass} if the local density $\alpha_2$ \cite{Watson} (stated in \cite[Theorem 5.6.3]{Kitaoka}) \emph{might} differ from its actual value by a factor of $2^n$ for a rank $n$ lattice.
      We do not know if this has been resolved.
      The cautious reader may therefore wish to multiply the bounds of Proposition \ref{mfbounds1}, \ref{mfbounds2} and Theorem \ref{gtthm} by a factor of $4^{1/4}$.
}.
  \end{thm}
  \begin{proof}
    Immediate from Propositions \ref{mfbounds1}, \ref{mfbounds2}, Corollary  \ref{weight3cuspformcor} and Theorem \ref{lwcft}.
  \end{proof}
    \begin{cor}
    Under the conditions of Theorem \ref{gtthm}, the moduli $\mc$ of deformation generalised Kummer varieties of dimension 4 with polarisation of degree $2d$ and split type is of general type if the orthogonal modular variety $\fc(\Gamma)$ is.
  \end{cor}
  \begin{proof}
    By \cite[Theorem 3.10]{Handbook}, there is an open immersion from $\mc \rightarrow \fc(\Gamma)$.
\end{proof}
 \section{Acknowledgements}
I thank Professor Gregory Sankaran for piquing my interest in general type results back when I was a PhD student and for all manner of interesting conversations over the years.
I also thank the anonymous referee for suggesting the approach of \S\ref{lwcfsec}, which greatly improved upon my earlier attempts.
Part of the paper was completed when I worked at the University of Bonn and I acknowledge the grant SFB/TR 45 \emph{`Periods, Moduli Spaces and Arithmetic of Algebraic Varieties'} of the DFG (German Research Foundation) under which I was partially supported.
I am especially grateful to Professor Huybrechts for providing  excellent working conditions and a very stimulating environment. 

 \bibliographystyle{amsalpha}
\bibliography{bib}{}
\noindent \texttt{matthew.r.dawes@bath.edu}
\end{document}